\documentclass[twoside,11pt,abbrvbib]{article}

\usepackage{blindtext}

\usepackage{amsthm}
\usepackage[preprint]{jmlr2e}

\usepackage[dvipsnames]{xcolor}
\usepackage{amsmath,bm}
\usepackage{array}
\usepackage{dsfont}
\usepackage{url}	
\usepackage{ulem}	
\usepackage{xspace}	
\usepackage{booktabs}
\usepackage{rotating}	
\usepackage{pdflscape}  
\usepackage{fancyvrb}   
\usepackage{epstopdf}
\usepackage{bbm}	
\usepackage{amsfonts}
\usepackage{latexsym}
\usepackage{caption}
\usepackage{subcaption}

\graphicspath{{figures/}}

\newcommand{\mtc}{\mathcal}
\newcommand{\mbf}{\bm}

\let\originalleft\left
\let\originalright\right
\renewcommand{\left}{\mathopen{}\mathclose\bgroup\originalleft}
\renewcommand{\right}{\aftergroup\egroup\originalright}

\usepackage[linecolor=blue!60!,backgroundcolor=blue!10!,textwidth=2.5cm,textsize=scriptsize]{todonotes}

\RequirePackage{xstring}

\newcommand{\paren}[2][a]{%
\IfEqCase{#1}{%
{a}{\left(#2\right)}%
{0}{(#2)}%
{1}{\big(#2\big)}%
{2}{\Big(#2\Big)}%
{3}{\bigg(#2\bigg)}%
{4}{\Bigg(#2\Bigg)}%
}[\PackageError{paren}{Undefined option to paren: #1}{}]%
}
\newcommand{\norm}[2][a]{%
\IfEqCase{#1}{%
{a}{\left\lVert#2\right\rVert}%
{0}{\lVert#2\rVert}%
{1}{\big\lVert#2\big\rVert}%
{2}{\Big\lVert#2\Big\rVert}%
{3}{\bigg\lVert#2\bigg\rVert}%
{4}{\Bigg\lVert#2\Bigg\rVert}%
}[\PackageError{norm}{Undefined option to norm: #1}{}]%
}
\newcommand{\brac}[2][a]{%
\IfEqCase{#1}{%
{a}{\left[#2\right]}%
{0}{[#2]}%
{1}{\big[#2\big]}%
{2}{\Big[#2\Big]}%
{3}{\bigg[#2\bigg]}%
{4}{\Bigg[#2\Bigg]}%
}[\PackageError{brac}{Undefined option to brac: #1}{}]%
}
\newcommand{\inner}[2][a]{%
\IfEqCase{#1}{%
{a}{\left\langle#2\right\rangle}%
{0}{\langle#2\rangle}%
{1}{\big\langle#2\big\rangle}%
{2}{\Big\langle#2\Big\rangle}%
{3}{\bigg\langle#2\bigg\rangle}%
{4}{\Bigg\langle#2\Bigg\rangle}%
}[\PackageError{inner}{Undefined option to inner: #1}{}]%
}
\newcommand{\abs}[2][a]{
\IfEqCase{#1}{%
{a}{\left\vert#2\right\rvert}%
{0}{\vert#2\rvert}%
{1}{\big\vert#2\big\rvert}%
{2}{\Big\vert#2\Big\rvert}%
{3}{\bigg\vert#2\bigg\rvert}%
{4}{\Bigg\vert#2\Bigg\rvert}%
}[\PackageError{abs}{Undefined option to abs: #1}{}]%
}

\newcommand{\set}[2][a]{
\IfEqCase{#1}{%
{a}{\left\{#2\right\}}%
{0}{\{#2\}}%
{1}{\big\{#2\big\}}%
{2}{\Big\{#2\Big\}}%
{3}{\bigg\{#2\bigg\}}%
{4}{\Bigg\{#2\Bigg\}}%
}[\PackageError{set}{Undefined option to set: #1}{}]%
}

\newcommand{\ind}[2][a]{{\mbf{1}\set[#1]{#2}}}

\newcommand{\var}[2][a]{{\rm Var}\brac[#1]{#2}}
 
\newcommand{\e}[2][a]{\mbe\brac[#1]{#2}}

\newcommand{\prob}[2][a]{\mbp\brac[#1]{#2}}

\newcommand{\eps}{\varepsilon}

\def\cF{{\mtc{F}}}

\def\cO{{\mtc{O}}}

\newcommand{\mbe}{\mathbb{E}}

\newcommand{\mbp}{\mathbb{P}}

\newcommand{\wt}[1]{{\widetilde{#1}}}
\newcommand{\wh}[1]{{\widehat{#1}}}
\newcommand{\ol}[1]{\overline{#1}}

\renewcommand{\emph}[1]{{\it #1}}

\newcommand{\E}{\mathbf{E}}
\newcommand{\R}{\mathbb{R}}
\renewcommand{\l}{\ell}
\renewcommand{\P}{\mathbf{P}}
\newcommand{\cH}{\mathcal{H}}

\newcommand{\BH}{\textnormal{BH}}

\newcommand{\FDP}{\textnormal{FDP}}

\ShortHeadings{Consistent FDP envelopes}{Iqraa Meah, Gilles Blanchard and Etienne Roquain}
\firstpageno{1}

\begin{document}

\title{False discovery proportion envelopes with {$m$-consistency}}

\author{ \name Iqraa Meah \email iqraa.meah@inserm.fr \\
	\addr Center for Research in Epidemiology and StatisticS (CRESS)\\
	Universit\'e Paris Cit\'e and Universit\'e Sorbonne Paris Nord, Inserm, INRAE\\
	F-75004 Paris, France
%       \addr Laboratoire de Probabilit\'es, Statitique et Mod\'elisation (LPSM)\\
%       Sorbonne Universit\'e\\
%       4 place Jussieu, 75005, Paris
       \AND
       \name Gilles Blanchard \email gilles.blanchard@universite-paris-saclay.fr \\
       \addr Institut Math\'ematiques de Orsay (IMO)\\
	CNRS, Inria, Universit\'e Paris-Saclay\\
	F-91405 Orsay Cedex
        \AND
       \name Etienne Roquain \email etienne.roquain@sorbonne-universite.fr \\
       \addr 
Laboratoire de Probabilit\'es, Statistique et Mod\'elisation, \\CNRS, Sorbonne Universit\'e, Universit\'e de Paris.\\
       Sorbonne Universit\'e\\
       4 place Jussieu, 75005, Paris}

%\editor{Gabor Lugosi}

\maketitle

\begin{abstract}%   <- trailing '%' for backward compatibility of .sty file
	We provide new {nonasymptotic} false discovery proportion (FDP) confidence envelopes in several multiple testing settings relevant for modern high dimensional-data
  	methods. We revisit the multiple testing scenarios considered in the recent work of \cite{katsevich2020simultaneous}:
  	top-$k$, preordered (including knockoffs), online. {Our emphasis is on obtaining FDP {confidence} bounds that both have nonasymptotical coverage and are asymptotically accurate in a specific sense, as the number $m$ of tested hypotheses grows. Namely, we introduce
          and study the property (which we call $m$-consistency) %, which requires
          that the confidence bound converges to or below the desired level $\alpha$ when applied to a {specific} reference $\alpha$-level false discovery rate (FDR) controlling procedure.}
% 	 converge, \gil{as the number $m$ of tested hypotheses grows, to or below} the desired level $\alpha$ when applied to a \gil{specific} reference $\alpha$-level false discovery rate (FDR) controlling procedure \gil{ --- a property we call
%   $m$-consistency with respect to the reference procedure.}
 	{In this perspective, we derive new bounds that provide improvements over existing ones, both theoretically and practically, and are suitable for situations where at least a moderate number of rejections is expected}. These improvements are illustrated with numerical experiments and real data examples. In particular, the improvement is significant in the knockoffs setting, which shows the impact of the method for a practical use. 
	As side results, we introduce a new confidence envelope for the empirical cumulative distribution function of i.i.d. uniform variables and we provide new power results in sparse cases, both being of independent interest.	
\end{abstract}


\section{Introduction}

\subsection{Background}

Multiple inference is a crucial issue in many modern, high dimensional, and massive data sets, for which a large number of variables are considered and many questions naturally emerge, either simultaneously or sequentially. 
Recent statistical inference has thus turned to designing methods that guard against false discoveries and selection effect, see \cite{cui2021handbook,robertson2023online} for recent reviews on that topic. A key quantity is typically  the false discovery proportion (FDP), that is, the proportion of false discoveries within the selection \citep{BenjaminiHochberg95}. 

Among classical methods, finding confidence bounds on the FDP that are valid after a user data-driven selection (`post hoc' FDP bounds), has retained attention since the seminal works of \cite{genovese2004stochastic,genovese2006exceedance,goeman2011multiple}. 
The strategy followed by these works is to build confidence bounds {\it valid uniformly over all selection subsets}, which {\it de facto} provides a bound valid for any data-driven selection subset.
A number of such FDP bounds have been proposed since, either based on 
a `closed testing' paradigm \citep{hemerik2019permutation,goeman2019simultaneous,GHS2021,vesely2023permutation},   a `reference family' \citep{blanchard2020post,durand2020post}, or a specific prior distribution in a Bayesian framework \citep{perrot2021post}. 
It should also be noted that methods providing bounds valid uniformly over {\it some particular} selection subsets  can also be used to provide bounds
valid on {\it any} subsets by using 
%all provide {\it in fine} bounds valid uniformly over all selection subsets, by applying 
an `interpolation' technique, see, e.g., \cite{blanchard2020post}.
This is the case for instance for bounds based upon an empirical distribution function confidence band, as investigated by \cite{meinshausen2005lower,meinshausen2006false,MeinRice2006,dumbgen2023new}. 
Loosely, we will refer to such (potentially partial) FDP bounds as {\it FDP confidence envelopes} in the sequel.

Recently, finding FDP confidence envelopes has been extended to different contexts of interest in \cite{katsevich2020simultaneous} (KR below for short), including knockoffs \citep{barber2015controlling,CFJL2018} and online multiple testing \citep{AharoniRossetGAI}. For this, their bounds are tailored on particular nested `paths', and employ accurate martingale techniques. In addition, \cite{li2024simultaneous} have recently investigated  specifically the case of the knockoffs setting  by using a `joint' $k$-FWER error rate control (see also \citealp{genovese2006exceedance,meinshausen2006false,blanchard2020post}), possibly in combination with closed testing.

\subsection{New insight: {$m$-consistency}}

{The main thrust of this paper is to look at FDP confidence envelopes from the angle of a particular property that we call {\it {$m$}-consistency} {($m$ denoting the number of hypotheses under consideration)}. First, recall that the false discovery rate (FDR) is the expectation of the FDP, which is a type I error rate measure with increasing popularity since the seminal work of \cite{BenjaminiHochberg95}. Informally, an FDP confidence envelope is {$m$}-consistent, 
{in relation to} a {given} {reference} FDR-controlling selection {procedure, if the envelope value for the set output by that procedure} 
{converges} to (or below) the corresponding nominal {FDR} value, at least asymptotically {when the total number $m$ of hypotheses tends
to infinity.}}
This property is important for several reasons:
\begin{itemize}
\item FDR controlling procedures output particular selection sets that are widely used in practice. Hence, it is very useful to provide an accurate FDP confidence bound for these particular rejection sets. This is the case for instance for the commonly used Benjamini-Hochberg (BH) procedure at a level $\alpha$ % --- or even for a data dependent choice of the level $\hat{\alpha}$ \todo{to keep $\hat{\alpha}$?}---
 for which the FDP bound should be close to $\alpha$, %(or $\hat{\alpha}$),
  at least in `favorable' cases;
\item a zoo of FDP confidence envelopes have been proposed in previous literature, and we see $m$-consistency as a principled way to discard some of them while putting the emphasis on others; 
\item searching for $m$-consistency can also lead to new bounds that are accurate for a moderate sample size.
\end{itemize}

{It turns out that most of the existing {FDP} bounds, while being accurate in certain regimes, are not {$m$}-consistent {with respect to classical FDR controlling procedures}. In  particular, this is the case for {the Simes bound \citep{Sim1986} and} those of \cite{katsevich2020simultaneous} {with respect to the BH procedure,} because of a {an incompressible} constant factor (larger than 1) in front of the FDP estimate. The present paper proposes to fill this gap by introducing new envelopes that are {$m$}-consistent. In a nutshell, we replace  the constant in front of the FDP estimate by a function that tends to $1$ in an  %particular 
asymptotical regime {under broad conditions including sparse asymptotical settings (i.e., the proportion of false null hypotheses tends to 0)}. 
{We stress that this notion of $m$-consistency only concerns a vanishing gap
between the FDP bound and the FDR nominal level of a reference procedure. It does
{\it not} require that the individual tests are consistent in the usual sense, i.e. we do
not assume that they reject a false null hypothesis with probability tending to 1;
nor does it require that the reference procedure has {full asymptotic power.}}}
%As a result, in addition to obtain an FDR controlling procedure, we also have a confidence upper-bound on the FDP of that procedure, and the latter is guaranteed to be small enough to be close to  the nominal level (in a certain sense).

%Since we evoke consistency, it is worth emphasizing that 
{Let us emphasize again that the envelopes developed in this work have coverage holding in a {\it non-asymptotical} sense. Here, {$m$-}consistency means that on top of this strong non-asymptotical guarantee, the {FDP} bound satisfies an additional sharpness condition in an asymptotical sense and for some scenarios of interest, including sparse ones.}

\subsection{Settings}

Following \cite{katsevich2020simultaneous}, we consider the three following multiple testing settings for which a `path' means a (possibly random) nested sequence of candidate rejection sets:
\begin{itemize}
\item Top-$k$:  the classical multiple testing setting where the user tests a finite number $m$ of null hypotheses and observes simultaneously a family of corresponding $p$-values.
 %% $m$-null hypotheses. %%% ?????
  This is the framework of the seminal paper of \cite{BenjaminiHochberg95} and of the majority of the follow-up papers. In that case, the path  is composed of the hypotheses corresponding to the top-$k$ most significant $p$-values (i.e. ranked
  in increasing order), for varying $k$.
\item Pre-ordered: we observe $p$-values for a finite set of cardinal $m$ of null hypotheses, which are {\it a priori} arranged according to some ordering. In that setting, the signal (if any) is primarily carried by the ordering:
  alternatives are expected to be more likely to have a small rank.
  Correspondingly the path in that case is obtained by $p$-value thresholding (for fixed threshold) of the first $k$ hypotheses
  w.r.t. that order, for varying $k$. A typical instance is the knockoffs setting \citep{barber2015controlling,CFJL2018}, where the null hypotheses come from a high-dimensional linear regression model and one wants to test whether each of the $m$ variables is associated with the response. The ordering is data-dependent and comes from an ancillary statistic independent of the tests themselves, so that one can argue conditionally and consider the ordering (and path) as fixed.
%The observations are made once the (model-$X$) knockoff strategy of \cite{barber2015controlling} is used, which gives ordered binary $p$-values.
\item Online: the null hypotheses come sequentially, and there is a corresponding %continuous and
  potentially infinite stream of $p$-values. An irrevocable decision
  (reject or not) has to be taken in turn for each new hypothesis,  depending on past observations only. The path is naturally defined according to the set of rejections until time $t$,
  for varying $t$. 
  {It is markedly different from the pre-ordered setting because decisions are irrevocable and the set of nulls is not a pre-specified finite set.} %\todo{[GB] Sample $\rightarrow$ set?}
  %decision times.
\end{itemize}
Let us introduce notation that encompasses the three settings mentioned above: the set of hypotheses is denoted by $\cH$ (potentially infinite), the set of null hypotheses $\cH_0$ is an unknown subset of $\cH$, and a path $\Pi=(R_k, k\ge 1)$ (with convention $R_0=\emptyset$) is an ordered sequence of nested subsets of $\cH$ that depends only on the observations. A confidence envelope is a sequence $(\ol{\FDP}_k, k\ge 1)$ (with convention $\ol{\FDP}_0=0$) of random variables valued in $[0,1]$, depending only on the observations, such that, for some pre-specified level $\delta$, we have
\begin{equation}\label{confenvelop}
	\P\paren{\forall k\geq 1, \FDP(R_k)\le \ol{\FDP}_k}\ge 1-\delta,
\end{equation}
where $\FDP(R_k)=\frac{|R_k\cap \cH_0|}{|R_k|\vee 1}$ is the FDP of the set $R_k$.
In \eqref{confenvelop}, the guarantee is non-asymptotic and uniform in $k$, which means that it corresponds to confidence bounds valid  uniformly over the subsets of the path. %, providing more informations than just being a confidence bound on a particular subset.
Also, the distribution $\P$ is relative to the  $p$-value model, which will be specified further on and depends on the considered framework.

\begin{remark}[Interpolation]\label{rem:interp}
  %On can notice here that
  Any FDP confidence envelope of the type \eqref{confenvelop} also leads to a post hoc FDP bound valid uniformly for all $R\subset \cH$: specifically, by using the interpolation method (see, e.g., \citealp{blanchard2020post,GHS2021,li2024simultaneous}), if \eqref{confenvelop} holds then the same property also holds with the sharper bound $(\wt{\FDP}_k, k\geq 1)$ given by 
\begin{align}
\wt{\FDP}_k& = 
\frac{\min_{k'\leq k} \set{%|R_k\cap (R_{k'})^c|
|R_k|- |R_{k'}|
+|R_{k'}| \ol{\FDP}_{k'} }}{|R_k|\vee 1} \leq \ol{\FDP}_k
\label{equ-interp3},
\end{align}
due to the fact that the number of false positives in $R_k$ is always bounded by the number of false positives in $R_{k'}\subset R_k$ plus the number of elements of  $R_k\setminus R_{k'}$.
%$|R_k\cap \cH_0|\leq |R_k\cap R_{k'}^c| + |R_{k'}\cap \cH_0| \leq |R_k\cap R_{k'}^c| + |R_{k'}|\ol{\FDP}_{k'} = |R_k|-|R_k\cap R_{k'}| + |R_{k'}|\ol{\FDP}_{k'}$
\end{remark}

Particular subsets of $\Pi=(R_k, k\geq 1)$ that are of interest are those controlling the FDR. Given a nominal level $\alpha$, a `reference' procedure chooses a data-dependent $\hat{k}_\alpha$ such that $\e[1]{\FDP(R_{\hat{k}_\alpha})} \leq \alpha$. Depending on the setting, we consider different reference procedures:
%the considered reference procedures take different forms:
\begin{itemize}
\item Top-$k$ setting: the reference FDR controlling procedure is the Benjamini-Hochberg (BH) step-up procedure, see \cite{BenjaminiHochberg95};
\item Pre-ordered setting: the reference procedure is the Lei-Fithian (LF) adaptive Selective sequential step-up procedure, see \cite{lei2016power} (itself being a generalization of the procedure of \citealp{li2017accumulation});
 \item Online setting: the reference procedure is the (LORD) procedure, see \cite{javanmard2018online} and more precisely the improved version of \cite{ramdas2017online}.
\end{itemize}
As announced, for all these procedures, the {\it expectation} of the FDP (that is, the FDR) is guaranteed to be below $\alpha$. 
%{However, this does not guarantee that the corresponding FDP is below $\alpha$ with high probability \citep{genovese2004stochastic,neuvial08asymptotic,neuvial13asymptotic,Korn2004,DR2015}. To fill the gap, the latter FDP control can nevertheless be recovered from \eqref{confenvelop} if the FDP bound is close to (or below) $\alpha$. }
%\et{On the other hand, it is commonly the case that
{Additionally, in many scenarios considered as prototypical in the literature, it has been established that}
{the FDP {of a reference procedure} concentrates around its expectation {as the number of tested
    hypotheses $m$ tends to infinity}, see, e.g., \cite{genovese2004stochastic,neuvial08asymptotic,neuvial13asymptotic}.}
{In such a situation, it is thus natural to expect that a confidence envelope on the FDP should asymptotically converge to (or below) the target level $\alpha$ when applied to a reference procedure.}
  %This motivates introducing the $m$-consistency notion in complement to the non-asymptotic control \eqref{confenvelop}.} 

  %\todo{Should we cite here a few earlier references about the asymptotic properties of the FDP: LLN, CLT (e.g. Neuvial...)?}
%Hence, the true FDP should also be not to above $\alpha$, at least in a framework where this variable concentrates around its expectation. An appropriate bound on the FDP should thus also be not too above $\alpha$. This is what the consistency property means.

%\et{Gives the procedures + control provided here?}
%Let us now make the definition of $m$-consistency
{This motivates, in complement to the non-asymptotic control \eqref{confenvelop},
the introduction of a notion of $m$-consistency of an FDP envelope in relation to a reference procedure
and a sequence of models, as follows.} % more precise.

%KR proposed such confidence envelopes, very elegant, but they loose the fact that $\ol{\FDP}_k$ should be close to $\alpha$ when choosing an element of the path that is supposed to control the FDR at level $\alpha$. This property is called consistency here.
\begin{definition}[$m$-consistency]\label{def-consistency}
  {Let $\delta,\alpha\in(0,1)$ be fixed.
  For each $m\geq 1$, let be
  \begin{itemize}
\item  $\P^{(m)}$  a multiple testing distribution model over the  set of null hypotheses $\cH^{(m)}=\{1,\dots,m\}$ and a set of true null hypotheses $\cH_0^{(m)}\subset \cH^{(m)}$ (denote $\cH_1^{(m)} = \cH^{(m)} \setminus \cH_0^{(m)}$);
%  as the set of null hypotheses, a path
\item  $\Pi^{(m)}=(R_k^{(m)}, k\geq 1)$ a possibly random path of nested subsets of $\cH^{(m)}$;
\item $(\ol{\FDP}^{(m)}_k, k\geq 1)$ a confidence envelope at level $1-\delta$ over that path, i.e. satisfying~\eqref{confenvelop};
 \item ${\hat{k}^{(m)}_\alpha}$ a procedure choosing a rejection set from the path $\Pi^{(m)}$ with guaranteed FDR control at level $\alpha$, i.e. satisfying
  $\E^{(m)}\brac{\FDP\paren[1]{R^{(m)}_{\hat{k}^{(m)}_\alpha}}} \leq \alpha$.
\end{itemize}
  %$\P=\P^{(m)}$). 
  %For any $\alpha \in (0,1)$, 
 Then the confidence envelope 
  %A confidence envelope $(\ol{\FDP}_k, k\geq 1)$ over a path $\Pi=(R_k, k\geq 1)$
  is said to be $m$-consistent with respect to the sequence $(\P^{(m)},m\geq 1)$ and %it satisfies  \eqref{confenvelop} and if, for 
to the FDR controlling procedure $R^{(m)}_{\hat{k}^{(m)}_\alpha}\in \Pi^{(m)}$  at level $\alpha$ %in a range $[\alpha_0,1)\subset (0,1)$, 
if for all $\epsilon>0$,
\begin{equation}\label{equ-consistency}
\lim_{m \rightarrow \infty} \P^{(m)} \left( %\sup_{\alpha\in [\alpha_0,1)} 
%\set{
  \ol{\FDP}^{(m)}_{\hat{k}^{(m)}_\alpha}-\alpha
\geq \epsilon\right) = 0.
\end{equation}}
%where the asymptotics is when the number of hypotheses $m$ tends to infinity.
%$ \ol{\FDP}_{\hat{k}}$ is asymptotically below $\alpha$ (in probability, and for some asymptotic and uniformly in $\alpha$ is some range ).
\end{definition}
{When applying this definition with respect to the reference procedures described above (BH/LF/LORD), we will speak about BH/LF/LORD $m$-consistency for short.}
%\todo{Say somewhere BH-consistency, LF-consistency, LORD-consistency if we want to underline the refernce procedure and not $m$}
In addition, in the above definition, $\P^{(m)}$ stands for a multiple testing model with $m$ hypotheses that is to be specified. We will be interested in standard model sequences that represent relevant practical situations, and in particular {\it sparse} cases where only a vanishing proportion of null hypotheses are false when $m$ tends to infinity. 
The above definition applies {transparently} for the two first considered settings (top-$k$ and pre-ordered). Note that due to \eqref{confenvelop}, we have 
 \begin{equation}\label{confenvelopkalpha}
	\P\paren{\forall \alpha\in (0,1), \FDP(R_{\hat{k}_\alpha})\leq \ol{\FDP}_{\hat{k}_\alpha}}\geq 1-\delta.
	 \end{equation}
 Hence, \eqref{equ-consistency} comes as an asymptotical accuracy guarantee in addition to the non-asymptotical coverage property \eqref{confenvelopkalpha}. {The uniformity in $\alpha$ in \eqref{confenvelopkalpha} allows for choosing $\alpha$ in a post hoc manner, while maintaining the false discovery control, that is, for any data-dependent choice of $\hat{\alpha}$, $\FDP(R_{\hat{k}_{\hat{\alpha}}})\leq \ol{\FDP}_{\hat{k}_{\hat{\alpha}}}$ with probability at least $1-\delta$. For $m$-consistency, a similar (but weaker) $\alpha$-uniformity property can be obtained, see Remark~\ref{rem:alphachapconsist} below.}

{In the third setting (online), the definition applies in the following sense: \eqref{equ-consistency} reads formally as
\begin{equation}\label{equ-consistency-online}
\lim_{m \rightarrow \infty} \P^{(\infty)} \left( %\sup_{\alpha\in [\alpha_0,1)} 
  \ol{\FDP}_{m}-\alpha
\geq \epsilon\right) = 0.
\end{equation}
Here, the distribution of the data $\P^{(m)}$ does not depend on $m$ and corresponds to $\P^{(\infty)}$ the joint distribution of the countable sequence of observations. % and replace $\P^{(m)}$ in \eqref{equ-consistency}.  
The path $\Pi^{(m)}=(R_k^{(m)}, k\geq 1)=(R_k, k\geq 1)$ does not depend on $m$ and corresponds to the sequence of rejections along the time. %, which thus coincides in that setting with the (online) procedure of interest. 
Also, $\hat{k}_\alpha^{(m)} =m$  in that setting, that is, the bound is built for $R_k$ for $k=m$. In contrast with the two first settings, the procedure of interest is represented by the path itself rather than by some particular choice $\hat{k}_\alpha^{(m)} $ of $k$.
}

{
{

\begin{remark}\label{rem:goeman}
{A notion of consistency for FDP bounds has been
  introduced by \citet[Section 7]{goeman2019simultaneous}.
  In our terminology their notion could be dubbed $(n,m)$-consistency, as
  both the number of hypotheses $m$ and a parameter $n$ modelizing signal-to-noise
  ratio (e.g., size of the underlying sample, or signal strength) grow to infinity. The authors
  consider the following scenario:
  \begin{itemize}
    \item the bound is considered on any set $S_m$ such that, conditional to
      their index being in the set $S_m$, the $p$-values are independently drawn from a mixture
      \begin{equation} \label{eq:goemanmixture}
        \gamma \mathrm{Unif}[0,1] + (1-\gamma) Q_{1,n}.
      \end{equation}
    \item $Q_{1,n}$ approaches a Dirac $\delta_1$ distribution as $n\rightarrow \infty$. (Individual test consistency
      under the alternative)
    \item $\abs{S_m}$ remains lower bounded by a linear function of $m$.
  \end{itemize}
  Under these assumptions the authors show that the Simes-based closed testing bound for $S_m$ is consistent,
  in the sense that it converges asymptotically to $\gamma$. The above assumptions allow,
  in the non-sparse case, to take $S_m=\cH_1^{(m)}$
  and further conclude to the consistency of the bound for any rejection set of size growing linearly with $m$.\\
%  In that reference,
%  the evaluation of
%  the bound is considered on any set $S_m$ such that, conditional to
%  their index being in the set $S_m$, the $p$-values are independently drawn from a mixture $\gamma \mathrm{Unif}[0,1] %+ (1-\gamma) Q_{1,n}$, such that $Q_{1,n}$ approaches a Dirac $\delta_1$
%  distribution as $n\rightarrow \infty$ (there, $n$ is a parameter modelizing
%  signal-to-noise ratio)  \et{and $\gamma\in (0,1)$ is kept fixed}; consistency in that context means that the bound
%  should be asymptotically close to $\gamma$, as $(|S_m|, n)$ both grow to infinity.
%  In such a context \citet[Section 7]{goeman2019simultaneous} shows consistency
%  of the Simes-based closed testing bounds.\\
  % {If the individual tests are consistent (i.e. the distribution of non-null $p$-values
  %   converges to $\delta_1$, it is possible to apply the result of \citet{goeman2019simultaneous}
  %   to $S_m = \cH_1^{(m)}$ and conclude )}
  The notion considered in the present work is of a different nature, as at least one of the above assumptions
  is not met
  in the typical scenarios we will consider:
  \begin{itemize}
  \item in the non-sparse setting, we do not require individual test consistency and consider fixed signal strength.
  \item in the sparse setting, we consider growing signal strength as $m\rightarrow \infty$, which implies individual test consistency but in a way that is only sightly above the threshold of global signal detectability.
    In such a scenario $\abs[1]{\cH_1^{(m)}}$ does not grow linearly with $m$ and neither do the size of rejections sets
    $\abs{S_m}$ of interest.
  \item we evaluate the bound on rejection sets
    from reference procedures that do not follow the posited mixture distribution~\eqref{eq:goemanmixture}.
    For example, conditional to being rejected by BH the $p$-values do not follow a Uniform/Alternative independent mixture
  distribution.
  \end{itemize}
  %first, we do not require
  %individual test consistency \et{(in the non-sparse setting)}.
  % Second,  In the individual tests are consistent, \citet[Section 7]{goeman2019simultaneous} to $S_m=\cH_1$  \todo{A revoir au vu de nos échanges avec Jelle?}
  Thus, in the scenarios we will focus on, some ambiguity is asymptotically remaining and one cannot expect full asymptotic signal identifiability.
  In fact, we show (see Section~\ref{sec:closed-testing}) that in such a scenario, the Simes-based closed testing bound is not $m$-consistent
  in general.}
\end{remark}

%\todo{\et{Put uniformity in $\alpha$ in a side remark ?}}

%\todo{Former sentence -- discuss: Moreover, the uniformity in $\alpha$ in \eqref{confenvelopkalpha}-\eqref{equ-consistency} allows for choosing $\alpha$ in a post hoc manner, while maintaining the false discovery control and without paying too much in accuracy, that is, for any data-dependent choice of $\hat{\alpha}$, $\FDP(R_{\hat{k}_{\hat{\alpha}}})\leq \ol{\FDP}_{\hat{k}_{\hat{\alpha}}}$ with probability at least $1-\delta$, with $\ol{\FDP}_{\hat{k}_{\hat{\alpha}}}\lesssim \hat{\alpha}(1+o(1))$ in `good' cases.}

\subsection{Contributions}

Our findings are as follows: 
\begin{itemize}
\item In each of the considered settings (top-$k$, pre-ordered, online), we provide new (non-asymptotical) FDP confidence envelopes 
that are $m$-consistent under general conditions, including sparse configurations, see Proposition~\ref{prop-consistencytopk}, Corollary~\ref{cor-consistencytopk} (top-$k$), Proposition~\ref{prop-consistency-preordered}, Corollary~\ref{cor-consistency-preordered} (pre-ordered)  and Proposition~\ref{prop-consistency-online}, Corollary~\ref{cor-consistency-online} (online). Table~\ref{tab:recap} provides a summary of the considered procedures in the different contexts, including the existing and new ones. It is worth noting that in the top-$k$ setting, the envelope based on the DKW inequality \citep{Mass1990} is consistent under moderate sparsity assumptions only, while the new  envelope based on Wellner's inequality \citep{shorack2009empirical} covers  the whole sparsity range (Corollary~\ref{cor-consistencytopk}).
\item As a byproduct, our results provide (non-asymptotical) confidence bounds on the FDP for standard FDR-controlling procedures,  that are also asymptotically sharp ($m$-consistency) and for which a data-driven choice of the level $\alpha$ is allowed. 
In particular, in the top-$k$ setting,  this gives a new sharp confidence bound for the achieved FDP of the BH procedure while tuning the level from the same data, see \eqref{BHtopkWellner} below.
%is particularly useful for a practitioner who wants a sharp confidence bound of the achieved FDP of the BH procedure while tuning the level from the same data.
%For instance, in the top-$k$ setting, this is particularly useful for a practitioner who wants a sharp confidence bound of the achieved FDP of the BH procedure while tuning the level from the same data.
%A byproduct of our results is that a post hoc choice of the level  

\item {In the top-$k$ setting, we also theoretically show that the Simes bound can be $m$-inconsistent, even after applying a closed-testing improvement \citep{goeman2011multiple}, see Section~\ref{sec:closed-testing}. Hence, closed-testing does not solve the $m$-inconsistency issue in itself. Also, we develop {\it adaptive} versions of our bounds (for which the proportion of null hypotheses is simultaneously estimated), which can be seen as an improvement stage that is less computationally demanding than closed-testing, while bringing already a large part of the latter's improvement when combined with the interpolation technique of Remark~\ref{rem:interp}, see Section~\ref{sec:closedexp}.}
%Our bounds can be further improved by using a closed-testing approach 
%we also develop {\it adaptive} envelopes, for which the proportion of null hypotheses is simultaneously estimated, see Section~\ref{sec:adaptive}. 
%This is a novel approach with respect to existing literature and it is shown to improve significantly the bounds on simulations in `dense' situations, see Section~\ref{sec:num}. 
%\textcolor{red}{to be removed? no very novel given the closed testing bound...}

\item In the pre-ordered setting, including the `knockoff' case, we introduce new envelopes, called `Freedman' and `KR-U', which are the two first (provably) $m$-consistent confidence bounds in that context to our knowledge. This is an important contribution since the knockoff method is one of the leading methodologies  in the covariate testing literature of the last decade. In addition, KR-U is   shown to behave suitably, even for moderate sample size, see Section~\ref{sec:num}. %Our work provides  the first consistent confidence bound  for such a selection procedure (to our knowledge).

\item {In the online setting, the new envelopes, called `Freedman' and `KR-U', provide an additional information on the LORD procedure via an FDP upper-envelope uniformly valid along time, and converging to the prescribed level $\alpha$ on the long run, see Figure~\ref{fig:onlineimpc}.
}

%\item \iq{ In the online setting, we introduce new envelopes, also called `Freedman' and `KR-U', which are derived similarly to those in the pre-ordered setting. 
%Note that, even though the online setting is a setting where $p$-values have a natural ordre (given by time) and the new bounds are based on the same mathematical tools as in the pre-ordered setting, both settings are indeed distincts since the path ans the procedure are intertwined in the online setting so that the the qualitative conclusions are not of the same extent.}

\item Our study is based on dedicated tools of independent interest,  based on uniform versions of classical deviation inequalities, see Corollary~\ref{cor:wellner} (Wellner's inequality), Corollary~\ref{cor:freedman} (Freedman's inequality). {Both can be seen as a form of `stitching' together elementary inequalities, see \cite{Howardetal2021} for recent developments
    of this principle. %The bounds developed here are presented in a self-contained manner. 
    {In addition, the Freedman-type bounds in the pre-ordered and online settings  are both based on a single martingale-type result, see Theorem~\ref{th-mainresult}.}
   }
%\item They are shown to improve the existing bounds in common scenario even for moderate $m$ (especially for the new uniform KR bounds, hereafter KR-U) on numerical examples, Section~\ref{sec:num}. 
%We also propose other new FDP confidence envelopes (uniform version of KR's inequalities), which are less explicit but sharper in practical applications. 
%\item \et{Talk about online also!}
\end{itemize}

\begin{table}[h!]
\begin{tabular}{|c|cccccc|}
\hline
& Simes & DKW & KR & Wellner (new) & Freedman (new) & KR-U (new)  \\\hline
Top-$k$&  No & Yes &  No & Yes & & \\
Pre-ordered&   &  &  No &  & Yes & Yes\\
Online & & &  No & & Yes & Yes\\
\hline
\end{tabular}
\caption{{$m$-consistency property (Yes or No) for different envelopes, depending on the considered contexts.   `$m$-consistent' means consistent at least in a particular (reasonable) configuration and with respect to the corresponding reference procedure BH/LF/LORD. Blank means undefined in that context. \label{tab:recap}}}
\end{table}

\begin{remark}[$\alpha$-uniform $m$-consistency]\label{rem:alphachapconsist}
{In the top-$k$ and pre-ordered setting, 
it is possible to show slightly more than the convergence \eqref{equ-consistency}, namely we can aim for  
$$\P^{(m)} \left( \sup_{\alpha\in [\alpha_0,1)} 
\set{
  \ol{\FDP}^{(m)}_{\hat{k}^{(m)}_\alpha}-\alpha}
\geq \epsilon\right) = o(1),$$
for any fixed $\alpha_0\in (0,1)$. 
More precisely, we can easily show that this $\alpha$-uniform $m$-consistency property can be obtained in Corollaries~\ref{cor-consistencytopk}~and~\ref{cor-consistency-preordered} below by monotonicity of the reference procedure. 
This allows to account for possible `data snooping' from the user, that is, the consistency property also holds for $\alpha=\hat{\alpha}$ possibly depending on the data, provided that it is larger than some $\alpha_0>0$.
However, for the online setting, such uniformity is out of reach since the full path itself already depends on $\alpha$.}
\end{remark}

\begin{remark}[FDP concentration and $m$-consistency]
{Given the FDR control and \eqref{equ-consistency}, in cases where the FDP of the reference procedure concentrates around its expectation as $m$ grows \citep{genovese2004stochastic,neuvial08asymptotic,neuvial13asymptotic}, we expect that bounds of the form $\alpha + \Delta_{m,\alpha,\delta}$ with $\Delta_{m,\alpha,\delta}=o(1)$ should hold in the sense of Definition~\ref{def-consistency}, and thus $m$-consistent bounds can be built. In such situations, using $m$-inconsistent bounds (such as the Simes bound in the top-$k$ setting) is questionable.}
\end{remark}

\section{Results in the top-$k$ case}

\subsection{Top-$k$ setting}\label{sec:topksetting}

We consider the classical multiple setting where we observe $m$ independent $p$-values $p_1, \dots, p_m$, testing $m$ null hypotheses $H_1, \dots, H_m$. The set of true nulls is denoted by $\cH_0$, which is of cardinal $m_0$ and we denote $\pi_0=m_0/m\in (0,1)$.  We assume that the $p$-values are uniformly distributed under the null, that is, for all $i\in \cH_0$, $p_i\sim U(0,1)$. 

We consider here the task of building a $(1-\delta)$-confidence envelope \eqref{confenvelop} for the top-$k$ path
%$(1-\delta)$-confidence envelope for the top-$k$ path 
\begin{equation}\label{pathtopk}
R_k=\{1\leq i\leq m\::\: p_i\leq  p_{(k)}\},\:\:\: k=1,\dots,m.
\end{equation}
A rejection set of particular interest is the BH rejection set, given by $R_{\hat{k}_\alpha}$ %$=\{1\leq i\leq m\::\: p_i\leq p_{(\hat{k}_\alpha)}\}$ 
where 
\begin{align}\label{equBH}
\hat{k}_\alpha=\max\left\{k \in \mathbb{N}\::\: \wh{\FDP}_k \leq \alpha\right\}, \:\:\wh{\FDP}_k =mp_{k}/k,
\end{align}
(with the convention $R_0=\emptyset$). %Note that we also have $p_{(\hat{k})}\leq \alpha \hat{k}/m$.

\subsection{Existing envelopes}

Let us first review the prominent confidence envelopes that have been considered in the literature.
Let $U_1,\dots,U_n$ be $n\geq 1$ i.i.d. uniform random variables. For $\delta\in (0,1)$, each of the following (uniform) inequalities holds with probability at least $1-\delta$:
\begin{itemize}
\item \cite{Sim1986} (or \citealp{robbins1954one}): for all $t\in (0,1)$,
$
n^{-1} \sum_{i=1}^n \ind{U_i\leq t} \leq t /\delta.
$
\item DKW \citep{Mass1990}: for all $t\in (0,1)$,
$
n^{-1} \sum_{i=1}^n \ind{U_i\leq t} \leq t + \sqrt{\log(1/\delta)/2}\: n^{-1/2} .
$
\item KR \citep{katsevich2020simultaneous} (for $\delta\leq 0.31$), for all $t\in (0,1)$,
$
n^{-1} \sum_{i=1}^n \ind{U_i\leq t}\leq  \frac{\log(1/\delta)}{\log(1+\log(1/\delta))}(1/n + t).
$
\end{itemize}

Taking $(U_1,\dots,U_n)=(p_i,i\in \cH_0)$, $n=m_0$, and $t=p_{(k)}$ in the bounds above gives the following  confidence envelopes (in the sense of \eqref{confenvelop}) for the top-$k$ path: for $k\in\{ 1,\dots,m\}$,
\begin{align}
\ol{\FDP}^{{\tiny \mbox{Simes}}}_k &:=1\wedge \frac{ m p_{(k)}}{k\delta}; \label{topkSimes}\\
\ol{\FDP}^{{\tiny \mbox{DKW}}}_k &:=  1\wedge\left(\frac{ m p_{(k)}}{k} + \frac{m^{1/2}\sqrt{0.5 \log 1/\delta}}{k}\right) ; \label{topkDKW}\\
\ol{\FDP}^{{\tiny \mbox{KR}}}_k &:= 1\wedge\left( \frac{\log(1/\delta)}{\log(1+\log(1/\delta))} \left(\frac{ m p_{(k)}}{k}  + 1/k\right)\right) ,
\label{topkKR}\end{align}
the last inequality requiring in addition $\delta\leq 0.31$. Note that we can slightly improve these bounds by taking appropriate integer parts, but we will ignore this detail further on for the sake of simplicity.

\subsection{New envelope}

In addition to the above envelopes, this section presents a new one deduced from a new `uniform' variation of Wellner's inequality (recalled in Lemma~\ref{Wellnerineq}). Let us first define the function
%\gil{(the log-Laplace transform of a Poisson variable)}
\begin{equation}\label{funh}
h(\lambda)=\lambda(\log \lambda -1)+1,\:\:\:\lambda>1.
\end{equation}
Lemma~\ref{lem:h} gathers some properties of $h$, including explicit accurate bounds for $h$ and $h^{-1}$.

\begin{proposition}[Uniform version of Wellner's inequality]\label{cor:wellner}
Let $U_1,\dots,U_n$ be $n\geq 1$ i.i.d. uniform random variables and $\kappa= \pi^2/6$. For all $\delta \in (0,1)$, we have with probability at least $1-\delta $,
\begin{equation}\label{WellnerUniform}
\forall t \in (0,1), \:\: n^{-1} \sum_{i=1}^n \ind{U_i\leq t} \leq t\: h^{-1}\left(\frac{\log(\kappa/\delta) + 2 \log\left(\lceil  \log_2(1/t) \rceil\right)}{n g(t) }\right),
\end{equation}
for $g(t)=2^{-\lceil  \log_2(1/t) \rceil}/(1-2^{-\lceil  \log_2(1/t) \rceil})\geq t/2$ and $h(\cdot)$ defined by \eqref{funh}. %$h(\lambda)=\lambda(\log \lambda -1)+1$ increasing one to one map from $(1,\infty)$ to $(0,\infty)$.
In particular, with probability at least $1-\delta $,
\begin{equation}\label{WellnerUniformSimple}
\forall t \in (0,1),\:\:  n^{-1} \sum_{i=1}^n \ind{U_i\leq t} \leq t\: h^{-1}\left(\frac{2\log(\kappa/\delta) + 4 \log\left(1+  \log_2(1/t) \right)}{n t }\right).
\end{equation}
\end{proposition}

The proof of Proposition~\ref{cor:wellner} is given in Section~\ref{sec:newecdf}. It immediately leads to the following result.

\begin{theorem}\label{th-BH}
In the top-$k$ setting of Section~\ref{sec:topksetting}, the following quantity is a $(1-\delta)$-confidence envelope in the sense of \eqref{confenvelop} for the top-$k$ path:
\begin{equation}\label{topkWellner}
\ol{\FDP}_k^{{\tiny \mbox{Well}}} :=    1\wedge\left( \frac{ m p_{(k)}}{k}\: h^{-1}\left(\frac{2\log(\kappa/\delta) + 4 \log\left(1+  \log_2(1/p_{(k)}) \right)}{m p_{(k)} }\right)\right),\end{equation}
with $\kappa=\pi^2/6$.
\end{theorem}

\begin{proof}
We use \eqref{WellnerUniformSimple} for $(U_1,\dots,U_n)=(p_i,i\in \cH_0)$, $n=m_0$, and $t=p_{(k)}$. We conclude by using $m_0\leq m$ and the monotonicity property of Lemma~\ref{lem:h}.
\end{proof}

\begin{remark}
Denoting by $\ol{F}_n(t)$ the RHS of \eqref{WellnerUniformSimple},  we can easily check
\begin{align*}
 \sup_{t\in ((\log\log n)/n,1)}\left( \sqrt{n}\frac{\ol{F}_n(t)-t}{\sqrt{t \log\left(1+  \log_2(1/t) \right)}}\right) 
 &=O(1),
\end{align*}
with a constant possibly depending on $\delta$. The iterated logarithm in the denominator is known from classical
asymptotic theory (convergence to a Brownian bridge) to be unimprovable for a uniform bound in the vicinity of $0$;
in this sense the above is a `finite law of the iterated logarithm (LIL) bound' \citep{jamieson2014lil}.
\end{remark}

\subsection{FDP confidence bounds for BH and $m$-consistency}

Applying the previous bounds for the particular BH rejection sets $R_{\hat{k}_\alpha}$ (see \eqref{equBH}) leads to the following result.

\begin{corollary}\label{cor-BH}
In the top-$k$ setting of Section~\ref{sec:topksetting}, for any $\alpha,\delta\in (0,1)$, the following quantities are $(1-\delta)$-confidence bounds for  $\FDP(R_{\hat{k}_\alpha})$, the FDP of the BH procedure at level $\alpha$:
\begin{align}
\ol{\FDP}^{{\tiny \mbox{Simes}}}_\alpha &:= %1\wedge \frac{ m p_{(\hat{k}_\alpha)}}{\hat{k}_\alpha\delta}\leq 
\ind{\hat{k}_\alpha\geq 1}\wedge (\alpha/\delta); \label{BHtopkSimes}\\
\ol{\FDP}^{{\tiny \mbox{DKW}}}_\alpha &:=  \ind{\hat{k}_\alpha\geq 1}\wedge \left(\alpha + \frac{m^{1/2}\sqrt{0.5 \log 1/\delta}}{1\vee \hat{k}_\alpha}\right) ; \label{BHtopkDKW}\\
\ol{\FDP}^{{\tiny \mbox{KR}}}_\alpha &:= \ind{\hat{k}_\alpha\geq 1}\wedge \left( \frac{\log(1/\delta)}{\log(1+\log(1/\delta))} \left( \alpha + 1/(1\vee \hat{k}_\alpha)\right) \right) ;
\label{BHtopkKR}\\
\ol{\FDP}_\alpha^{{\tiny \mbox{Well}}} &:=   \ind{\hat{k}_\alpha\geq 1}\wedge \left(   \alpha\: h^{-1}\left(\frac{2\log(\kappa/\delta) + 4 \log\left(1+  \log_2\Big(\frac{m}{\alpha(1\vee \hat{k}_\alpha)}\Big) \right)}{\alpha (1\vee \hat{k}_\alpha) }\right)\right),\label{BHtopkWellner}
\end{align}
where $\kappa=\pi^2/6$, $\hat{k}_\alpha$ denotes the number of rejections of the BH procedure \eqref{equBH} at level $\alpha$, 
and where the KR bound requires in addition $\delta\leq 0.31$.
 Moreover, these bounds are also valid uniformly in $\alpha\in (0,1)$, in the sense that
 $$
 \P(\forall \alpha\in (0,1), \FDP(R_{\hat{k}_\alpha})\leq \ol{\FDP}_\alpha^{{\tiny \mbox{Method}}})\geq 1-\delta, \:\:\mbox{Method}\in \{{ \mbox{Simes}},{ \mbox{DKW}},{ \mbox{KR}},{ \mbox{Well}}\},
 $$
 and thus also when using a post hoc choice $\alpha=\hat{\alpha}$ of the level.
\end{corollary}

\begin{proof}
For \eqref{BHtopkWellner}, we use \eqref{WellnerUniformSimple} for $(U_1,\dots,U_n)=(p_i,i\in \cH_0)$, $n=m_0$, and $t=\alpha (1\vee \hat{k}_\alpha)/m$.
\end{proof}

Let us now consider the $m$-consistency property \eqref{equ-consistency}, by using BH as reference procedure.
Among the four above bounds, it is apparent that Simes and KR are not BH $m$-consistent, because of the constant in front of $\alpha$; namely, for all $m$,
$$
\ol{\FDP}^{{\tiny \mbox{Simes}}}_\alpha \wedge \ol{\FDP}^{{\tiny \mbox{KR}}}_\alpha\geq (1\wedge (c \alpha))\wedge\ind{\hat{k}_\alpha\geq 1},
$$
for some constant $c>1$, which implies the BH $m$-inconsistency for a sequence of model such that $\hat{k}_\alpha\geq 1$ with an asymptotically non-null probability. By contrast, $\ol{\FDP}_\alpha^{{\tiny \mbox{DKW}}}$  and $\ol{\FDP}_\alpha^{{\tiny \mbox{Well}}} $ are BH $m$-consistent in the following sense.

{
%\todo{GB: why is the proposition restricted to BH? Isn't the condition on $\wh{k}_\alpha$ + FDR guarantee on the procedure sufficient?}
\begin{proposition}\label{prop-consistencytopk}
Let us consider any model sequence $\P^{(m)}$ in the top-$k$ setting and denote by $\hat{k}_{\alpha}$ the number of nulls rejected by the BH procedure at level $\alpha$. Then the following envelopes are 
BH $m$-consistent in the sense of \eqref{equ-consistency}: 
\begin{itemize}
\item $\ol{\FDP}_\alpha^{{\tiny \mbox{DKW}}}$ if $ m^{1/2}/\hat{k}_{\alpha}=o_P(1)$;
\item $\ol{\FDP}_\alpha^{{\tiny \mbox{Well}}} $ if $(\log\log m)/ \hat{k}_{\alpha}=o_P(1)$.
\end{itemize} 
\end{proposition}
The latter result means that the BH procedure at level $\alpha$ should make enough rejections in order to provide $m$-consistency. In the two-group model of  \citep{ETST2001} with a fixed proportion of alternatives in $(0,1)$ (that is, a ``dense'' case), under some assumptions on the alternative distribution, and assuming $\alpha$ above a critical value, \cite{Chi2007} showed that $\hat{k}_{\alpha}$ is asymptotically of the order of $m$ and thus the DKW and Wellner bounds are both $m$-consistent. % (in a top-$k$ model including random effects).
Another exemple, including sparse situations, is considered in the next section.
}
%{Mettre une proposition pour ce qui est au dessus}

\subsection{{BH $m$-consistency in a prototypical model}\label{sec:powertopk}}

{
\begin{definition}
The sparse one-sided Gaussian location model of parameter $m,b,c,\beta$, denoted as $\P^{(m)}_{b,c,\beta}$, is defined as follows: $p_i=\ol{\Phi}(X_i)$, $1\leq i\leq m$, the  $X_i$'s are independent, with $X_i\sim \mathcal{N}(0,1)$ for $i\in \cH_0$ and  $X_i\sim \mathcal{N}(\mu_m,1)$ otherwise, for $\mu_m=\sqrt{2\beta\log m}+b$, $b>0$, and $m_1=|\cH_1|=\lfloor cm^{1-\beta}\rfloor$, $c\in (0,1)$, $\beta\in [0,1)$. 
\end{definition}
}
{
Note that $\beta=0$ is the dense case for which the alternative mean $\mu_m=b>0$ is a fixed quantity, which means that the individual tests do not have full power\footnote{This setting is in sharp contrast with \citet[Section 7]{goeman2019simultaneous} which assumed asymptotical full power for the individual tests and the dense case, see also Remark~\ref{rem:goeman}.}, even asymptotically w.r.t. $m$. By contrast, $\beta>0$ corresponds to the sparse case, for which $\mu_m=\sqrt{2\beta\log m}+b$ tends to infinity. In both cases, the magnitude of alternative means is defined to be on the `edge of detectability' where the BH procedure has some non-trivial power, see \cite{bogdan2011asymptotic,neuvial2012false,abraham2021sharp} for instance.
}

{
\begin{theorem}\label{th:BHasymp}
Let $\alpha\in (0,1)$. In the above one-sided Gaussian location model $\P^{(m)}_{b,c,\beta}$, the number of rejections $\hat{k}_\alpha$ of the BH procedure is such that, as $m$ grows to infinity,
\begin{equation}\label{mainresultlemmatopk}
%\P(\alpha \hat{k}_\alpha/m < t^*_m)\vee \P(\alpha \hat{k}_\alpha/m >t^\sharp_m) \leq e^{-d m_1 } , \mbox{ for } t^*_m \gtrsim m_1/m , \:\:t^\sharp_m \lesssim m_1/m,
\P^{(m)}_{b,c,\beta}( t^*_m\leq \alpha \hat{k}_\alpha/m \leq t^\sharp_m) \geq 1-2e^{-d m_1 } , \mbox{ for } m_1/m  \lesssim t^*_m \leq t^\sharp_m \lesssim m_1/m,
\end{equation}
for some constant $d>0$ (depending on $\alpha$, $\beta$, $b$), 
where $t^*_m\in (0,1)$ is the unique solution of $G_m(t)=2t/\alpha$, 
$t^\sharp_m\in (0,1)$ is the unique solution of $G_m(t)=0.5 t/\alpha$,
 and where $G_m(t)=(m_0/m) t + (m_1/m) \ol{\Phi}(\ol{\Phi}^{-1}(t)-\mu_m)$, with $\ol{\Phi}(z)=\P(Z\geq z)$, $z\in \R$.
\end{theorem}
Theorem~\ref{th:BHasymp} is proved in Section~\ref{sec:proof:th:BHasymp}. It implies $\hat{k}_{\alpha}\asymp_{\P^{(m)}_{b,c,\beta}} m^{1-\beta}$, which leads to the following result.
\begin{proposition}\label{cor-consistencytopk}
Let us consider the sequence of sparse one-sided Gaussian location models $(\P^{(m)}_{b,c,\beta},m\geq 1)$ with fixed parameters  $b\in\R$, $c\in(0,1)$ and a sparsity parameter $\beta\in [0,1)$, as defined above. Then we have for all $\alpha\in (0,1)$,
\begin{align*}
\ol{\FDP}_\alpha^{{\tiny \mbox{DKW}}}-\alpha& \asymp_{\P^{(m)}_{b,c,\beta}} m^{-1/2+\beta};\\
\ol{\FDP}_\alpha^{{\tiny \mbox{Well}}}-\alpha &\asymp_{\P^{(m)}_{b,c,\beta}} %\sqrt{\frac{\log\log(m)}{m^{1-\beta}}}
\sqrt{\log\log(m)} \:m^{-1/2+\beta/2},
\end{align*}
where $u_m\asymp_P v_m$ stands for $u_m=O_P(v_m)$ and $v_m=O_P(u_m)$.
In particular, concerning the BH $m$-consistency \eqref{equ-consistency}  for the model sequence $(\P^{(m)}_{b,c,\beta},m\geq 1)$:
\begin{itemize}
\item the DKW bound \eqref{topkDKW},\eqref{BHtopkDKW} is $m$-consistent when $\beta<1/2$ but fails to be for $\beta\geq 1/2$; 
%for the DKW envelope \eqref{topkDKW} and the corresponding bound \eqref{BHtopkDKW}, the consistency  \eqref{equ-consistency} holds when $\beta<1/2$ but fails for $\beta\geq 1/2$; 
\item the Wellner bound \eqref{topkWellner},\eqref{BHtopkWellner} is $m$-consistent for any arbitrary $\beta\in (0,1)$.
%for the Wellner envelope \eqref{topkWellner} and the corresponding bound \eqref{BHtopkWellner}, the consistency  \eqref{equ-consistency} holds for any arbitrary $\beta\in (0,1)$.
\end{itemize}
%In addition, the consistency \eqref{equ-consistency} fails for DKW envelope \eqref{topkDKW} when $\beta\geq 1/2$.
% \eqref{equ-consistency} holds for any fixed $\alpha\in (0,1)$ in either of the two following cases:
\end{proposition}
}
%\et{Is $\asymp_P$ standard or should we define it?}

%\begin{proof}
%By Theorem~\ref{th:BHasymp},  we have $\hat{k}_{\alpha}\asymp_{\P^{(m)}_{b,c,\beta}} m^{1-\beta}$.
%%both $m^{1-\beta}/ \hat{k}_{\alpha}=O_P(1)$ and $ \hat{k}_{\alpha}/m^{1-\beta}=O_P(1)$. 
%This gives the result, by applying in addition Lemma~\ref{lem:h} for the Wellner bound.
%\end{proof}
{
Corollary~\ref{cor-consistencytopk} shows the superiority of the Wellner bound on the DKW bound for achieving the $m$-consistency property on a particular sparse sequence models: while the DKW bound needs a model dense enough ($\beta<1/2$), the Wellner bound covers the whole sparsity range $\beta\in(0,1)$.
}

\subsection{Adaptive envelopes}\label{sec:adaptive}

Let us consider the following upper-bounds for $m_0$:
\begin{align}
\hat{m}_0^{{\tiny \mbox{Simes}}} &:= m\wedge \inf_{t\in (0,\delta)}  \frac{V_t}{1-t/\delta}\label{m0chapSimes};\\
\hat{m}_0^{{\tiny \mbox{DKW}}} &:= m\wedge \inf_{t\in (0,1)}  \left( \frac{C^{1/2}}{2(1-t)} + \sqrt{\frac{C}{4(1-t)^2}+\frac{V_t}{1-t}}\right)^2\label{m0chapDKW};\\
\hat{m}_0^{{\tiny \mbox{KR}}} &:= m\wedge \inf_{t\in (0,1/C')}  \frac{C'+V_t}{1-C' t}\label{m0chapKR};\\
\hat{m}_0^{{\tiny \mbox{Well}}} &:= m\wedge \inf_{t\in (0,1)}  \left( \sqrt{\frac{tC_t}{2(1-t)^2}} + \sqrt{\frac{C_t}{{2(1-t)^2}}+\frac{V_t}{1-t}}\right)^2\label{m0chapWellnerU}, 
\end{align}
where $V_t=\sum_{i=1}^m \ind{p_i> t}$, $C=\log(1/\delta)/2$, $C'=\frac{\log(1/\delta)}{\log(1+\log(1/\delta))}$, $C_t=2\log(\kappa/\delta) + 4 \log\left(1+  \log_2(1/t)\right)$, $\kappa=\pi^2/6$. 
Since $V_t/(1-t)$ corresponds to the so-called Storey estimator \cite{storey2002direct}, these four estimators can all be seen as Storey-type confidence bounds, each including a specific deviation term that takes into account the probability error $\delta$. Note that $\hat{m}_0^{{\tiny \mbox{DKW}}}$ was already proposed in  \cite{durand2020post}.

%{$\hat{m}_0^{{\tiny \mbox{Well}}}$ maybe not so good because $t$ close to $0$ is not the spirit of Storey estimator}

\begin{proposition}\label{prop:adapttopk}
In the top-$k$ setting  of Section~\ref{sec:topksetting}, the envelopes defined by \eqref{topkSimes}, \eqref{topkDKW}, \eqref{topkKR} and \eqref{topkWellner} with $m$ replaced by the corresponding bound $\hat{m}_0^{{\tiny \mbox{Simes}}}$ \eqref{m0chapSimes}, $\hat{m}_0^{{\tiny \mbox{DKW}}}$ \eqref{m0chapDKW}, $\hat{m}_0^{{\tiny \mbox{KR}}}$ \eqref{m0chapKR} or $\hat{m}_0^{{\tiny \mbox{Well}}}$ \eqref{m0chapWellnerU}, respectively, are also $(1-\delta)$-confidence envelopes in the sense of \eqref{confenvelop} for the top-$k$ path.
\end{proposition}

%\et{Poser $\kappa=\pi^2/6$?}

We can easily check that these four adaptive envelopes all uniformly improve their own non-adaptive counterpart.
The proof of Proposition~\ref{prop:adapttopk} is provided in Section~\ref{proofprop:adapttopk}.

\begin{remark}
In practice, the bounds $\hat{m}_0^{{\tiny \mbox{Simes}}}$ \eqref{m0chapSimes}, $\hat{m}_0^{{\tiny \mbox{DKW}}}$ \eqref{m0chapDKW}, $\hat{m}_0^{{\tiny \mbox{KR}}}$ \eqref{m0chapKR} or $\hat{m}_0^{{\tiny \mbox{Well}}}$ \eqref{m0chapWellnerU} can be computed by taking an infimum over $t=p_{(k)}$, $1\leq k\leq m$ and by replacing $V_t$ by $m-k$.
\end{remark}

Applying Proposition~\ref{prop:adapttopk} for the BH procedure, this gives rise to the following adaptive confidence bounds.

\begin{corollary}\label{cor-BH-adapt}
In the top-$k$ setting of Section~\ref{sec:topksetting}, for any $\alpha,\delta\in (0,1)$, the following quantities are $(1-\delta)$-confidence bounds for the FDP of the BH procedure at level $\alpha$:
\begin{align}
\ol{\FDP}^{{\tiny \mbox{Simes-adapt}}}_\alpha &:=1\wedge \alpha(\hat{m}_0^{{\tiny \mbox{Simes}}}/{m})/\delta; \label{BHtopkSimes-adapt}\\
\ol{\FDP}^{{\tiny \mbox{DKW-adapt}}}_\alpha &:=  1\wedge \Big(\alpha(\hat{m}_0^{{\tiny \mbox{DKW}}}/{m}) + \frac{(\hat{m}_0^{{\tiny \mbox{DKW}}})^{1/2}\sqrt{0.5 \log 1/\delta}}{1\vee \hat{k}_\alpha} \Big); \label{BHtopkDKW-adapt}\\
\ol{\FDP}^{{\tiny \mbox{KR-adapt}}}_\alpha &:=  1\wedge \left(\frac{\log(1/\delta)}{\log(1+\log(1/\delta))} \left( \alpha(\hat{m}_0^{{\tiny \mbox{KR}}}/{m}) + 1/(1\vee \hat{k}_\alpha)\right)\right) ;
\label{BHtopkKR-adapt}\\
\ol{\FDP}_\alpha^{{\tiny \mbox{Well-adapt}}} &:=    1\wedge \left(  \alpha (\hat{m}_0^{{\tiny \mbox{Well}}}/{m})\: h^{-1}\left(\frac{2\log(\kappa/\delta) + 4 \log\left(1+  \log_2\Big(\frac{m}{\alpha(1\vee \hat{k}_\alpha)}\Big) \right)}{\alpha (1\vee \hat{k}_\alpha) \hat{m}_0^{{\tiny \mbox{Well}}}/{m} }\right)\right),\label{BHtopkWellner-adapt}
\end{align}
where $\kappa=\pi^2/6$, $\hat{k}_\alpha$ denotes the number of rejections of BH procedure \eqref{equBH} at level $\alpha$, 
and where the KR-adapt bound requires in addition $\delta\leq 0.31$.
 Moreover, these bounds are also valid uniformly in $\alpha\in (0,1)$ and thus also when using a post hoc choice $\alpha=\hat{\alpha}$ of the level.
\end{corollary}

\begin{proof}
For \eqref{BHtopkWellner-adapt}, we use \eqref{WellnerUniformSimple} for $(U_1,\dots,U_n)=(p_i,i\in \cH_0)$, $n=m_0$, $t=\alpha (1\vee \hat{k}_\alpha)/m$, and the fact that $m_0\leq \hat{m}_0^{{\tiny \mbox{Well}}}$ on the considered event by the proof in Section~\ref{proofprop:adapttopk}. The other bounds are proved similarly. 
\end{proof}

%\todo{Simes adaptive still not consistent}

\subsection{Interpolated bounds}\label{sec:interptopk}

According to Remark~\ref{rem:interp}, the coverage \eqref{confenvelop} is still valid after the interpolation operation given by \eqref{equ-interp3}. As a result, the above confidence envelopes can be improved as follows:
%\eqref{equ-interp} (note that formula \eqref{equ-interp3} applies for Simes, DKW, KR and formula \eqref{equ-interp2} applies for Wellner). 
\begin{align}
\wt{\FDP}^{{\tiny \mbox{Simes}}}_k &:=
\min_{k'\leq k} \{k-k'+  k'\wedge (m p_{(k')}/\delta)\} /k;
%\wt{\FDP}^{{\tiny \mbox{Simes}}}_\alpha &= \min_{k'\leq \hat{k}_\alpha} \{\hat{k}_\alpha-k'+  k'\wedge (m p_{(k')}/\delta)\} /\hat{k}_\alpha
%1 -\max_{k'\leq k} \{k'-  k'\wedge (m p_{(k')}/\delta)\} /k
 \label{topkSimesInterp}\\
\wt{\FDP}^{{\tiny \mbox{DKW}}}_k &:= 
\min_{k'\leq k} \{k-k'+ k'\wedge (m p_{(k')} + m^{1/2}\sqrt{0.5 \log 1/\delta})\} /k; \label{topkDKWInterp}\\
\wt{\FDP}^{{\tiny \mbox{KR}}}_k &:= \min_{k'\leq k} \left\{k-k'+ k'\wedge\left( \frac{\log(1/\delta)}{\log(1+\log(1/\delta))} \left( m p_{(k')}  + 1\right)\right)\right\} /k;
\label{topkKRInterp}\\
\wt{\FDP}_k^{{\tiny \mbox{Well}}} &:=  \min_{k'\leq k} \left\{k-k'+ k'\wedge\left(    m p_{(k')}\: h^{-1}\left(\frac{2\log(\kappa/\delta) + 4 \log\left(1+  \log_2(1/p_{(k')}) \right)}{m p_{(k')} }\right)\right)\right\}%\nonumber\\
%&\vee \left\{k- \min_{k'\geq k+1 } \left\{k'\wedge\left(    m p_{(k')}\: h^{-1}\left(\frac{2\log(\kappa/\delta) + 4 \log\left(1+  \log_2(1/p_{(k')}) \right)}{m p_{(k')} }\right)\right)\right\}\right\} 
/k,
\label{topkWellnerInterp}
\end{align}
respectively. When applied to BH rejection set, this also provides new confidence bounds $\wt{\FDP}^{{\tiny \mbox{Simes}}}_\alpha $, 
$\wt{\FDP}^{{\tiny \mbox{DKW}}}_\alpha$, $\wt{\FDP}^{{\tiny \mbox{KR}}}_\alpha $, $
\wt{\FDP}_\alpha^{{\tiny \mbox{Well}}}$, that can further be improved by replacing $m$ by the corresponding estimator $\hat{m}_0$.

%\todo{Possibly point out: (1) we can combine adaptive (previous section) with interpolation (2) if we use Hommel's $\wh{m}_0$ for adaptation + interpolation, this is equivalent to closed testing based on Simes, and cannot be further improved based on Simes}
%
%\todo{Say that these improvements for Simes/KR cannot lead to consistency either!}
%
%\todo{Do an additional section or remark saying that Lemma 6 de 'only closed testing' can improve a bit more. In fact, it is the same with a different estimator of $m_0$! Do a proof somewhere?}
%
%\todo{There is a regime where Simes close testing is not consistent whereas the others are!}
%
%\et{Gilles contrex:
%On the event where global Simes do not reject, we have the trivial thing
%\begin{align}
%\wt{\FDP}^{{\tiny \mbox{Simes}}}_k =
%\min_{k'\leq k} \{k-k'+  k'\wedge (m p_{(k')}/\delta)\} /k = k
%\end{align}
%but it could be that $\hat{k}_\alpha$ converges to infinity...}
%
%\et{Etienne contrex: DU 
%\begin{align}
%\wt{\FDP}^{{\tiny \mbox{Simes}}}_{k_\alpha} = (k_\alpha-m_1)_+/k_\alpha = (1-m_1/k_\alpha)_+ 
%\end{align}
%with $k_\alpha= m_1/(1/\alpha - \pi_0)$. Hence $\wt{\FDP}^{{\tiny \mbox{Simes}}}_{k_\alpha} =(1-(1/\alpha - \pi_0))_+ = (1-1/\alpha + \pi_0)_+=0$ for $\alpha\leq 0.5$.
%}

\subsection{{Comparison to closed testing based on Simes local tests}}\label{sec:closed-testing}
{
Our bounds can be further improved by using a closed-testing approach \citep{goeman2011multiple} (see Lemma~6 of \citealp{GHS2021} for an explicit formula). It is legitimate to ask if the $m$-inconsistency of the Simes bound is still true with this refinement. The following result establishes that, as expected, the closed-testing version of Simes bound is still $m$-inconsistent. It uses 
%
%: while closed-testing versions of our bounds are always more accurate, the adaptive bounds are less computationally demanding while yield similar accuracy??? Section~\ref{sec:closedexp}
%
%
%The following result holds true in
 the top-$k$ setting of Section~\ref{sec:topksetting}, for which we added random effects for the true/false null hypotheses (two-group model of \citealp{ETST2001}).
%\gil{
% \begin{proposition}
% Given a target FDR level $\alpha \in (0,1)$, there exists a sequence of models $\P^{(m)}$ with proportion of nulls $\pi_0^{(m)} \rightarrow \pi_0^*<1$ and
% fixed alternative distribution (not depending on $m$); and a confidence level $\delta>0$ 
% such that the Simed-based closed testing bound at level $\delta$ is $BH(\alpha)$-inconsistent, but
% $\ol{\FDP}_\alpha^{{\tiny \mbox{DKW}}}$ and $\ol{\FDP}_\alpha^{{\tiny \mbox{Well}}}$ are.
% \end{proposition}
\begin{proposition}\label{prop:inconsist}
  Consider an i.i.d. mixture model $\P^{(m)}_{\pi_0,G}$ of $m$ independent $p$-values with proportion of nulls
  $\pi_0$ and marginal CDF independent of $m$ given by
  \[
    F(t) = \pi_0 t + (1-\pi_0)G(t),
    \]
    with $G$ the cdf of an (alternative) distribution having continuous decreasing density $g$ on $[0,1]$ (so that $g(0)>1$). Then if $\alpha,\delta$ are such that
    \[
      0 < \delta < \frac{1}{\pi_0 + (1-\pi_0)g(0)} < \alpha <1,
      \]
      the Simes-based closed testing bound\footnote{See \cite{goeman2011multiple,goeman2019simultaneous} for a formal definition and useful formulations.} at level $\delta$ is not $\BH(\alpha)$-$m$-consistent
      in the model $\P^{(m)}_{\pi_0,G}$, but
 $\ol{\FDP}_\alpha^{{\tiny \mbox{DKW}}}$ and $\ol{\FDP}_\alpha^{{\tiny \mbox{Well}}}$ are.
\end{proposition}
%The proof of Proposition~\ref{prop:inconsist} is given in Section~\ref{sec:proof:prop:inconsist}. 
%\et{Note that the setting of Proposition~\ref{prop:inconsist} includes the one of Proposition~\ref{cor-consistencytopk} for which DKW's and Wellner's bounds are $\BH(\alpha)$-$m$-consistent.}
 Since closed testing bounds are by essence more accurate than adaptive/interpolated ones,  Proposition~\ref{prop:inconsist} also shows that the versions \eqref{BHtopkSimes-adapt} and \eqref{topkSimesInterp} of the Simes bound are not $\BH$-$m$-consistent.
 Also, numerical experiments suggest that the improvement brought by closed testing is only modest when compared to the adaptive interpolated versions of our bounds (see Section~\ref{sec:closedexp}),
 which are less computationally demanding.
}

\section{Results in the pre-ordered case}

In this section, we build $m$-consistent envelopes in the case where the $p$-values are ordered a priori, which covers the famous `knockoff' case.

\subsection{Pre-ordered setting}\label{sec:knockoffsetting}

Let  $\pi:\{1,\dots,m\}\to \{1,\dots,m\}$ be some ordering of the $p$-values that is considered as given and deterministic (possibly coming from independent data). 
The pre-ordered setting is formally the same as the one of Section~\ref{sec:topksetting}, except that the $p$-value set is explored according to $\pi(1),\pi(2), \dots, \pi(m)$. The rationale behind this is that the alternative null hypotheses $\cH_1=\{1,\dots,m\}\backslash \cH_0$ are implicitly expected to be more likely to have a small rank in the ordering $\pi$
% occur in the beginning of $\pi(1),\pi(2), \dots, \pi(m)$
(although this condition is not needed for the future controlling results to hold).

Formally, the considered path is 
\begin{equation}\label{pathpreordered}
R_k=\{\pi( i) \::\: 1\leq i\leq k,  \:\:p_{\pi(i)}\leq  s\},\:\:\: k=1,\dots,m,
\end{equation}
for some fixed additional threshold $s\in (0,1]$ (possibly determined from independent data) and can serve to make a selection. The aim is still to find envelopes $(\ol{\FDP}_k)_k$ satisfying \eqref{confenvelop} for this path while being $m$-consistent. To set up properly the consistency, we should consider an FDR controlling procedure that is suitable in this setting. For this, we consider the Lei Fithian (LF) adaptive Selective sequential step-up procedure \citep{lei2016power}. The latter is defined by $R_{\hat{k}_\alpha}$ where 
\begin{equation}\label{kchapeauLF}
\hat{k}_\alpha = \max \left\{ k \in \{0, \dots, m \}: \wh{\FDP}_k\leq \alpha \right\}, \:\:\wh{\FDP}_k=\frac{s}{1-\lambda}\frac{ 1 + \sum_{i=1}^{k} \ind{p_{\pi(i)} >  \lambda}}{1\vee \sum_{i=1}^{k} \ind{p_{\pi(i)} \leq s}},
\end{equation}
where $\lambda\in[0,1)$ is an additional parameter.
The `knockoff' setting of \cite{barber2015controlling} can be seen as a particular case of this pre-ordered setting, where the $p$-values are independent and binary, the ordering is independent of the $p$-values and $s=\lambda=1/2$. The LF procedure reduces in that case to the classical Barber and Cand\`es (BC) procedure.

\subsection{New confidence envelopes}

The first envelope is as follows.

\begin{theorem}\label{th-preordered-Freed}
	Consider the pre-ordered setting of Section~\ref{sec:knockoffsetting} with $s\in (0,1]$. For all $\delta \in (0, 1)$, $\lambda\in [0,1)$, the following is a $(1 - \delta)$-confidence envelope for the ordered path \eqref{pathpreordered} in the sense of \eqref{confenvelop}: 
	\begin{equation} \label{preorderedFreed}
		\ol{\FDP}_k^{{\tiny \mbox{Freed}}} :=1\wedge \frac{\frac{s}{1-\lambda}\sum_{i=1}^k \ind{p_{\pi(i)} >  \lambda} + \Delta(\nu k) }{\sum_{i=1}^k \ind{p_{\pi(i)} \leq  s} }, \: k \geq 1,
	\end{equation}
	where 
	$\Delta(u)=2\sqrt{\varepsilon_u}\sqrt{{(u\vee 1)}} +  {\frac{1}{2}} \varepsilon_u$, 
	$\varepsilon_u=\log((1+\kappa)/\delta)+ 2\log\left(1+\log_2\left(u \vee 1\right) \right)$, $u>0$, $\kappa= \pi^2/6$ and $\nu=s( 1 + \min(s,\lambda)/(1-\lambda))$.
\end{theorem}

The proof of Theorem~\ref{th-preordered-Freed} is a direct consequence of a more general result (Theorem~\ref{th-mainresult}), itself being a consequence of a uniform version of Freedman's inequality (see Section~\ref{sec:freedman}). 

The second result is based on the KR envelope \citep{katsevich2020simultaneous}:
\begin{equation}\label{FDPKRknockoff}
	\ol{\FDP}^{\mbox{\tiny KR}}_{k} :=1\wedge \left(\frac{ \log(1/\delta)}{a \log(1+\frac{1 - \delta^{B/a}}{B})} \frac{a +\frac{s}{1-\lambda} \sum_{i=1}^{k} \ind{p_{\pi(i)} >  \lambda}}{  1\vee\sum_{i=1}^{k} \ind{p_{\pi(i)} \leq  s}}\right),
\end{equation}
where $a>0$ is some parameter, $B=s/(1-\lambda)$ and it is assumed $\lambda\geq s$. While the default choice in KR is $a=1$, we can  build up a new envelope by taking a union bound over $a \in \mathbb{N}\backslash\{0\}$:
\begin{theorem}\label{th-preordered-KR}
	Consider the pre-ordered setting of Section~\ref{sec:knockoffsetting} with $s\in (0,1]$. For all $\delta \in (0, 1)$ and $\lambda\in [s,1]$, the following is a $(1 - \delta)$-confidence envelope for the ordered path \eqref{pathpreordered} in the sense of \eqref{confenvelop}: 
	\begin{equation} \label{preorderedKR}
		\ol{\FDP}_k^{{\tiny \mbox{KR-U}}} := 1\wedge \min_{a\in \mathbb{N}\backslash\{0\}}\left\{ \frac{ \log(1/\delta_a)}{a \log(1+\frac{1 - \delta_a^{B/a}}{B})} \frac{a +\frac{s}{1-\lambda} \sum_{i=1}^{k} \ind{p_{\pi(i)} >  \lambda}}{  1\vee\sum_{i=1}^{k} \ind{p_{\pi(i)} \leq  s}}\right\}, \: k \geq 1,
	\end{equation}
	for $\delta_a=\delta/(\kappa a^2)$, $a\geq 1$, for $B=s/(1-\lambda)$, $\kappa=\pi^2/6$.
\end{theorem}

The envelope \eqref{preorderedKR} is less explicit than  \eqref{preorderedFreed} but has a better behavior in practice, as we will see in the numerical experiments of Section~\ref{sec:num}. %A reason is that the constants that are better calibrated in 

\subsection{Confidence bounds for LF and $m$-consistency}

Recall that the LF procedure \eqref{kchapeauLF} is the reference FDR-controlling procedure in this setting. 
Applying the above envelopes for the LF procedure gives the following confidence bounds.

\begin{corollary}\label{cor-LF}
In the pre-ordered setting of Section~\ref{sec:knockoffsetting} with a selection threshold $s\in (0,1]$, for any $\alpha,\delta\in (0,1)$, $\lambda\in [s,1]$ the following quantities are $(1-\delta)$-confidence bounds for the FDP of the LF procedure with parameters $s,\lambda$ at level $\alpha$:
\begin{align}
\ol{\FDP}^{\mbox{\tiny KR}}_{\alpha} &:= 1\wedge \left(\frac{ \log(1/\delta)}{ \log(1+\frac{1 - \delta^{B}}{B})} \left(\alpha + 1/( 1\vee\hat{r}_\alpha)\right)\right) ;
\label{LFpreorderedKR}\\
\ol{\FDP}^{{\tiny \mbox{Freed}}}_\alpha &:= 1\wedge \left(\alpha +  \Delta(\nu \hat{k}_\alpha)/(1\vee\hat{r}_\alpha)\right)\label{LFpreorderedFreedU}\\
\ol{\FDP}_\alpha^{{\tiny \mbox{KR-U}}} &:=   1\wedge  \min_{1\leq a\leq 1\vee\hat{r}_\alpha}\left\{ \frac{ \log(1/\delta_a)}{a \log(1+\frac{1 - \delta_a^{B/a}}{B})} \left(\alpha + a/( 1\vee\hat{r}_\alpha)\right)\right\}
\label{LFpreorderedKRU},
\end{align}
for $\nu=s(1+s/(1-\lambda))$,  $B=s/(1-\lambda)$, $\delta_a=\delta/(\kappa a^2)$, $a\geq 1$, $\kappa=\pi^2/6$, $\Delta(\cdot)$ defined in Theorem~\ref{th-preordered-Freed} and where $\hat{k}_\alpha$ is as in \eqref{kchapeauLF} and $\hat{r}_\alpha= \sum_{i=1}^{\hat{k}_\alpha} \ind{p_{\pi(i)} \leq  s}$ denotes the number of rejections of LF procedure at level $\alpha$.
In addition, these bounds are also valid uniformly in $\alpha\in (0,1)$ in the sense that
$$
 \P(\forall \alpha\in (0,1), \FDP(R_{\hat{k}_\alpha})\leq \ol{\FDP}_\alpha^{{\tiny \mbox{Method}}})\geq 1-\delta, \:\:\mbox{ for } \mbox{Method}\in \{{ \mbox{KR}},{ \mbox{Freed}},{ \mbox{KR-U}}\},
$$
and thus also when using a post hoc choice $\alpha=\hat{\alpha}$ of the level.
\end{corollary}

\begin{proof}
This is direct by applying \eqref{FDPKRknockoff} ($a=1$), \eqref{preorderedFreed} and \eqref{preorderedKR} to the rejection set $R_{\hat{k}_\alpha}$. %, and by noting that $ \frac{s}{1-\lambda}\frac{ 1 + \sum_{i=1}^{\hat{k}_\alpha} \ind{p_{\pi(i)} >  \lambda}}{1\vee \sum_{i=1}^{\hat{k}_\alpha} \ind{p_{\pi(i)} \leq s}}\leq \alpha$ by definition of $\hat{k}_\alpha$.
\end{proof}

Let us now study the consistency property \eqref{equ-consistency}.
It is apparent that KR is never LF $m$-consistent: namely, for all $m\geq 1$,
$$
 \ol{\FDP}^{{\tiny \mbox{KR}}}_\alpha\geq 1\wedge c \alpha,
$$
for some constant $c>1$.  By contrast, $\ol{\FDP}^{{\tiny \mbox{Freed}}}_\alpha$ is LF $m$-consistent if $\Delta(\nu m)/\hat{r}_\alpha$ tends to $0$ in probability, that is, $(m\log\log m)^{1/2}/\hat{r}_\alpha=o_P(1)$.
For $\ol{\FDP}^{{\tiny \mbox{KR-U}}}_\alpha$, we always have
$$
\ol{\FDP}^{{\tiny \mbox{KR-U}}}_\alpha\leq \frac{ \log(1/\delta_{\hat{a}})}{\hat{a} \log(1+\frac{1 - \delta_{\hat{a}}^{B/\hat{a}}}{B})} \left(\alpha + 1/( 1\vee\hat{r}_\alpha)^{1/2}\right)
$$
by considering $\hat{a}=\lfloor (1\vee \hat{r}_\alpha)^{1/2}\rfloor$. By Lemma~\ref{lem:KRconstant}, this  provides consistency \eqref{equ-consistency} as soon as $1/\hat{r}_\alpha=o_P(1)$. This is summarized in the next result.

{
\begin{proposition}\label{prop-consistency-preordered}
Let us consider any model sequence $\P^{(m)}$ in the pre-ordered setting and denote the rejection number of the LF procedure at level $\alpha$ by $\hat{r}_{\alpha}$. Then the following envelopes are 
LF $m$-consistent in the sense of \eqref{equ-consistency}: 
\begin{itemize}
\item $\ol{\FDP}_\alpha^{{\tiny \mbox{Freed}}}$ if $(m\log\log m)^{1/2}/\hat{r}_\alpha=o_P(1)$;
\item $\ol{\FDP}_\alpha^{{\tiny \mbox{KR-U}}} $ if $1/\hat{r}_\alpha=o_P(1)$.
\end{itemize} 
\end{proposition}
The latter result means that the LF procedure at level $\alpha$ should make enough rejections in order to provide $m$-consistency.
This is exemplified in a particular model in the next section.
}

\subsection{{LF $m$-consistency in the generalized VCT model}}\label{sec:powerpreordered}

{
  We provide here a model example where conditions of Proposition~\ref{prop-consistency-preordered} are satisfied. We consider the varying coefficient two-groups (VCT) model of \cite{lei2016power}, that we generalize to the possible sparse case. 

%  We introduce below a model generalizing the one of \cite{lei2016power} to the possibly sparse case.
  Here, without loss of generality we assume that the ordering $\pi$ is identity, that is, $\pi(i)=i$ for all $i\in \{1,\dots,m\}$.
{Below, with some abuse, the notation $\pi$ will be re-used to stick with the notation of \cite{lei2016power}.}
%\todo{sparse?}
\begin{definition}\label{def:VCT} {Let $m$ be a positive integer, $\beta$ (sparsity parameter) a real  in $[0,1)$, $F_0,F_1$ two c.d.f.s on $[0,1]$ with $F_0(t)\leq t$ for all $t\in [0,1]$, and $\pi:[0,\infty)\to [0,1)$ some measurable function (instantaneous signal probability function) with $\pi(0)>0$ and $\pi(x)=\pi(1)$
    for $x\geq 1$.}
  
The generalized VCT model of parameters $m,\pi,\beta,F_0,F_1$, denoted as $\P^{(m)}_{\pi,\beta,F_0,F_1}$, is the $p$-value mixture model where  $(p_k,H_k)\in [0,1]\times \{0,1\}$, $1\leq k\leq m$, are independent and generated as follows:
\begin{itemize}
\item the variables $H_k$, $1\leq k\leq m$, are independent and $\P(H_k=1)=\pi_m(k/m)$, $1\leq k\leq m$,
  with $\pi_m(x)=\pi(m^{\beta} x)$, $x\geq 0$; % where $\pi:[0,\infty)\to [0,1)$ is some measurable function (instantaneous signal probability function) with $\pi(0)>0$ and $\pi(x)=\pi(1)$ for $x\geq 1$ and $\beta\in [0,1)$ is a sparsity parameter.  
%null probabilities that are heterogeneous in time: $\P(H_k=1)=\pi(k/n)$ for some measurable non-negative function $\pi:[0,1]\to [0,1]$.
\item conditionally on $H_1,\dots,H_k$, the $p$-values $p_k$, $1\leq k\leq m$, are independent, 
  with
  %a marginal distribution super-uniform under the null: $p_k\:|\: H_k=0 \sim F_0$, $1\leq k\leq m$, where $F_0$ is a c.d.f. with $F_0(t)\leq t$ for all $t\in [0,1]$; %  \et{Should we assume $F_0(t)=t$?}
  {$p_k\:|\: \set{H_k=i} \sim F_i$, $1\leq k\leq m, i\in \set{0,1}$}.
  % where $F_1$ is some alternative c.d.f. %common continuous differentiable and strictly concave c.d.f. %\et{Should we assume $F_1$ ?}
\end{itemize}
We denote $\Pi(t):=t^{-1}\int_0^t \pi(s)ds$, with $\Pi(0)=\pi(0)$ and $$\Pi_m(t):=t^{-1}\int_0^t \pi_m(s)ds = t^{-1}\int_0^{t} \pi(m^{\beta}s)ds =  m^{-\beta}t^{-1}\int_0^{m^{\beta}t} \pi(s)ds = \Pi(m^{\beta}t)$$ the expected fraction of signal before time $mt$. We also let $\pi_1:=\Pi_m(1)=\int_0^1 \pi_m(s)ds = m^{-\beta} \Pi(1)$ the overall expected fraction of signal. We consider the asymptotic where $m$ tends to infinity and $F_0,F_1$ are fixed. 
\end{definition}
When $\beta=0$, $\pi_m$, $\Pi_m$ are  fixed and we recover the dense VTC model introduced in  \cite{lei2016power} (also noting that we are slightly more general because $F_0$ is possibly non-uniform and $F_1$ not concave). The above formulation can also handle the sparse case for which $\beta\in (0,1)$ and the probability to generate a signal is shrunk to $0$ by a factor $m^\beta$. For instance, if $\pi(1)=0$, the model only generates null {hypotheses and corresponding} $p$-values $p_{k+1},\dots,p_m$ for $k\geq m^{1-\beta}$.
}

{
%\et{More general than because  here. Important because encompasses the knockoff case. Paerhaps even more important: possibly sparse case allowed.} 
  We now analyze the asymptotic behavior of the number of rejections of the LF procedure. By following the same heuristic as in \cite{lei2016power}  (which is justified by a concentration argument),
  we have from \eqref{kchapeauLF} that for $k=\lfloor mt\rfloor$, 
\begin{align*}
\wh{\FDP}_k&=\frac{s}{1-\lambda}\frac{ 1 + \sum_{i=1}^{k} \ind{p_{i} >  \lambda}}{1\vee \sum_{i=1}^{k} \ind{p_{i} \leq s}}\\
&\approx \frac{s}{1-\lambda}\frac{ \left(\sum_{i=1}^{k}(1-\pi_m(i/m))\right) (1-F_0(\lambda)) +  \left(\sum_{i=1}^{k}\pi_m(i/m)\right) (1-F_1(\lambda))  }{ \left(\sum_{i=1}^{k}(1-\pi_m(i/m))\right) F_0(s) +  \left(\sum_{i=1}^{k}\pi_m(i/m)\right) F_1(s) } \\
&\approx \frac{ 1 +  \Pi_m(t) \left( \frac{1-F_1(\lambda)}{1-\lambda} -1\right) }{1 +  \Pi_m(t) \left( \frac{F_1(s)}{s} -1\right) } = \FDP^{\infty}(m^\beta t) ,
\end{align*}
by assuming $F_0(s)=s$, $F_0(\lambda)=\lambda$, $F_1(s)>s$, $F_1(\lambda)>\lambda$ and by letting
\begin{equation}\label{equ-FDPstarLF}
\FDP^{\infty}(t)= \frac{ 1 +  \Pi(t) \left( \frac{1-F_1(\lambda)}{1-\lambda} -1\right) }{1 +  \Pi(t) \left( \frac{F_1(s)}{s} -1\right) },\:\:\: t\geq 0 .
\end{equation}
By \eqref{kchapeauLF}, the quantity $\hat{k}_\alpha/m^{1-\beta}$ should be asymptotically close to  
\begin{equation}\label{equ-tstarLF}
t^*_\alpha=\max\{t\in[0,+\infty)\::\: \FDP^{\infty}(t)\leq \alpha\},
\end{equation}
with the convention $t^*_\alpha=+\infty$ if the set is not upper bounded. 
 We should however ensure that the latter set is not empty. For this, we let  
\begin{align}\label{alphalimit}
\underline{\alpha}&=\frac{ 1 +  \pi(0) \left( \frac{1-F_1(\lambda)}{1-\lambda} -1\right) }{1 +  \pi(0) \left( \frac{F_1(s)}{s} -1\right) }. %;\:\:\:
%\ol{\alpha}=\frac{ 1 +  \Pi(1) \left( \frac{1-F_1(\lambda)}{1-\lambda} -1\right) }{1 +  \Pi(1) \left( \frac{F_1(s)}{s} -1\right) }
\end{align}
Hence, $\hat{r}_\alpha= \sum_{i=1}^{\hat{k}_\alpha} \ind{p_{i} \leq  s}$, the number of rejections of the LF procedure, should be close to $\left(\sum_{i=1}^{\hat{k}_\alpha}(1-\pi_m(i/m))\right) F_0(s) +  \left(\sum_{i=1}^{\hat{k}_\alpha}\pi_m(i/m)\right) F_1(s)\gtrsim \hat{k}_\alpha s\approx m^{1-\beta} t^*_\alpha s$. This heuristic is formalized in the next result.
%Note that $\Pi(t)\approx (nt)^{-1}\sum_{k=1}^{nt} \pi(k/n)\approx (nt)^{-1}\sum_{k=1}^{nt} H_k$
\begin{theorem}\label{th:powerLF}
Consider a generalized VCT model $\P^{(m)}_{\pi,\beta,F_0,F_1}$ with parameters $\beta,\pi,F_0,F_1$ (see Definition~\ref{def:VCT}) and the LF procedure with parameter $s,\lambda$ (see \eqref{kchapeauLF}), with the assumptions:
\begin{itemize}
\item[(i)] $\Pi:t\in [0,\infty)\to \R_+$ is continuous, decreasing, and $L$-Lipschitz;
\item[(ii)]  $F_0(s)=s$, $F_0(\lambda)=\lambda$, $F_1(s)>s$, $F_1(\lambda)>\lambda$;
\item[(iii)] $\alpha>\underline{\alpha}$ where $\underline{\alpha}$ is defined by \eqref{alphalimit}.
%$\alpha\in (\underline{\alpha},\ol{\alpha})$ where $\underline{\alpha},\ol{\alpha}$ are defined by \eqref{alphalimit}.
\end{itemize}
Let $\alpha'=(\underline{\alpha}+\alpha)/2 \in (\underline{\alpha},\alpha)$, $t^*_{\alpha'}\in (0,+\infty]$ given by \eqref{equ-tstarLF},   $t^*_{m}=t^*_{\alpha'}\wedge m^{\beta}$ and let $a\geq 1$ be an integer $a\leq  m^{1-\beta}t^*_{m}$ such that $
r = \frac{4}{a^{1/4}} \left(\frac{1}{s}+\frac{1}{1-\lambda}\right)
$ is small enough to provide $r\leq (\alpha-\underline{\alpha})/4$. 
Then the number of rejections $\hat{r}_\alpha= \sum_{i=1}^{\hat{k}_\alpha} \ind{p_{i} \leq  s}$ of the LF procedure  \eqref{kchapeauLF} is such that
\begin{equation}\label{mainresultlemmapreordered}
\P^{(m)}_{\pi,\beta,F_0,F_1}( \hat{r}_\alpha < r_m^*) \leq 2(2+a^{1/2})e^{-2a^{1/2}} ,\:\:\:\:r^*_m=\lfloor m^{1-\beta}t^*_{m}\rfloor s/2.
 \end{equation}
In particular, choosing $a=1+\lfloor (\log m)^2\rfloor$, we have as $m$ grows to infinity,
$
m^{1-\beta}/\hat{r}_\alpha =O_P(1).
$
\end{theorem}
Theorem~\ref{th:powerLF} is proved in Section~\ref{proof:th:powerLF}.
Condition (ii) is more general that in  \cite{lei2016power} and allows to handle binary $p$-values, like in the `knockoffs' situation (for which $F_0$ and $F_1$ are not continuous).
The condition (iii) was overlooked in \cite{lei2016power}, but it is needed to ensure the existence of $t^*_\alpha$. %As an illustration, the condition $\alpha\in (\underline{\alpha},\ol{\alpha})$ 
It reads equivalently 
\begin{equation}\label{equ-instantproba}
%\frac{ 1 +  \pi(0) \left( \frac{1-F_1(\lambda)}{1-\lambda} -1\right) }{1 +  \pi(0) \left( \frac{F_1(s)}{s} -1\right) }<\alpha<\frac{ 1 +  \Pi(1) \left( \frac{1-F_1(\lambda)}{1-\lambda} -1\right) }{1 +  \Pi(1) \left( \frac{F_1(s)}{s} -1\right) }
\pi(0)>\frac{1-\alpha}{1- \frac{1-F_1(\lambda)}{1-\lambda} + \alpha \left( \frac{F_1(s)}{s} -1\right)},%>\Pi(1)
\end{equation}
which ensures that the probability to generate a false null hypothesis is sufficiently large at the beginning of the $p$-value sequence, with a minimum amplitude function of $F_1(s)$ and $F_1(\lambda)$.  
Note that in the `knockoffs' case where  $s=\lambda=1/2$, we have $\underline{\alpha}=\frac{ 1 -  \pi(0) M }{1 +  \pi(0) M }$, where $M=2F_1(1/2)-1>0$ can be interpreted as a `margin'. Hence, the critical level $\underline{\alpha}$ is decreasing in $\pi(0) M$. Hence, the setting is more favorable either when $\pi(0)$ increases, or when the margin $M$ increases.
}

{
\begin{corollary}\label{cor-consistency-preordered}
  Consider the sequence of generalized VCT models $(\P^{(m)}_{\pi,\beta,F_0,F_1},m\geq 1)$ defined above.  Assume that the parameters $\pi,\beta,F_0,F_1$ satisfy the assumptions of Theorem~\ref{th:powerLF}.
  % given in Appendix~\ref{sec:powerpreordered} (assuming in particular that $\alpha>\underline{\alpha}$ where $\underline{\alpha}$ is defined by \eqref{alphalimit}).
  Then the consistency \eqref{equ-consistency} holds for the sequence $(\P^{(m)}_{\pi,\beta,F_0,F_1},m\geq 1)$ and for any LF procedure using $\lambda\geq s$ in either of the two following cases:
\begin{itemize}
\item for the KR-U envelope \eqref{preorderedKR} and the corresponding bound \eqref{LFpreorderedKRU}.
\item for the Freedman envelope  \eqref{preorderedFreed}
		 and the corresponding bound \eqref{LFpreorderedFreedU} if either $\lambda=s$ or $\beta<1/2$; % for the version of LF procedure using $\lambda=s$; 
\end{itemize}
\end{corollary}
%\et{In fact, uniform Freedman envelope always consistent because $m^{1-\beta}\gg \log\log m$!}
\begin{proof}
This is a direct consequence of Theorem~\ref{th:powerLF} because $
m^{1-\beta}/\hat{r}_\alpha =O_P(1)
$
 in that context and $\hat{r}_\alpha$ is nondecreasing in $\alpha$. 
 To see why the Freedman envelope is consistent when $\lambda=s$, we note that in this case $\hat{k}_\alpha=\sum_{i=1}^{\hat{k}_\alpha} \ind{p_{\pi(i)} \leq  s}+\sum_{i=1}^{\hat{k}_\alpha} \ind{p_{\pi(i)} >  \lambda} \leq (1+\alpha s/(1-\lambda)) (1\vee \hat{r}_\alpha)$, hence the quantity $\Delta(\nu \hat{k}_\alpha)/(1\vee\hat{r}_\alpha)$ is $o_P(1)$ as $1/\hat{r}_\alpha=o_P(1)$.
%\et{Generalize Theorem~2 of \cite{lei2016power}.}
\end{proof}
%\todo{Paragraphe redondant avec ce qui précède, à mettre à jour (ou supprimer)}
%We would like to emphasize that the power analysis made in Appendix~\ref{sec:powerpreordered} provides new insights with respect to \cite{lei2016power}. First, it accommodates the sparse case for which the probability of generating an alternative is tending to zero as $m$ tends to infinity. Second, it introduces a new criticality-type assumption (see \eqref{alphalimit}), which was overlooked in \cite{lei2016power}, but is necessary to get a non zero power at the limit (even in the dense case). Finally, it allows to deal with binary $p$-values, which corresponds to the usual `knockoff' situation.
%}

\begin{remark}
Similarly to Section~\ref{sec:interptopk} in the top-$k$ setting, the bounds KR, Freedman and {KR-U} can be improved by performing the interpolation operation \eqref{equ-interp3} in the pre-ordered setting.
\end{remark}

\section{Results in the online case}

\subsection{Online setting}\label{sec:onlinesetting}

We consider an infinite stream of $p$-values $p_1,p_2,\dots$ testing null hypotheses $H_1,H_2,\dots$, respectively. In the online setting, these $p$-values come one at a time and a decision should be made at each time immediately and irrevocably, possibly on the basis of past decisions.

The decision at time $k$ is to reject $H_k$ if  $p_k\leq \alpha_k$ for some critical value $\alpha_k$ only depending on the past decisions. 
An online procedure is thus defined by a sequence of critical values $\mathcal{A}=(\alpha_k, k\geq 1)$, that is predictable in the following sense
$$
\alpha_{k+1} \in \mathcal{G}_k = %\mathfrak{S}
\sigma
( \ind{p_i\leq \alpha_i}, i\leq k), \:\:k\geq 1.
$$
A classical assumption is that each null $p$-value is super-uniform conditionally on past decisions, that is,
\begin{equation}\label{onlinepvaluehyp}
\P(p_k\leq x\:|\: \mathcal{G}_k) \leq x,\:\:k\in \cH_0,
\end{equation}
where $\cH_0=\{k\geq 1\:|\: H_k=0\}$. Condition \eqref{onlinepvaluehyp} is for instance satisfied if the $p$-values are all mutually independent and marginally super-uniform under the null.

For a {\it fixed} procedure $\mathcal{A}$, we consider the path 
\begin{equation}\label{pathonline}
R_k=\{1\leq i\leq k\::\: p_i\leq \alpha_i \},\:\: k\geq 1.
\end{equation}
We will also denote 
\begin{equation}\label{nbrejections}
R(k)=|R_k|=\sum_{i=1}^k \ind{p_i\leq \alpha_i},\:\: k\geq 1,
\end{equation}
 the number of rejections before time $k$ of the considered procedure.
 A typical procedure controlling the online FDR is the LORD procedure 
\begin{equation}\label{equalphaLORD}
\alpha_k = W_0 \gamma_k + (\alpha-W_0) \gamma_{k-\tau_1} + \alpha \sum_{j\geq 2} \gamma_{k-\tau_j},
\end{equation}
where $W_0\in [0,\alpha]$, each $\tau_j$ is the first time with $j$ rejections, $(\gamma_k)_{k}$ is a non-negative (`spending') sequence with $\sum_{k\geq 0}\gamma_k\leq 1$ and $\gamma_{k}=0$ for $k<0$. The latter has been extensively studied in the literature \citep{FosterStineAlphainvest,AharoniRossetGAI,javanmard2018online}, and further improved by \cite{ramdas2017online}.
Under independence of the $p$-values and super-uniformity of the $p$-values under the null, the LORD procedure controls the online FDR in the sense of 
$$
\sup_{k\geq 1} \E [\FDP(R_k)]\leq \alpha,
$$
see Theorem~2 (b) in \cite{ramdas2017online}. Here, we consider the different (and somehow more demanding) task of finding a bound on the realized online FDP, by deriving confidence envelopes \eqref{confenvelop}. Note that this will be investigated for any online procedure and not only for LORD, see Section~\ref{sec:confenveloponline}. Also, we will study the consistency of the envelope for any LORD-type procedure in Section~\ref{seconlineconsist}.

\subsection{New confidence envelopes}\label{sec:confenveloponline}

The first envelope is a consequence of the general result stated in Theorem~\ref{th-mainresult}.

\begin{theorem}\label{th:FDPcontrolonline}
In the online setting described in Section~\ref{sec:onlinesetting}, consider any online procedure $\mathcal{A}=(\alpha_k, k\geq 1)$ and assume \eqref{onlinepvaluehyp}. Then for any $\delta\in (0,1)$, the following is a $(1-\delta)$-confidence envelope for the path \eqref{pathonline} in the sense of  \eqref{confenvelop}:
\begin{equation}\label{onlineFreedU}
\ol{\FDP}_{\mathcal{A},k}^{{\tiny \mbox{Freed}}}:=1\wedge \frac{ \sum_{i=1}^k \alpha_i + \Delta\left(\sum_{i=1}^k  \alpha_i\right)
%2\sqrt{\varepsilon\left(\sum_{i=1}^k  \alpha_i\right)}\sqrt{ \sum_{i=1}^k  \alpha_i} +  \varepsilon\left(\sum_{i=1}^k  \alpha_i\right)
}{1\vee R(k)},\:\:k\geq 1,
\end{equation}
where $R(k)$ is given by \eqref{nbrejections}, 
	$\Delta(u)=2\sqrt{\varepsilon_u}\sqrt{{u \vee 1}} +  {\frac{1}{2}} \varepsilon_u$, 
	$\varepsilon_u=\log((1+\kappa)/\delta)+ 2\log\left(1+\log_2\left(u\vee 1\right)\right)$, $u>0$ and $\kappa= \pi^2/6$. 
%where $\varepsilon(x)= \log(3/\delta)+ 2\log\left(1+\left(\log_2\left(x\right)\right)_+\right)$, $x\geq 0$, and  $R(k)=\sum_{i\leq k} \ind{p_i\leq \alpha_i}$ is the number of rejections before time $k$.
\end{theorem}

\begin{proof}
We apply Theorem~\ref{th-mainresult} in the online setting for $\lambda=0$ (and further upper-bounding each term $\ind{p_{\pi(i)} >  0}$ by $1$), $\pi(k)=k$, because \eqref{condlemma1} is satisfied by \eqref{onlinepvaluehyp}. %This gives \eqref{boundlemma1} and thus \eqref{equFDPbaronline}.
\end{proof}

Next, the envelope of \cite{katsevich2020simultaneous} is as follows
\begin{equation}\label{onlineKR}
\ol{\FDP}^{\mbox{\tiny KR}}_{\mathcal{A},k} := 1\wedge\left(\frac{ \log(1/\delta)}{ a\log(1+\log(1/\delta)/a)} \frac{a+\sum_{i=1}^k \alpha_i}{1\vee R(k)}\right),
\end{equation}
for some parameter $a>0$ to choose. While the default choice in \cite{katsevich2020simultaneous} is $a=1$, applying a union w.r.t. $a\in \mathbb{N}\backslash\{0\}$ provides the following result.

\begin{theorem}\label{th:FDPcontrolonline}
In the online setting described in Section~\ref{sec:onlinesetting} such that \eqref{onlinepvaluehyp} is satisfied, and for any online procedure $\mathcal{A}=(\alpha_k, k\geq 1)$, for any $\delta\in (0,1)$, the following is a $(1-\delta)$-confidence envelope for the path \eqref{pathonline} in the sense of  \eqref{confenvelop}:
\begin{equation}\label{onlineKRU}
\ol{\FDP}_{\mathcal{A},k}^{{\tiny \mbox{KR-U}}}:= 1\wedge \min_{a\in \mathbb{N}\backslash\{0\}} \left\{\frac{ \log(1/\delta_a)}{ a\log(1+\log(1/\delta_a)/a)} \frac{a+\sum_{i=1}^k \alpha_i}{1\vee R(k)}\right\},\:\:k\geq 1,
\end{equation}
where $R(k)$ is given by \eqref{nbrejections}, $\delta_a=\delta/(\kappa a^2)$, $a\geq 1$, for $\kappa=\pi^2/6$.
\end{theorem}

\begin{remark}
{Note that in the online setting, the obtained guarantee \eqref{confenvelop} is not uniform in the procedure $\mathcal{A}$ (in contrast with the envelopes in top-$k$ and preordered cases which were uniform in $k$ and thus also in the cut-off procedure).}
\end{remark}

\subsection{Confidence envelope for LORD-type procedures and $m$-consistency}\label{seconlineconsist}

We now turn to the special case of online procedures satisfying the following condition:
\begin{equation}\label{equ-condmFDR}
 \sum_{i=1}^{k}  \alpha_i \leq \alpha (1\vee R(k)),\:\:\:k\geq 1.
\end{equation}
Classically, this condition is sufficient to control the online FDR (if the $p$-values are independent and under an additional monotonicity assumption), see Theorem~2 (b) in \cite{ramdas2017online}. In particular, it is satisfied by LORD \eqref{equalphaLORD}.

\begin{corollary}\label{cor-LORD}
In the online setting described in Section~\ref{sec:onlinesetting}, consider any online procedure $\mathcal{A}=(\alpha_k, k\geq 1)$, satisfying \eqref{equ-condmFDR} for some $\alpha\in (0,1)$, and assume \eqref{onlinepvaluehyp}. Then for any $\delta\in (0,1)$,  the following quantities are $(1-\delta)$-confidence bounds for the FDP of the procedure: for all $k\geq 1$,
\begin{align}
\ol{\FDP}^{\mbox{\tiny KR}}_{\alpha,k} &:= 1\wedge\left( \frac{ \log(1/\delta)}{ \log(1+\log(1/\delta))} \left(\alpha + 1/(1\vee R(k)\right)\right) ;
\label{onlineKRLORD}\\
\ol{\FDP}^{{\tiny \mbox{Freed}}}_{\alpha,k} &:=  1\wedge\left(\alpha +\frac{  \Delta\left( \alpha(1\vee R(k))\right)
}{1\vee R(k)}\right),\:\:k\geq 1;
\label{onlineFreedU}\\
\ol{\FDP}^{{\tiny \mbox{KR-U}}}_{\alpha,k} &:=  1\wedge  \min_{a\geq 1} \left\{\frac{ \log(1/\delta_a)}{ a\log(1+\log(1/\delta_a)/a)} \left(\alpha +a/(1\vee R(k)\right)\right\}
\label{onlineKRU},
\end{align}
for $\delta_a=\delta/(\kappa a^2)$, $a\geq 1$, $\kappa=\pi^2/6$, $\Delta(\cdot)$ defined in Theorem~\ref{th:FDPcontrolonline} and where $R(k)$ is given by \eqref{nbrejections}.
\end{corollary}

\begin{proof}
This is direct by applying \eqref{onlineKR} ($a=1$), \eqref{onlineFreedU} and \eqref{onlineKRU} and by using the inequality \eqref{equ-condmFDR} in the corresponding bound. \end{proof}

Let us now consider these bounds for the LORD procedure \eqref{equalphaLORD}, and study the LORD $m$-consistency property for each envelope $\ol{\FDP}_{\alpha,k},k\geq 1$: 
for all $\epsilon>0$,
 \begin{equation}\label{equ-consistency-online}
 \lim_{k \to \infty} \P\left(  \ol{\FDP}_{\alpha,k}-\alpha\geq \epsilon\right) = 0.
 \end{equation}
 where the asymptotics is when the time $k$  tends to infinity.

Clearly, we have 
$\ol{\FDP}^{\mbox{\tiny KR}}_{\alpha}\geq 1\wedge (c \alpha)$ for all $k\geq 1$, where $c>1$ is a constant.  Hence, the KR  envelope is not LORD $m$-consistent. By contrast, it is apparent that both the Freedman envelope and the uniform KR envelope are LORD $m$-consistent provided that $1/R(k)=o_P(1)$ as $k$ tends to infinity (consider $a=\sqrt{1\vee R(k)}$ and use Lemma~\ref{lem:KRconstant} for the KR-U envelope). This is summarized in the next result.

{
\begin{proposition}\label{prop-consistency-online}
Let us consider any online model $\P$ for which \eqref{onlinepvaluehyp} is satisfied and the LORD procedure at level $\alpha$ which rejects $R(k)$ nulls at time $k$, then the  envelopes $(\ol{\FDP}_{\alpha,k}^{{\tiny \mbox{Freed}}},k\geq 1)$ and $(\ol{\FDP}_{\alpha,k}^{{\tiny \mbox{KR-U}}},k\geq 1)$ are 
LORD $m$-consistent in the sense of \eqref{equ-consistency-online} provided that  $1/R(k)=o_P(1)$ as $k$ tends to infinity.
\end{proposition}
The latter result means that the LORD procedure at level $\alpha$ should make enough rejections in order $m$-consistency to be guaranteed.
This condition is met in classical online models, as the next section shows.
}

\subsection{{LORD $m$-consistency in a vanilla online model}}\label{sec:onlinepower}

{
\begin{definition}
  The online one-sided Gaussian mixture model of parameters $\pi_1,F_1$, denoted by $\P_{\pi_1,F_1}$, is given by the i.i.d. $p$-value stream $(p_k,H_k)\in [0,1]\times \{0,1\}$, $k\geq 1$, %which is i.i.d.
  with
\begin{itemize}
\item  $\P(H_k=1)=\pi_1$ for some {\it fixed} $\pi_1\in (0,1)$;
\item $p$-values are uniform under the null: $p_k\:|\: H_k=0 \sim U(0,1)$;
\item $p$-values have the same alternative distribution: $p_k\:|\: H_k=1 \sim F_1$, where $F_1$ is the c.d.f. corresponding to the one-sided Gaussian problem, that is, $F_1(x)=\bar \Phi(\bar \Phi^{-1}(x)-\mu)$, $x\in [0,1]$, for some $\mu>0$.%some common c.d.f. that is continuously differentiable and strictly concave.
\end{itemize}
\end{definition}
Here, we make no sparsity assumption: $\pi_1$ is assumed to be constant across time. This will ensure that the online procedure maintains a chance to make discoveries even when the time grows to infinity.   
\begin{theorem}\label{thonlinepower}
Consider the one-sided Gaussian online mixture model and the LORD procedure with $W_0\in (0,\alpha)$ and a spending sequence $\gamma_k=\frac{1}{k (\log(k))^\gamma}$,  $\gamma>1$. Then its rejection number $R(k)$ at time $k$ satisfies: for all $a\in (0,1)$, $k\geq 1$,
\begin{align}\label{equonlinepower}
\P(R(k)< k^{1-a})\leq  c k^{-a},
\end{align}
where $c$ is some constant only depending on $\alpha$ ,$W_0$, $\gamma$, $\mu$ and  $\pi_1$.
In particular, $k^{1-a}/R(k)=O_P(1)$ when $k$ tends to infinity, {for any $a>0$}.
\end{theorem}
Theorem~\ref{thonlinepower} is proved in Section~\ref{proof:thonlinepower}.
\begin{corollary}\label{cor-consistency-online}
Consider the online one-sided  Gaussian mixture model $\P_{\pi_1,F_1}$ defined above and the LORD procedure with $W_0\in (0,\alpha)$ and a spending sequence $\gamma_k=\frac{1}{k (\log(k))^\gamma}$, $k\geq 1$ for $\gamma>1$. 
Then both the Freedman envelope \eqref{onlineFreedU} and the uniform KR envelope  \eqref{onlineKRU} are consistent in the sense of \eqref{equ-consistency-online} for the model $\P_{\pi_1,F_1}$. 
\end{corollary}
\begin{proof}
This is a direct consequence of Theorem~\ref{thonlinepower}, which provides  that $k^{1/2}/R(k)=O_{\P_{\pi_1,F_1}}(1)$ when $k$ tends to infinity.
\end{proof}
\begin{remark}
Similarly to Section~\ref{sec:interptopk} in the top-$k$ setting, the bounds KR, Freedman and KR-U can be improved by performing the interpolation operation \eqref{equ-interp3} in the online setting.
\end{remark}
}
\section{Numerical experiments}\label{sec:num}
 
% \et{Add coverages?}
In this section, we illustrate our findings by conducting numerical experiments\footnote{All our numerical experiments are reproducible from the code provided in the repository \href{https://github.com/iqm15/ConsistentFDP}{https://github.com/iqm15/ConsistentFDP}.}
 in each of the considered settings: top-$k$, pre-ordered and online. 
Throughout the experiments, the default value for $\delta$ is $0.25$ and the default number of replications to evaluate each FDP bound is $1000$. 
 
 % \et{Discuss how to minimize for KR-U?}
 
 \subsection{Top-$k$}

\begin{center}
\begin{figure}[h!]
\begin{tabular}{cc}
$\alpha=0.05$ & $\alpha=0.1$\\
\includegraphics[scale=0.12]{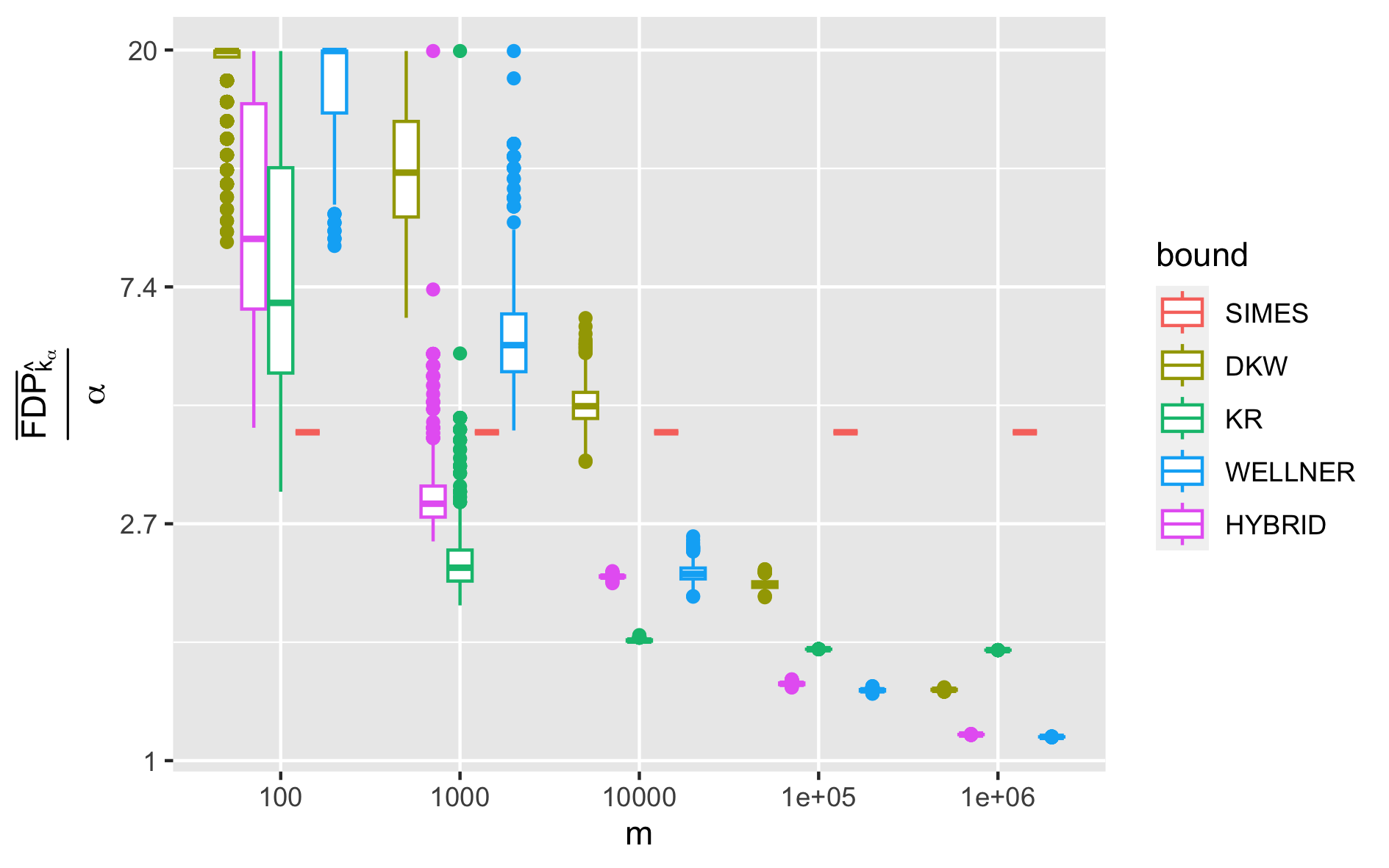}&\includegraphics[scale=0.12]{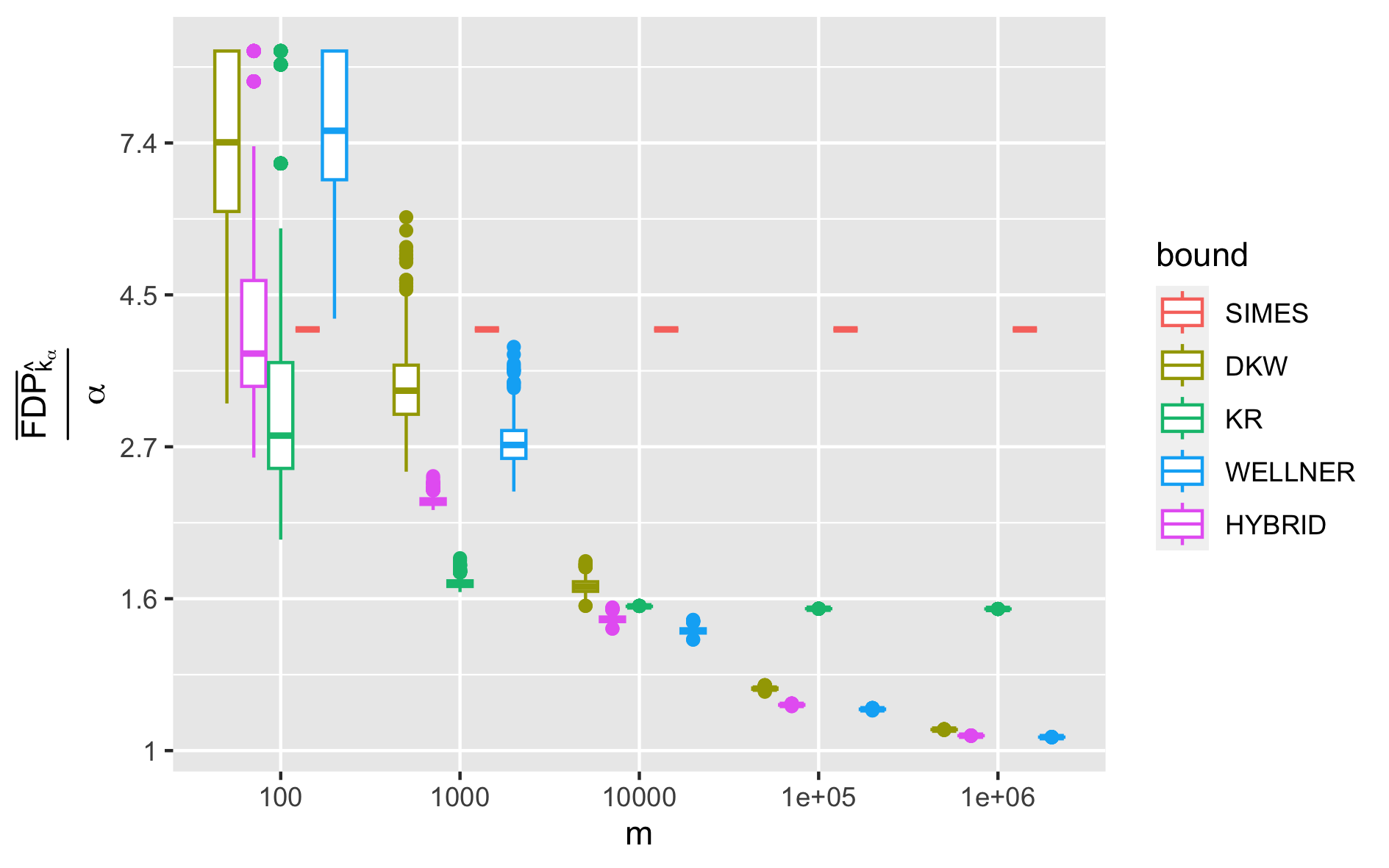}\\
$\alpha=0.15$ & $\alpha=0.2$\\
\includegraphics[scale=0.12]{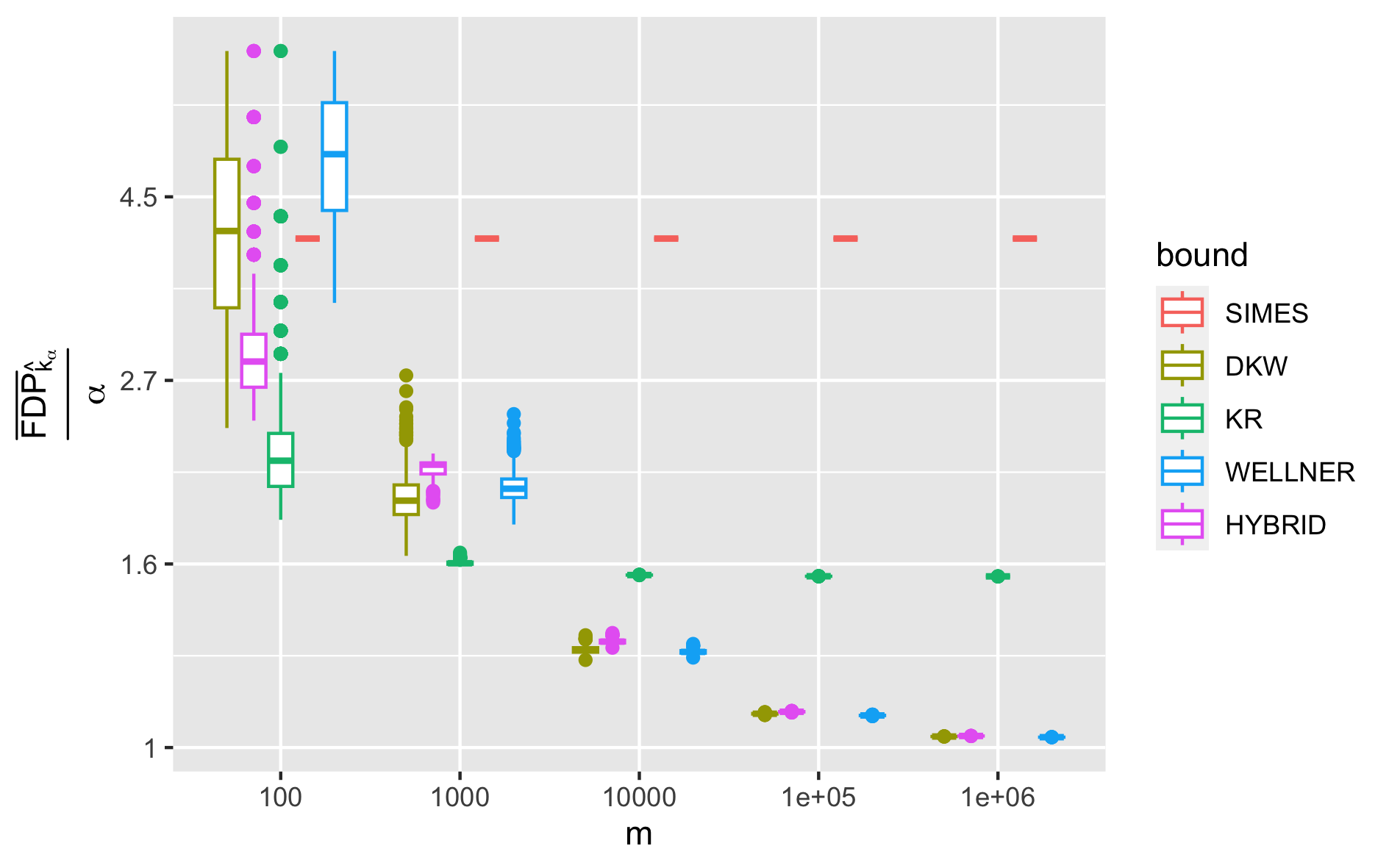}&\includegraphics[scale=0.12]{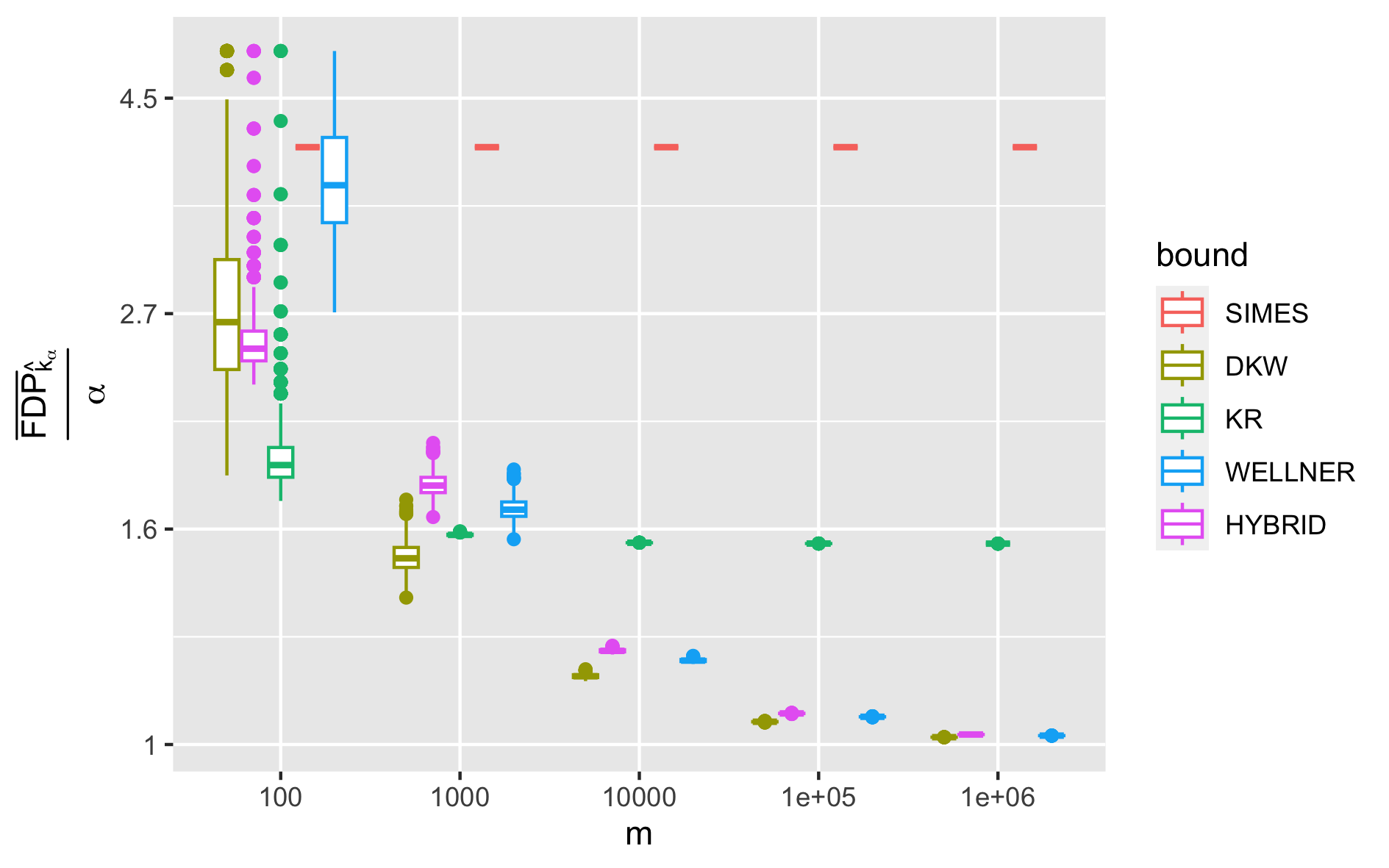}%\vspace{-5mm}
\end{tabular}
\caption{Top-$k$ dense case ($\pi_0=0.5$, $\mu=1.5$).\label{fig:topkdense}}
\end{figure}
\end{center}

Here, we consider the top-$k$ setting of Section~\ref{sec:topksetting}, for alternative $p$-values distributed as $F_1(x)=\ol{\Phi}(\ol{\Phi}^{-1}(x)-\mu)$ (one-sided Gaussian location model), and for different values of $\mu$ and of $\pi_0$. To investigate the consistency property, we take $m$ varying in the range $\{10^{i},2\leq i\leq 6\}$, and we consider the FDP bounds $ \ol{\FDP}^{{\tiny \mbox{Simes}}}_\alpha$ \eqref{BHtopkSimes}, 
$\ol{\FDP}^{{\tiny \mbox{DKW}}}_\alpha$ \eqref{BHtopkDKW}, $\ol{\FDP}^{{\tiny \mbox{KR}}}_\alpha$ \eqref{BHtopkKR}, $\ol{\FDP}_\alpha^{{\tiny \mbox{Well}}}$ \eqref{BHtopkWellner} for  $\alpha\in\{ 0.05,0.1,0.15,0.2\}$. We also add for comparison the hybrid bound 
$$
\ol{\FDP}^{{\tiny \mbox{Hybrid}}}_{\alpha,\delta} := \min\left(\ol{\FDP}^{{\tiny \mbox{KR}}}_{\alpha,\delta/2},\ol{\FDP}_{\alpha,\delta/2}^{{\tiny \mbox{Well}}} \right),
$$
 which also provides the correct coverage while being close to the best between the Wellner and KR bounds.
 
 Figure~\ref{fig:topkdense} displays boxplots of the different FDP bounds in the dense case for which $\pi_0=1/2$, $\mu=1.5$. When $m$ gets large, we clearly see the inconsistency of the bounds Simes, KR and the consistency of the bounds Wellner, Hybrid, DKW, which corroborates the theoretical findings (Corollary~\ref{cor-consistencytopk}). In sparser scenarios,  Figure~\ref{fig:topksparse} shows that the consistency is less obvious for the Wellner and Hybrid bounds and gets violated for the DKW bound when $m_1\propto m^{0.55}$, as predicted from Corollary~\ref{cor-consistencytopk} (regime $\beta\geq 1/2$).  Overall, the new bounds are expected to be better as the number of rejections gets larger and KR bounds remain better when the number of rejections is expected to be small. The hybrid bound hence might be a good compromise for a practical use.
 
 The adaptive versions of the bounds (Section~\ref{sec:adaptive}) are displayed on Figure~\ref{fig:topkadapt}. By comparing the left and the right panels, we see that the uniform improvement can be significant, especially for the Wellner and DKW bounds. By contrast, the improvement for KR is slightly worse. This can be explained from Figure~\ref{fig:topkadaptpi0chap}, that evaluates the quality of the different $\pi_0$ estimators. DKW, which is close to an optimized Storey-estimator, is the best, followed closely by the Wellner estimator.

\begin{center}
\begin{figure}[!]
\begin{tabular}{cc}
\includegraphics[scale=0.12]{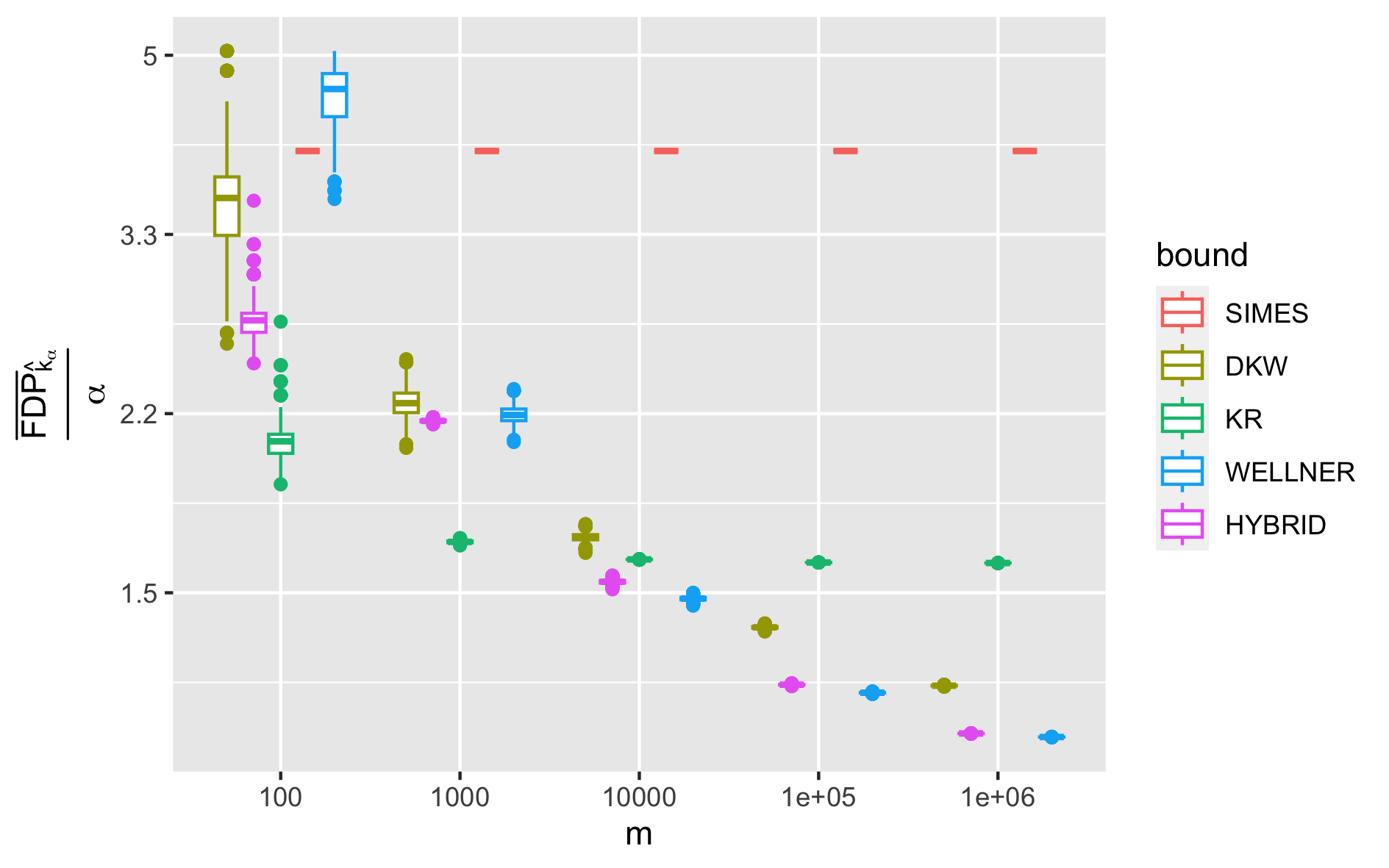}&\includegraphics[scale=0.12]{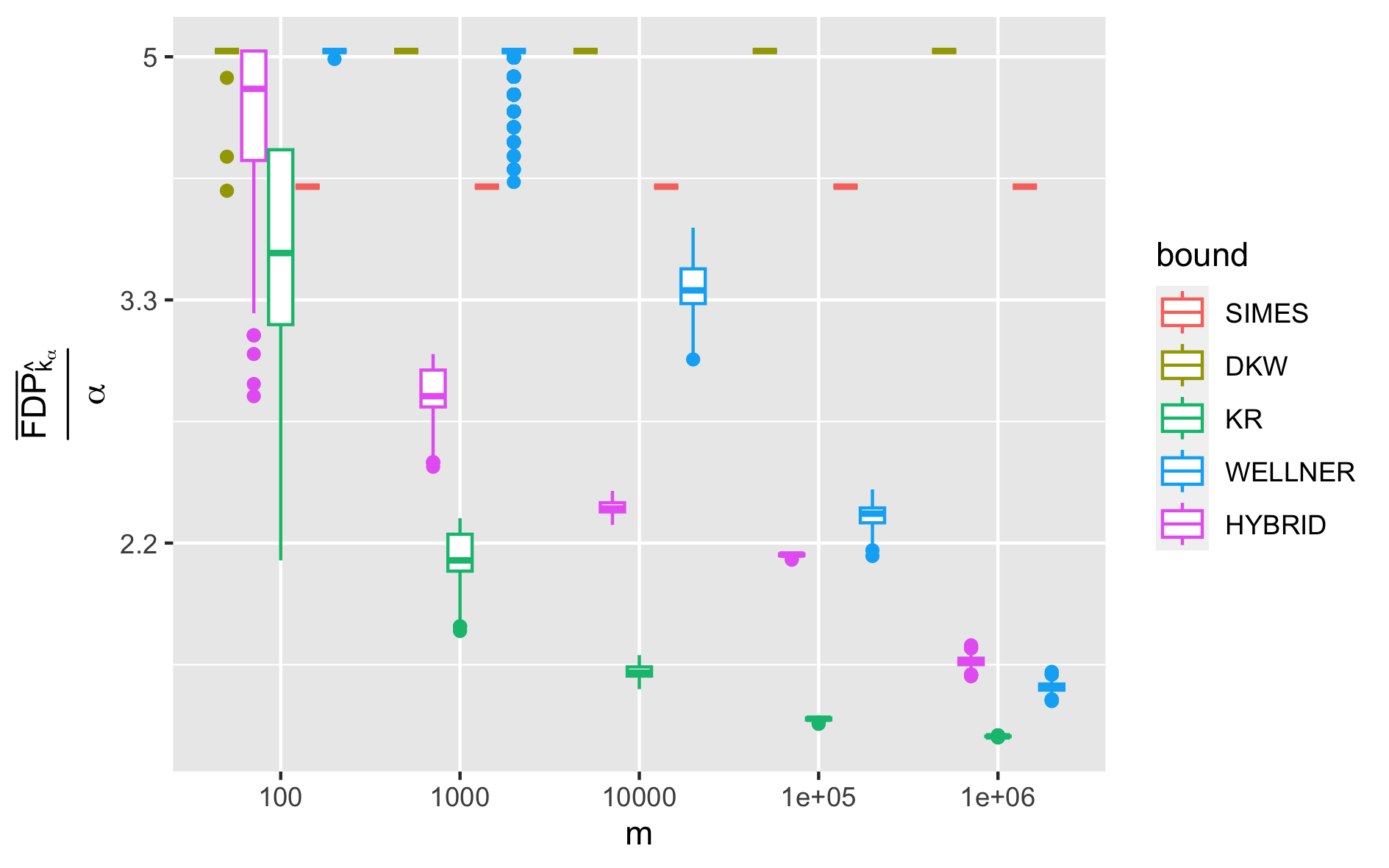}
%\vspace{-5mm}
\end{tabular}
\caption{Top-$k$ sparse case $\pi_0=1- 0.5 m^{-0.25}$, $\mu=\sqrt{2\log(m)}$  (left) $\pi_0= 1- 0.5 m^{-0.55}$, $\mu={10}$ (larger than $\sqrt{2\log (10^{6})}$)  (right), $\alpha=0.2$. \label{fig:topksparse}} %\et{right picture to update with the appropriate strong signal to make Wellner ok and DKW fail}
\end{figure}
\end{center}

\begin{center}
\begin{figure}[h!]
\begin{tabular}{cc}
Non adaptive & Adaptive \\
\includegraphics[scale=0.12]{topk/dense_alpha_0_2.png}&\includegraphics[scale=0.12]{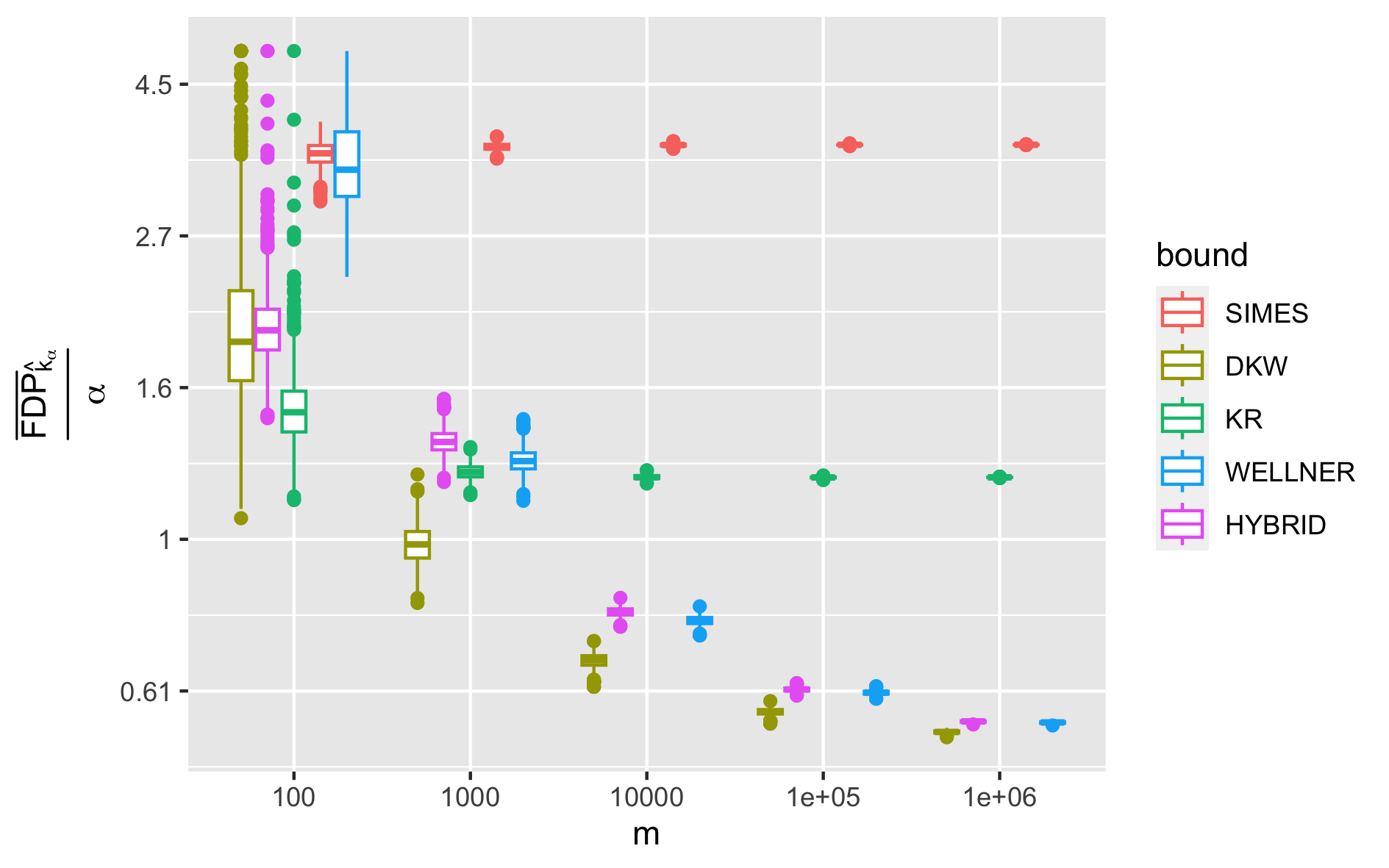}%\vspace{-5mm}
\end{tabular}
\caption{Top-$k$ dense case with nonadaptive bounds (left) and adaptive bounds (right) ($\pi_0=0.5$, $\alpha=0.2$). \label{fig:topkadapt}}
\end{figure}
\end{center}

\begin{center}
\begin{figure}[h!]
\begin{tabular}{cc}
\includegraphics[scale=0.2]{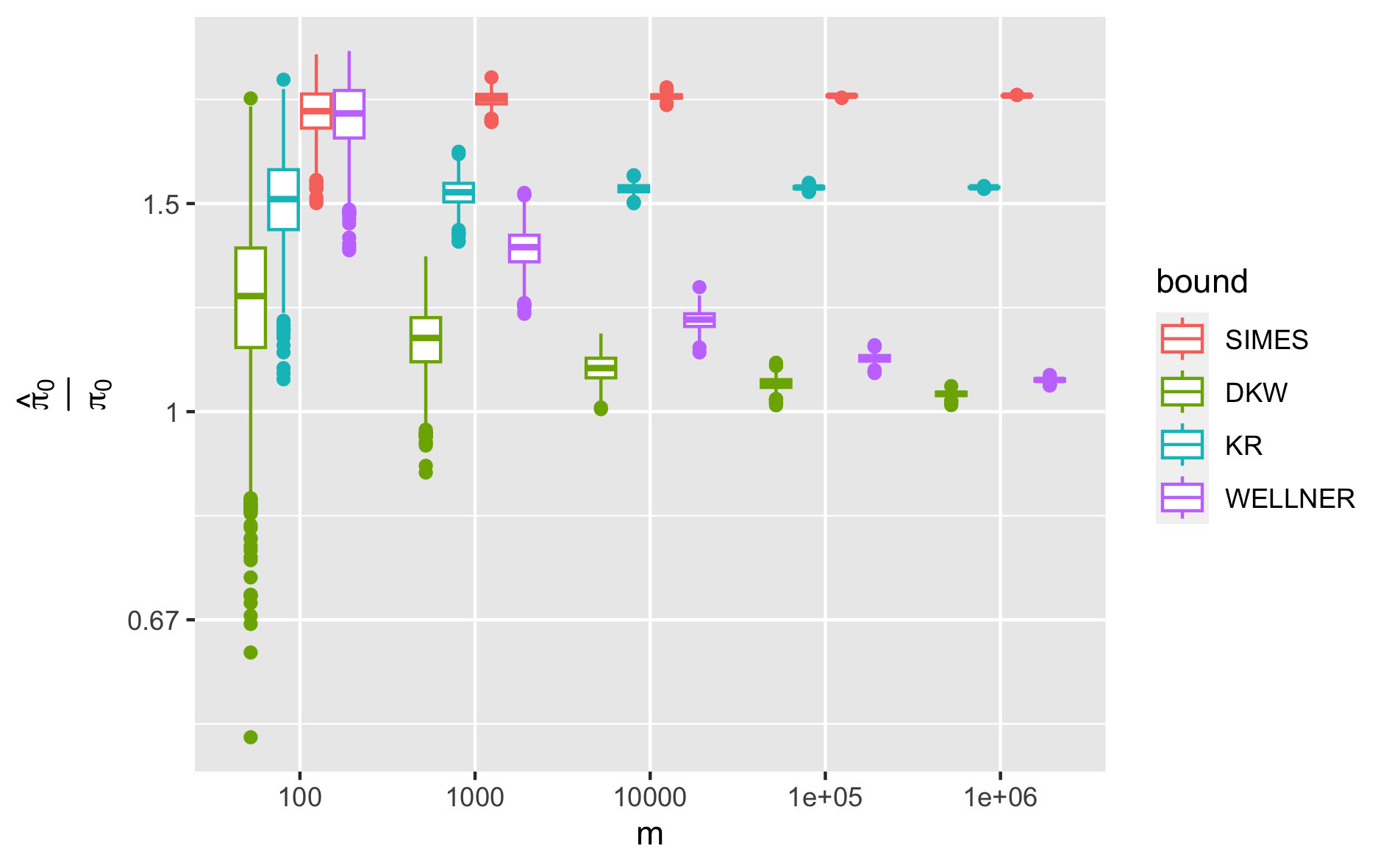}
\end{tabular}
\caption{Boxplots of the estimators $\hat{\pi}_0$ in the top-$k$ dense case ($\pi_0=0.5$, $\alpha=0.2$). %\vspace{-1cm}
\label{fig:topkadaptpi0chap}}
\end{figure}
\end{center}
 
  \begin{remark}
For clarity,  the bounds are displayed without the interpolation improvement \eqref{equ-interp3} (for top-$k$ and preordered). The figures are reproduced together with the interpolated bounds in Appendix~\ref{sec:addexp} for completeness. 
In a nutshell, the interpolation operation improves significantly the bounds mainly when they are not very sharp (typically small $m$ or very sparse scenarios). Hence, while it can be useful in practice, interpolation does not seem particularly relevant to study the consistency phenomenon. 
\end{remark}

 \subsection{Pre-ordered}

\begin{center}
\begin{figure}[h!]
\begin{tabular}{cc}
$\alpha=0.05$ & $\alpha=0.1$\\
\includegraphics[scale=0.12]{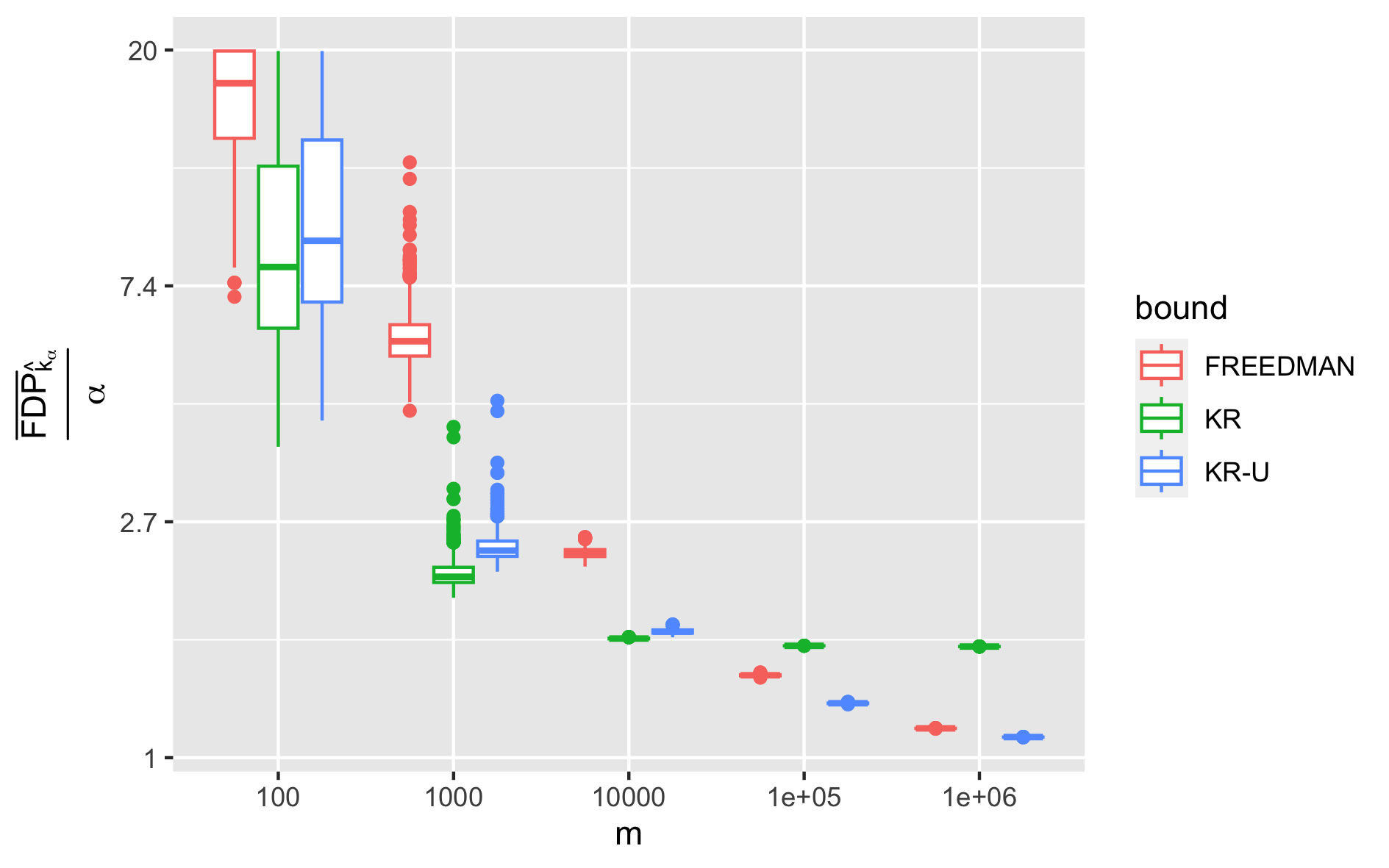}&\includegraphics[scale=0.12]{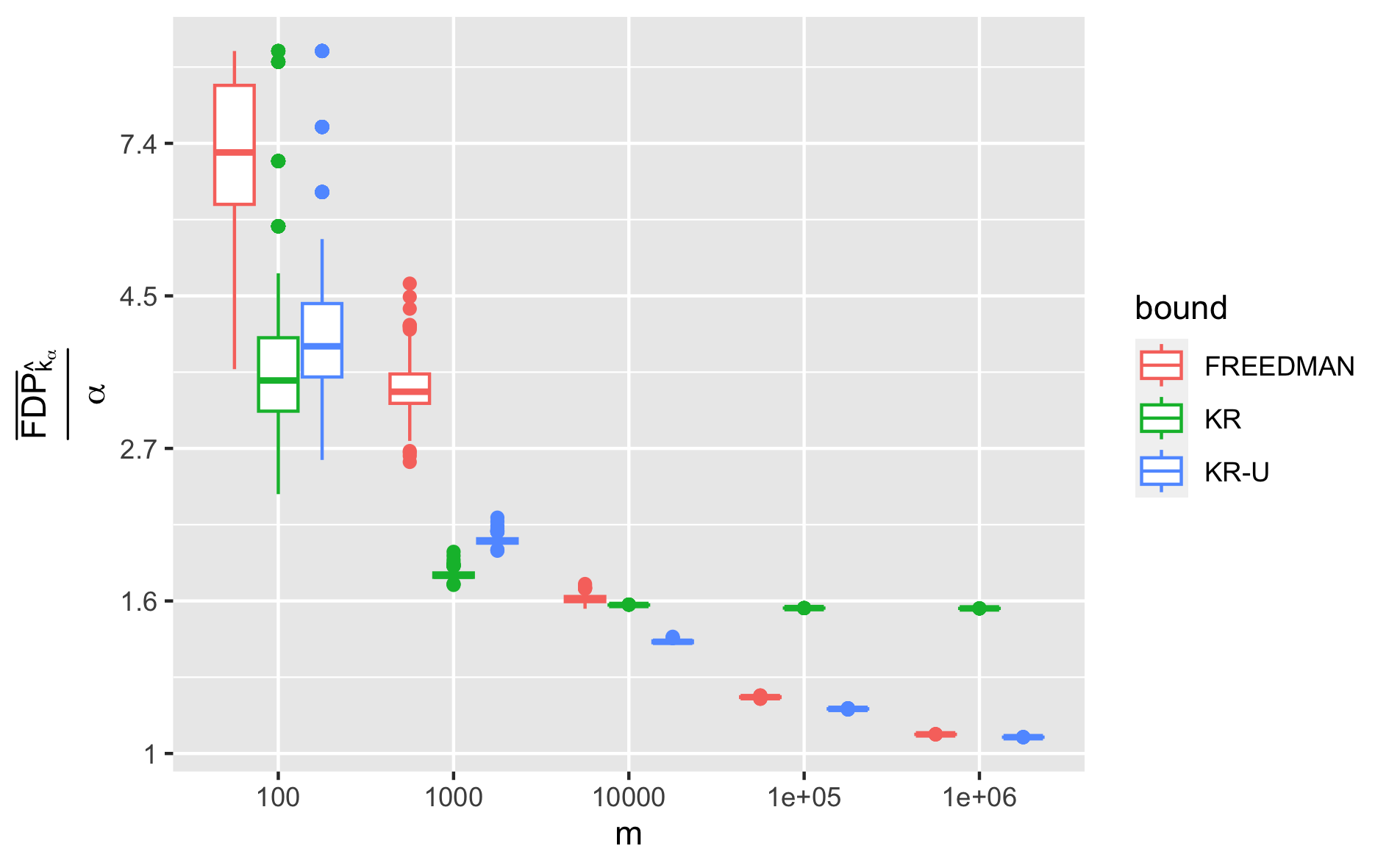}\\
$\alpha=0.15$ & $\alpha=0.2$\\
\includegraphics[scale=0.12]{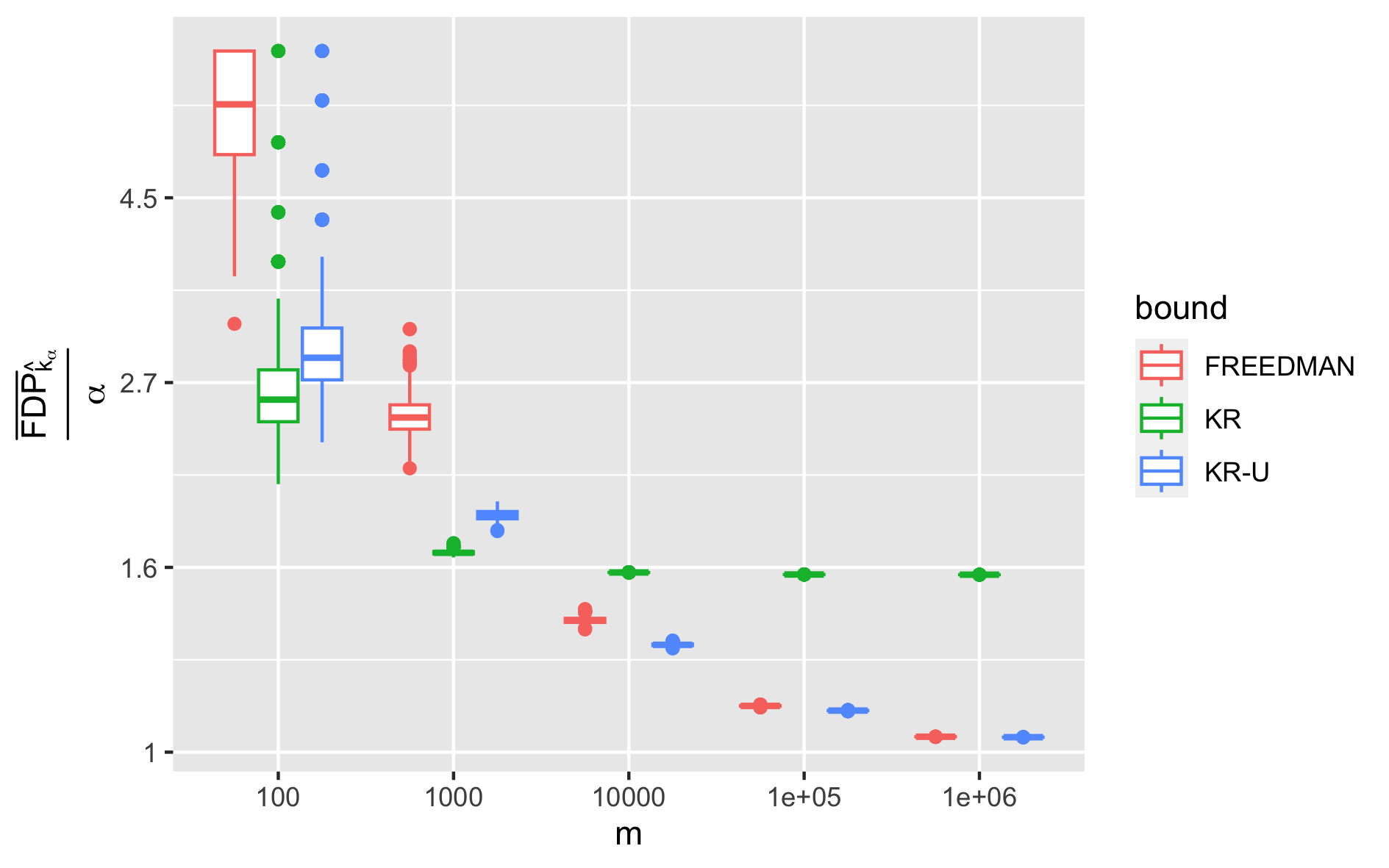}&\includegraphics[scale=0.12]{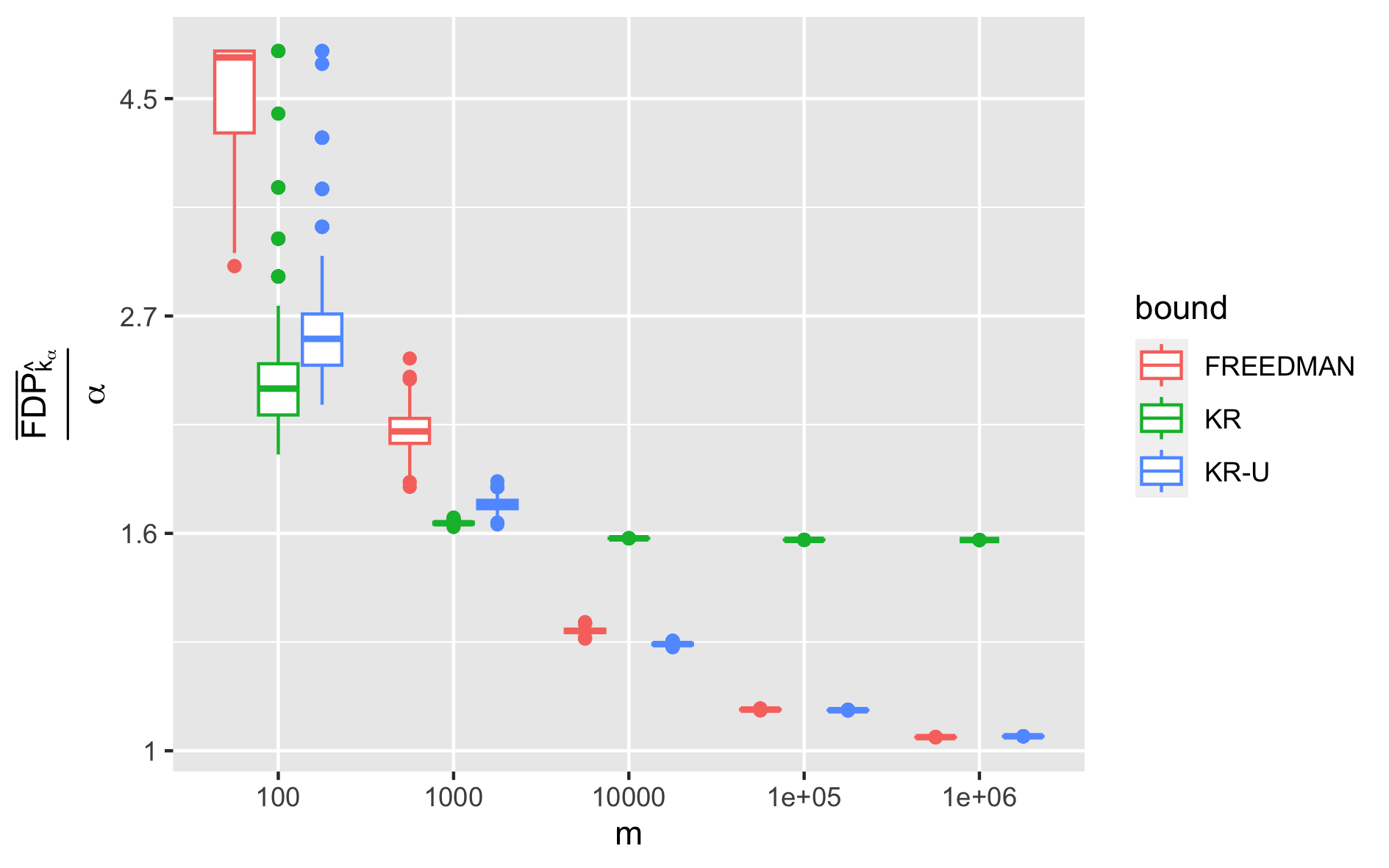}
\end{tabular}
\caption{Preordered dense ($\beta=0$) LF setting with LF procedure ($s=0.1 \alpha$, $\lambda=0.5$).\label{fig:preorderLFdense}}
\end{figure}
\end{center}
 
 We consider data generated as in the pre-ordered model presented in Section~\ref{sec:knockoffsetting} and more specifically as in the VCT model of Section~\ref{sec:powerpreordered}. The trueness/falseness of null hypotheses are generated independently, and the probability of generating an alternative is decreasing with the position $1\leq k\leq m$, and is given by $\pi(m^{\beta-1} k)$, where $\pi:[0,\infty)\to [0,1)$ is some function (see below) and $\beta\in[0,1)$ is the sparsity parameter. Once the process of true/false nulls is given, the $p$-values are generated according to either:  
 \begin{itemize}
 \item LF setting: $\pi(t)=\pi_{1} e^{-bt} \frac{b}{1 - e^{-b}}$, $t\geq 0$, so that $\Pi(1)=\pi_1$. Here $\pi_{1}$ is equal to $0.4$ and $b$, measuring the quality of the prior ordering, is equal to $2$. %{(with $\pi_{1} = 0.4$ the proportion of signal and $b=2$ a normalizing parameter that determines the quality of the prior ordering)}
 In addition, the alternative $p$-values are one-sided Gaussian with $\mu=1.5$. Note that this is the setting considered in the numerical experiments of \cite{lei2016power}. 
 \item Knockoff setting: $\pi(t)=1/2 + (0 \vee 1/2 (\frac{z-t}{z-1}))$, $t\geq 0$, with $z>1$ a parameter that determines how slowly the probability of observing signal deteriorates, taken equal to $30$. Then, the binary $p$-values are as follows: under the null $p_i=1/2$ or $1$ with equal probability. Under the alternative, $p_i=1/2$ with probability $0.9$ and $p_i=1$ otherwise.
\end{itemize}
In both settings, the dense (resp. sparse) case refers to the sparsity parameter value $\beta=0$ (resp. $\beta=0.25$).

 We consider the bounds $\ol{\FDP}^{\mbox{\tiny KR}}_{\alpha}$ \eqref{LFpreorderedKR}, 
$\ol{\FDP}^{{\tiny \mbox{Freed}}}_\alpha$ \eqref{LFpreorderedFreedU} and
$\ol{\FDP}_\alpha^{{\tiny \mbox{KR-U}}} $ \eqref{LFpreorderedKRU} 
 for the LF procedure across different values of $(\lambda,s)\in \{(1/2,0.1\alpha),(1/2,1/2)\}$, $m\in \{10^{i},2\leq i\leq 6\}$, and $\alpha\in\{ 0.05,0.1,0.15,0.2\}$. The procedure LF with   $(\lambda,s)=(1/2,1/2)$ is referred to as the Barber and Cand\`es (BC) procedure.

\begin{center}
\begin{figure}[h!]
\begin{tabular}{cc}
$\alpha=0.05$ & $\alpha=0.1$\\
\includegraphics[scale=0.12]{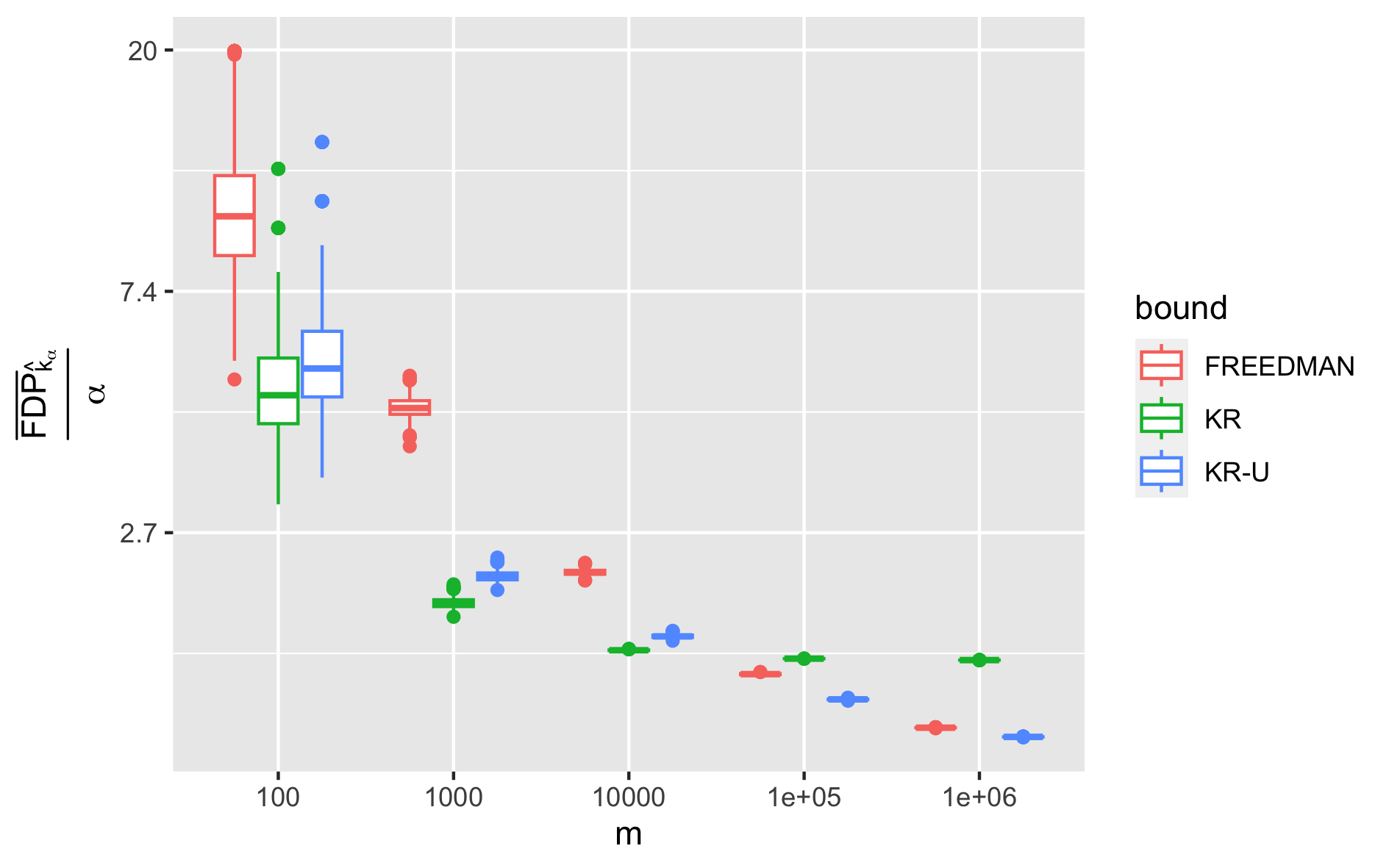}&\includegraphics[scale=0.12]{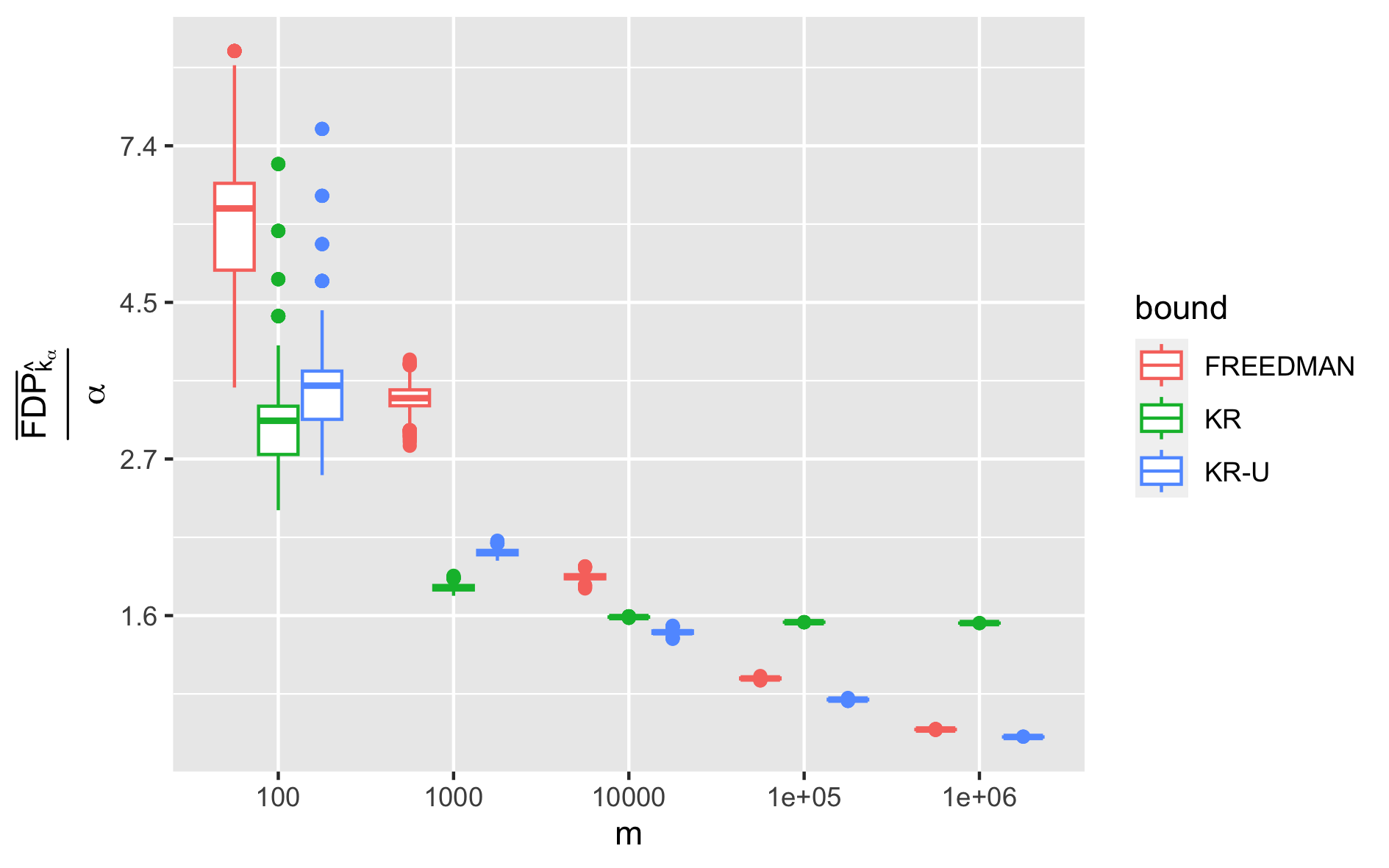}\\
$\alpha=0.15$ & $\alpha=0.2$\\
\includegraphics[scale=0.12]{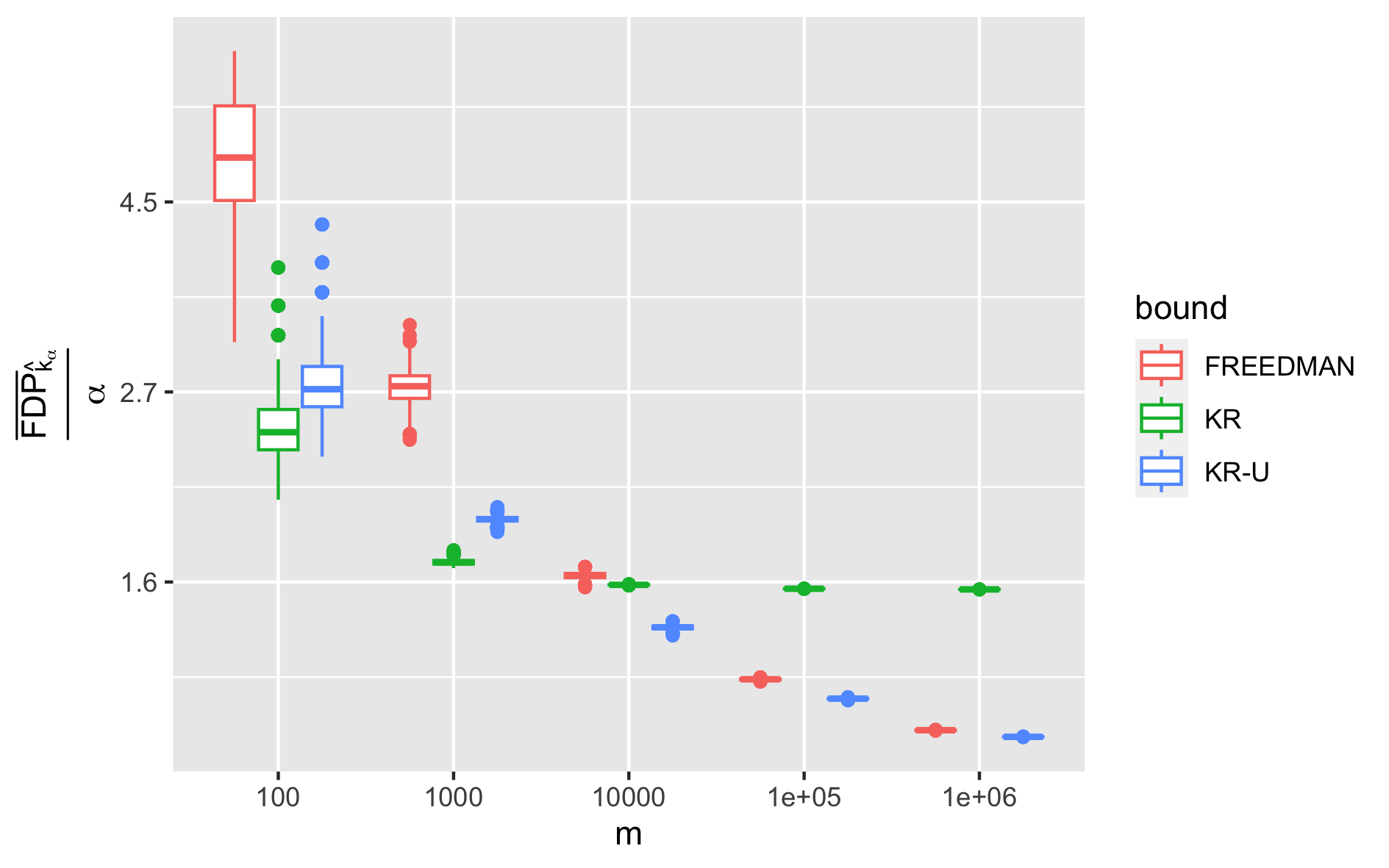}&\includegraphics[scale=0.12]{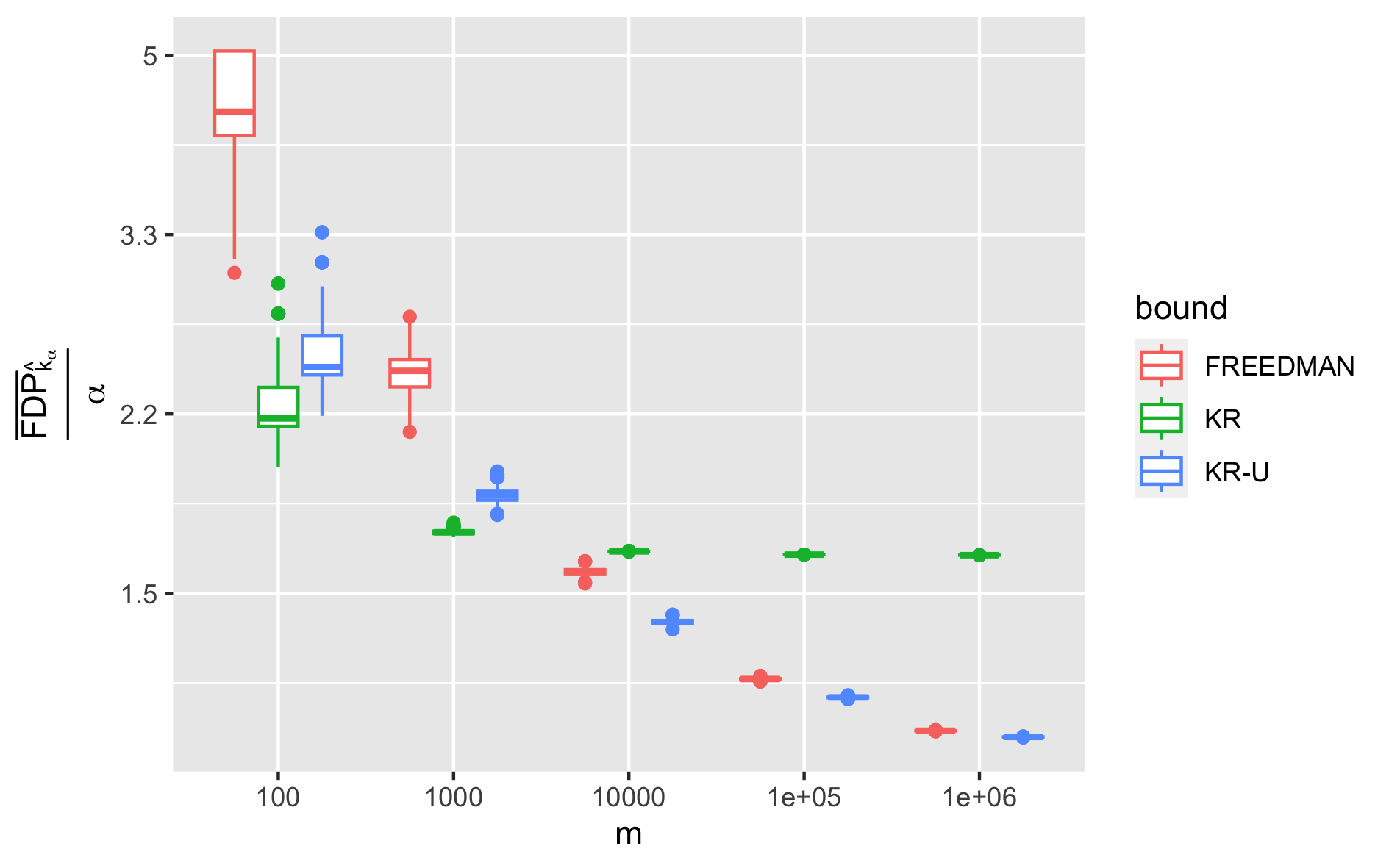}
\end{tabular}
\caption{Preordered sparse ($\beta=0.25$) LF setting with LF procedure ($s=0.1 \alpha$, $\lambda=0.5$). %\et{What is $\beta$ here? Try out $\beta=0.25$ where Freedman ok and $\beta=0.55$ where Freedman not ok? (similarly to the top-$k$ case)}
\label{fig:preorderLFsparse}}
\end{figure}
\end{center}

\begin{center}
\begin{figure}[h!]
\begin{tabular}{cc}
%$\alpha=0.05$ &
% $\alpha=0.1$\\
%\includegraphics[scale=0.12]{preorder/knockoff_alpha_0_05.png}&
%\includegraphics[scale=0.12]{preorder/knockoff_alpha_0_1.png}\\
$\alpha=0.15$ & $\alpha=0.2$\\
\includegraphics[scale=0.12]{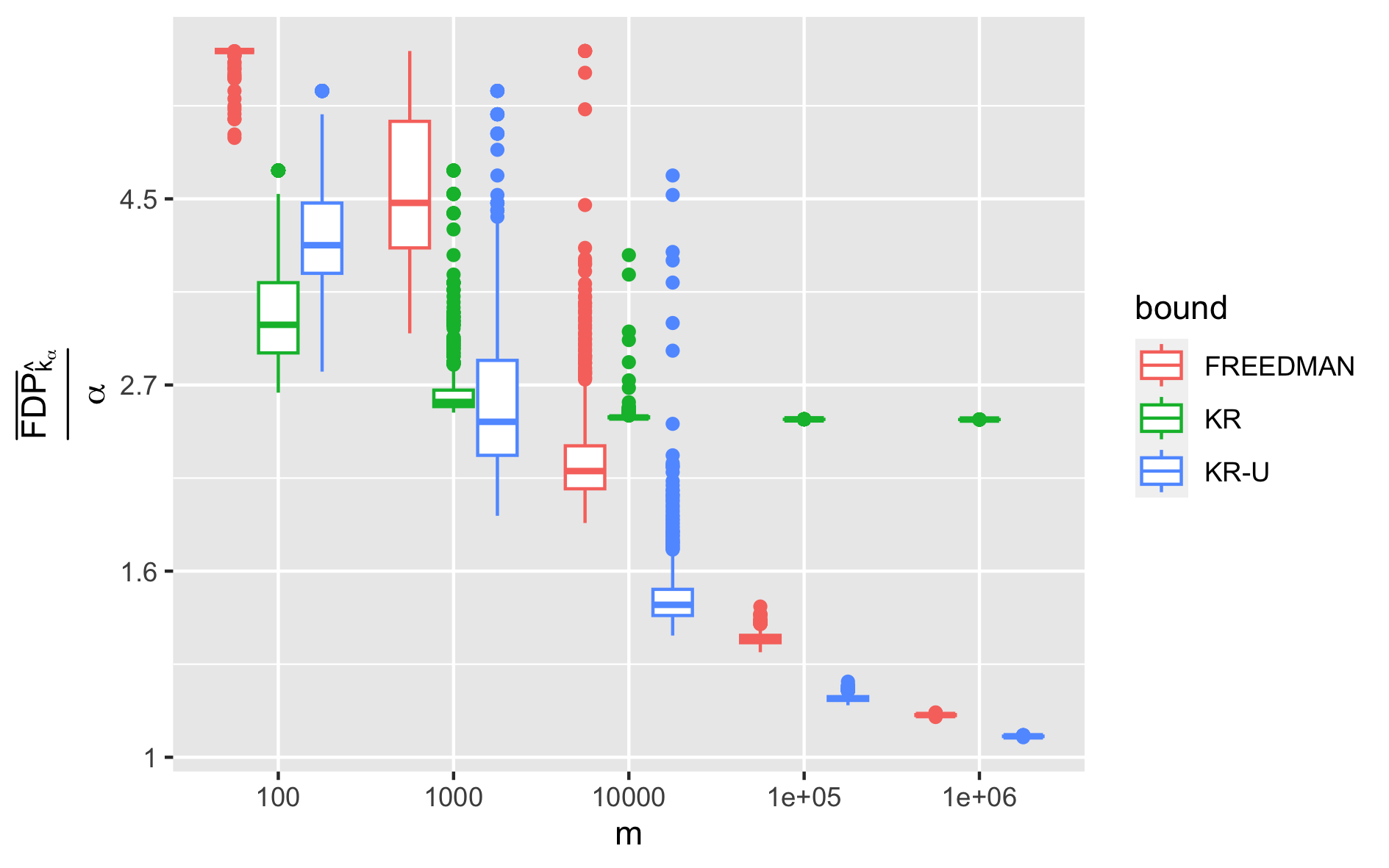}&\includegraphics[scale=0.12]{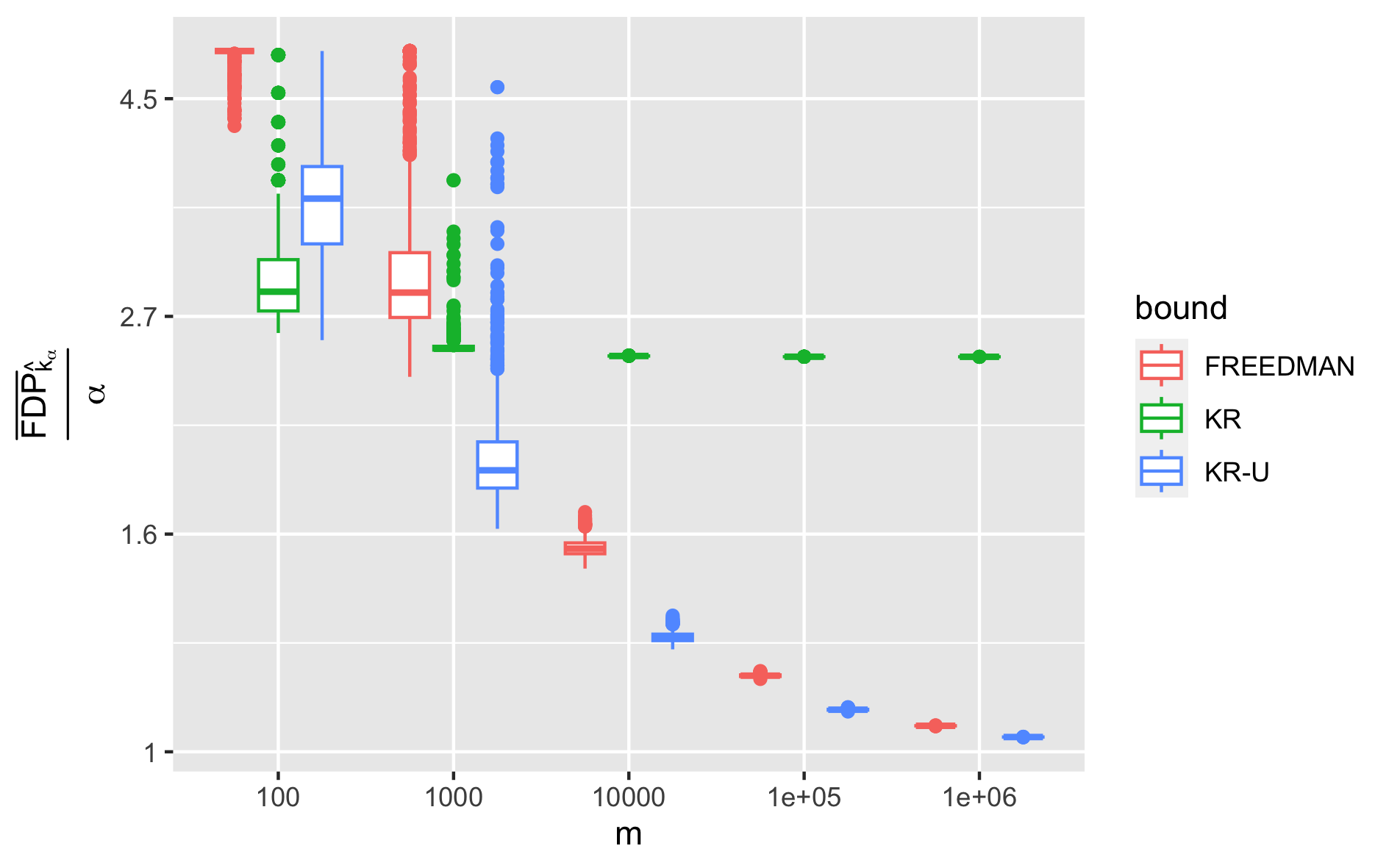}
\end{tabular}
\caption{Pre-ordered dense ($\beta=0$) knockoff setting with BC procedure (i.e., LF procedure with $s=\lambda=0.5$). \label{fig:preorderknockoff}}
\end{figure}
\end{center}
% (with ) and knockoff procedure (LF procedure with $\lambda=1/2$ and $s=1/2$)
Figure~\ref{fig:preorderLFdense} displays the boxplots of these FDP bounds for  the LF procedure with $(\lambda,s)=(1/2,0.1\alpha)$ in the LF setting with  $\beta=0$ (dense case). It is apparent that KR is not consistent, while the new bounds Freedman and KR-U are. Also, the bound KR-U is overall the best, losing almost nothing w.r.t. KR when the number of rejections is very small (say $m=100$ and $\alpha=0.05$ or $0.1$) and making a very significant improvement otherwise. Similar conclusions hold for the case of BC procedure, see Figure~\ref{fig:preorderknockoff}. 
Next, to stick with a very common scenario, we also investigate the sparse situation where the fraction of signal is small in the data, see Figures~\ref{fig:preorderLFsparse}~and~\ref{fig:preorderknockoffsparse}. As expected, while the conclusion is qualitatively the same, the rejection number gets smaller so that the consistency is  reached for largest values of $m$ (i.e., the convergence is `slowed down').

\begin{center}
\begin{figure}[!]
\begin{tabular}{cc}
%$\alpha=0.05$ &
% $\alpha=0.1$\\
%\includegraphics[scale=0.12]{preorder/knockoff_alpha_0_05.png}&
%\includegraphics[scale=0.12]{preorder/knockoff_alpha_0_1.png}\\
$\alpha=0.15$ & $\alpha=0.2$\\
\includegraphics[scale=0.12]{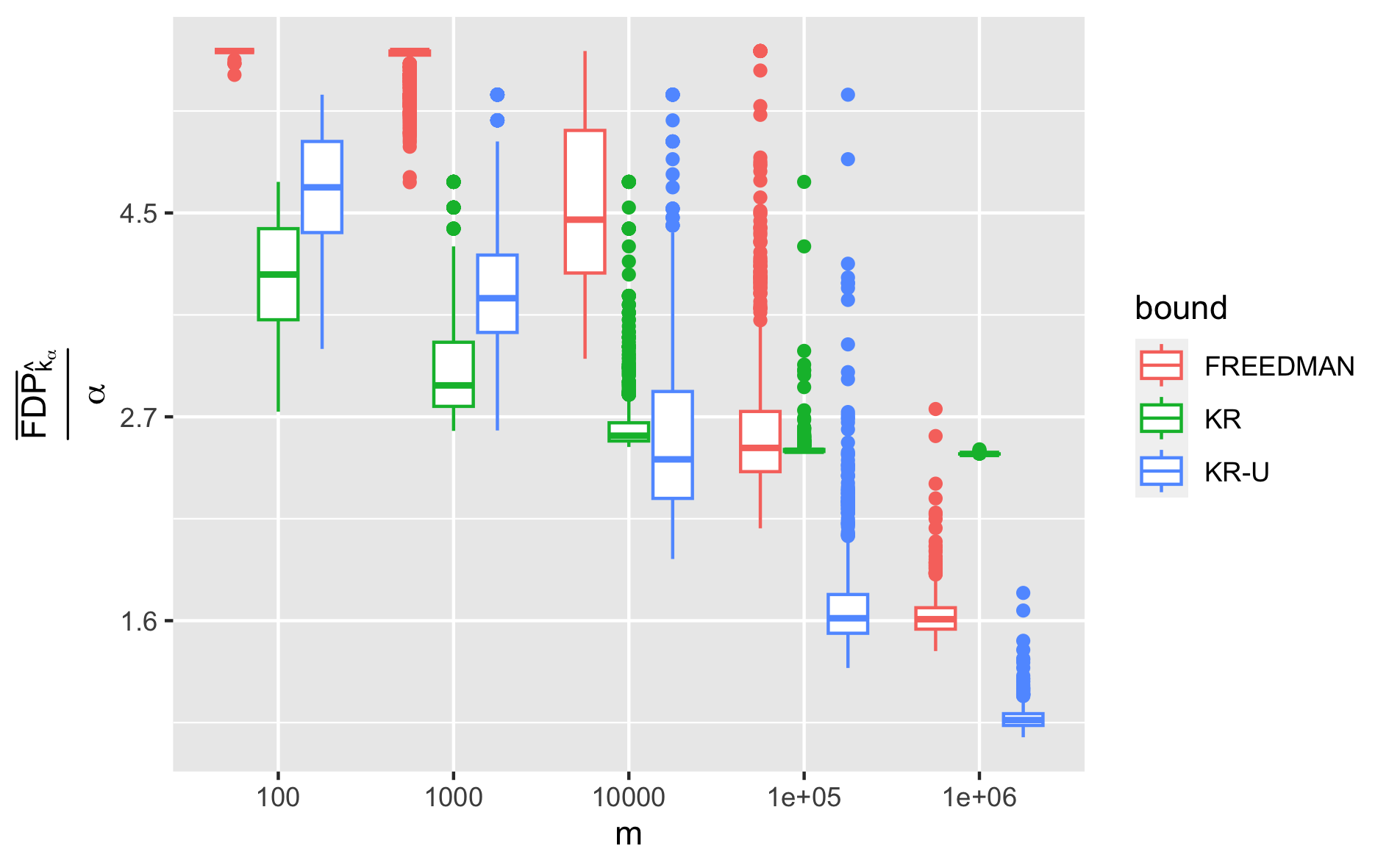}&\includegraphics[scale=0.12]{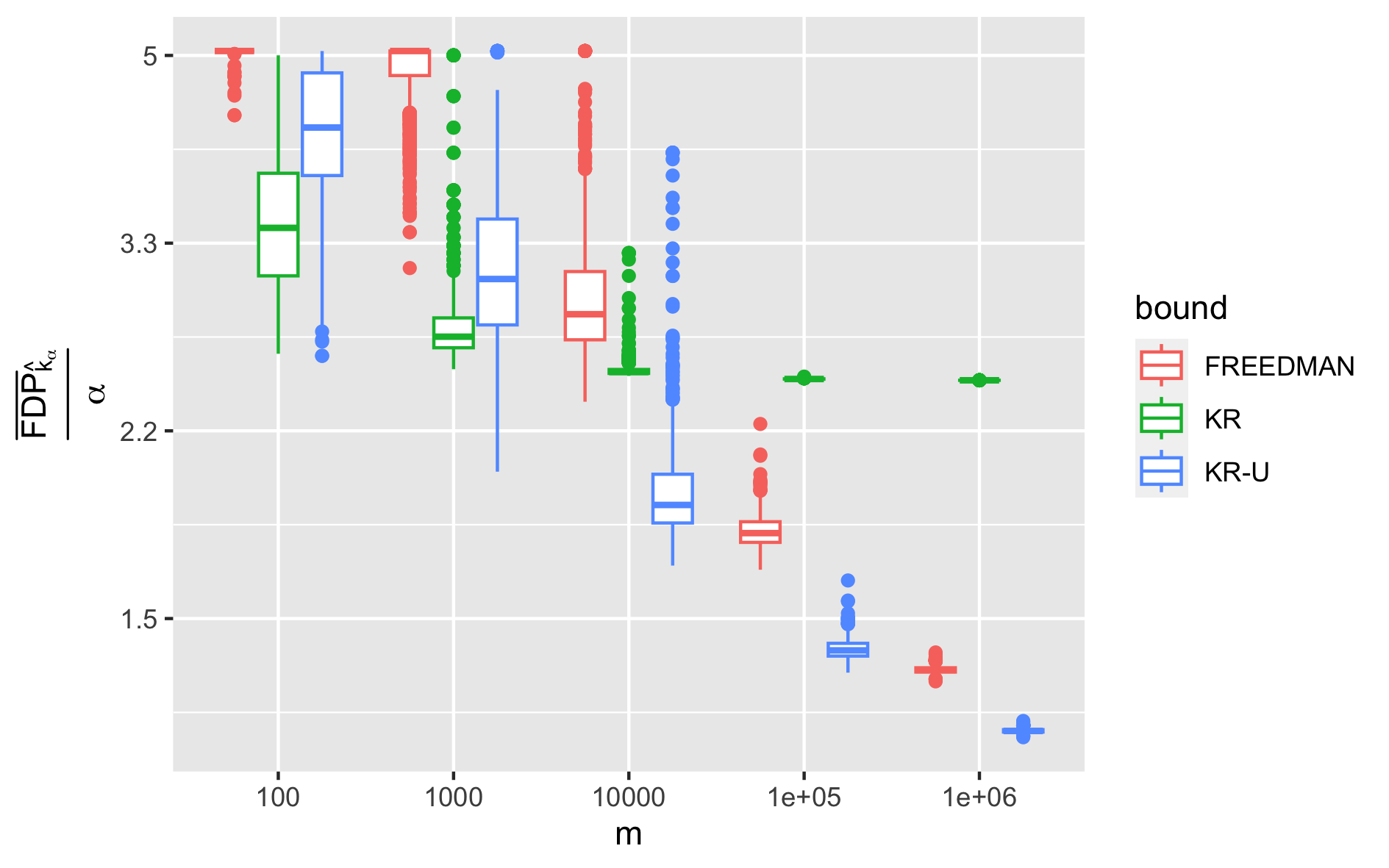}
\end{tabular}
\caption{Pre-ordered sparse ($\beta=0.25$) knockoff setting with BC procedure (i.e., LF procedure with $s=\lambda=0.5$). \label{fig:preorderknockoffsparse}}
\end{figure}
\end{center}

 \subsection{Online}
 
 We now consider the online case, by applying our method to the real data example coming  from the International Mice Phenotyping Consortium (IMPC) \citep{munozfuentes}, which is a consortium interested in the genotype effect on the phenotype. This data is collected in an online fashion for each gene of interest and is classically used in online detection works (see \cite{ramdas2017online} and references therein). 
 
 Figure~\ref{fig:onlineimpc} displays the FDP time-wise envelopes $k\mapsto \ol{\FDP}^{\mbox{\tiny KR}}_{\alpha,k}$ 
\eqref{onlineKRLORD}, $k\mapsto \ol{\FDP}^{{\tiny \mbox{Freed}}}_{\alpha,k}$ \eqref{onlineFreedU} and 
$k\mapsto \ol{\FDP}^{{\tiny \mbox{KR-U}}}_{\alpha,k}$ \eqref{onlineKRU}, for the LORD procedure \eqref{equalphaLORD} ($W_0=\alpha/2$ with the spending sequence $\gamma_k= k^{-1.6}$, $k\geq 1$). 
 As we can see, the Freedman and KR-U envelopes both tend to the nominal level $\alpha$, as opposed to the KR  envelope, which is well expected from the consistency theory. 
In addition, KR-U seems to outperform the Freedman envelope and while KR is (slightly) better than KR-U in the initial segment of the process $(k<300)$, we can see that KR-U gets rapidly more accurate.
 
 \begin{center}
\begin{figure}[h!]
\begin{tabular}{c}
\includegraphics[scale=0.18]{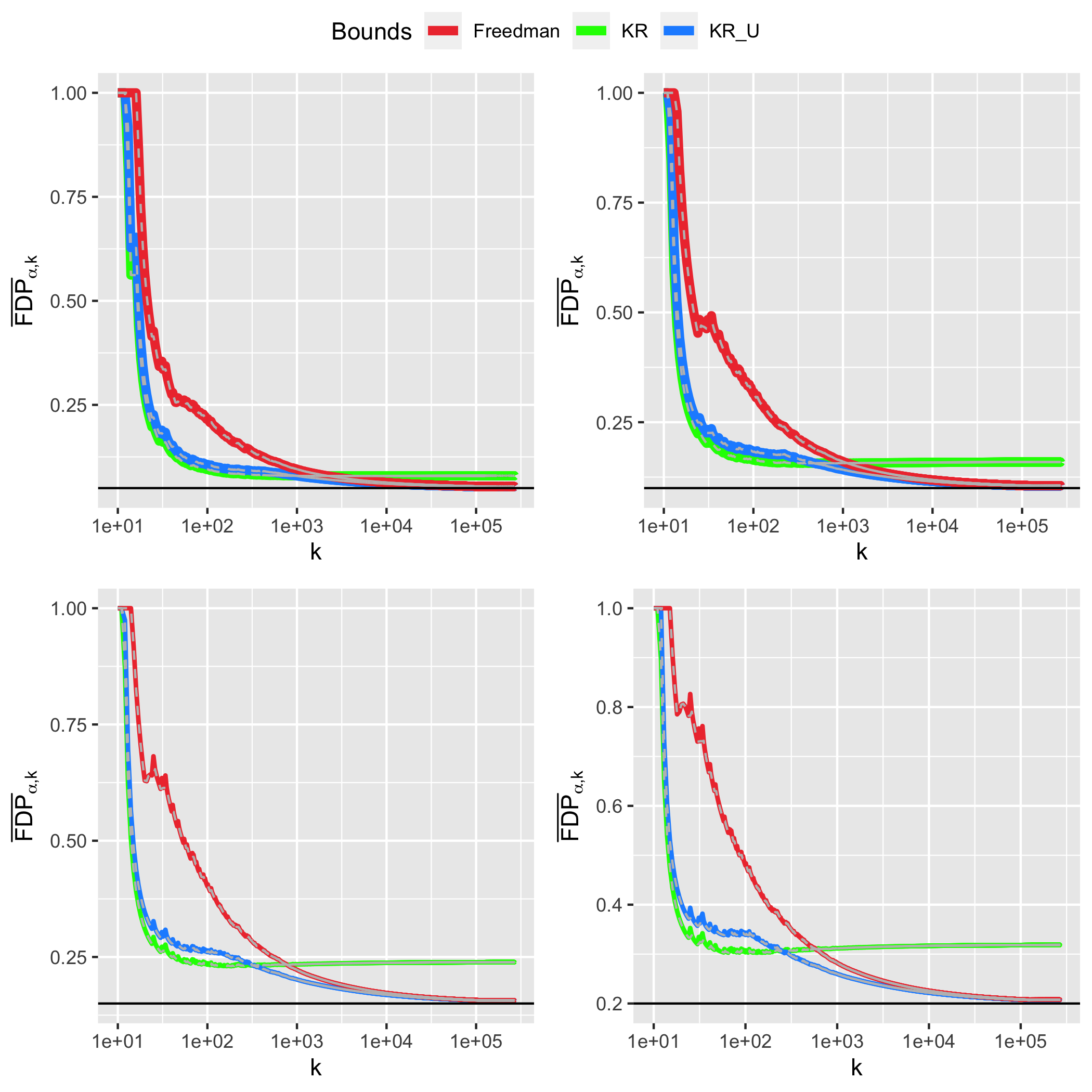}
\end{tabular}
\caption{Online FDP envelopes for LORD applied on IMPC data for four values of $\alpha\in \{0.05,0.1,0.15,0.2\}$ (horizontal black bars). The interpolated bounds are displayed for each procedure as a gray dashed line. \label{fig:onlineimpc}}
\end{figure}
\end{center}

\subsection{Comparison to \cite{li2024simultaneous}}\label{sec:comparli}

In this section, we compare the performances of the KR-U bound with respect to the recent bounds proposed in  \cite{li2024simultaneous}. For this, we reproduce the high dimensional Gaussian linear regression setting of Section~5.1 (a) therein, which generates binary $p$-values by applying the fixed-$X$ `sdp' knockoffs and the signed maximum lambda knockoff statistic of \cite{barber2015controlling}. Doing so, the $p$-values follow the preordered setting of Section~\ref{sec:knockoffsetting} and thus our bounds are non-asymptotically valid (note however that the $p$-values do not follow strictly speaking the VCT model of Section~\ref{sec:powerpreordered}). To be more specific, the considered Gaussian linear model $Y\sim \mathcal{N}(X\beta, I_n)$ is obtained by first generating $X$ and $\beta$ as follows: the correlated design matrix $X$ of size $n\times m$ is obtained by drawing $n=1500$ i.i.d. samples from the multivariate $m$-dimensional distribution $\mathcal{N}_m(0,\Sigma)$ where $\Sigma_{i,j}=0.6^{|i-j|}$, $1\leq i,j\leq m$; the signal vector $\beta\in \R^m$ is obtained by first randomly sampling a subset of $\{1,\dots,m\}$ of size $\lfloor \pi_1 m\rfloor$ for the non-zero entries of $\beta$ and then by setting all non-zero entries of $\beta$ equal to $a/\sqrt{n}$ for a given amplitude $a>0$.

 \begin{center}
\begin{figure}[h!]
\includegraphics[scale=0.17]{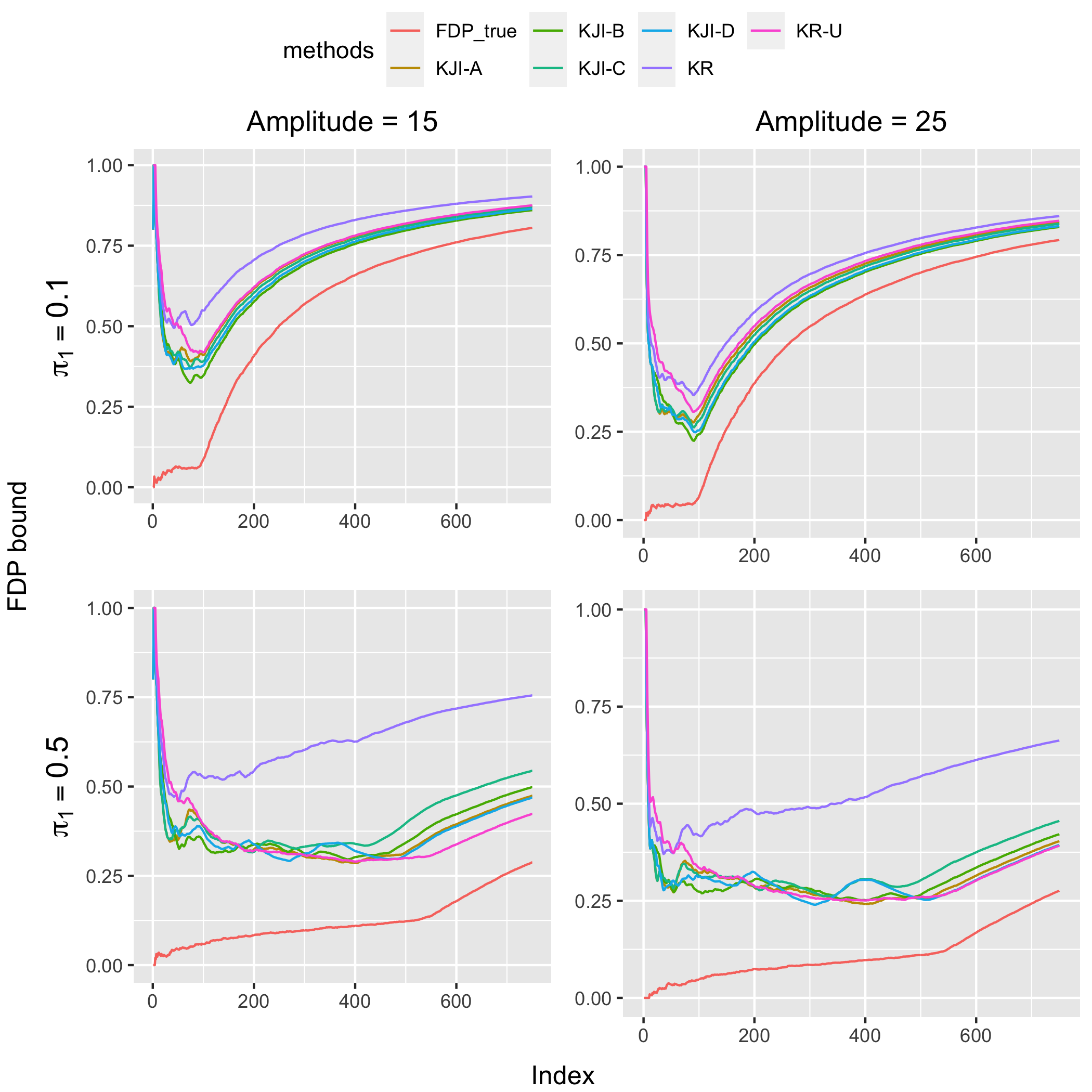}
\caption{Comparing the envelope $\wt{\FDP}_k^{{\tiny \mbox{KR-U}}}$, $k\geq 1$ given by \eqref{equ-interp3}-\eqref{preorderedKR} ($s=\lambda=0.5$)  to those of \cite{li2024simultaneous} in the Gaussian linear regression setting of Section~\ref{sec:comparli} for $m=1000$ (see text for more details). \label{fig:knockoffli}}
\end{figure}
\end{center}

First, in the spirit of Figure~3 in \cite{li2024simultaneous}, we display in Figure~\ref{fig:knockoffli} the envelope $(\wt{\FDP}_k^{{\tiny \mbox{KR-U}}}, k\geq 1)$ given by the interpolation \eqref{equ-interp3} of the envelope $(\ol{\FDP}_k^{{\tiny \mbox{KR-U}}}, k\geq 1)$ defined by \eqref{preorderedKR} (with $s=\lambda=1/2$), and compare it to those obtained in \cite{li2024simultaneous} (namely, KJI A/B/C/D) for $\pi_1\in \{0.1,0.5\}$, $a\in\{15,25\}$. We also set here $\delta=0.05$ to stick with the choice of  \cite{li2024simultaneous} (note that this requires to further calibrate the parameters of their method according to this value of $\delta$) and the number of replications is here only taken equal to $10$ for computational reasons. Markedly, the KR-U envelope becomes much better than KR and is competitive %\todo{est-ce qu'on peut prétendre que nos bornes sont computationally lighter parce que les bornes KJI sont basées sur du closed testing?} 
w.r.t.  KJI A/B/C/D, at least when $k$ is moderately large. 
As expected, the most favorable case for KR-U is when the signal has a large amplitude and is dense. 

Second, to stick with the consistency-oriented plots of the previous sections, we also display the corresponding FDP bounds for the BC procedure at level $\alpha\in\{0.15,0.2\}$ in Figure~\ref{fig:knockoffli_consist}. The conclusions are qualitatively similar.
 
\begin{center}
\begin{figure}[h!]
\includegraphics[scale=0.17]{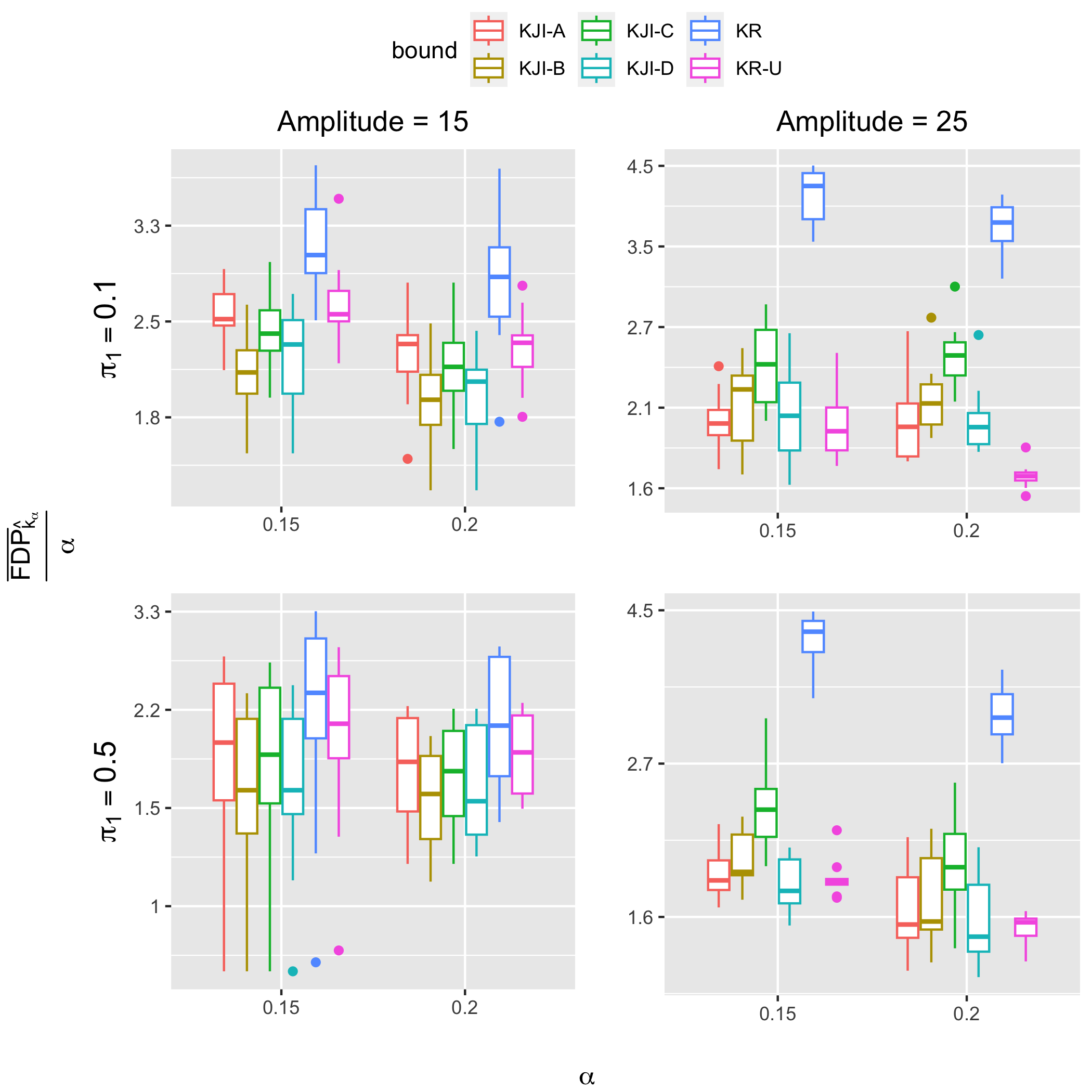}
\caption{Comparing the FDP bound $\wt{\FDP}_{\hat{k}_\alpha}^{{\tiny \mbox{KR-U}}}$ for $\hat{k}_\alpha$ the BC procedure \eqref{kchapeauLF} ($s=\lambda=0.5$) to those of \cite{li2024simultaneous} with respect to $\alpha\in\{0.15,0.2\} $ in the Gaussian linear regression setting of Section~\ref{sec:comparli} for $m=1000$ (see text for more details). \label{fig:knockoffli_consist}}
\end{figure}
\end{center}

\section{Conclusion}

The main point of this paper is to provide another point of view on FDP confidence bounds: we introduced a notion of $m$-consistency, %{which we consider as
  {a
  desirable asymptotical property which should act as a guiding principle
  when building such bounds,} by ensuring that the bound is sharp enough on particular FDR controlling rejection sets. Doing so, some previous bounds were shown to be inconsistent, including the {original} KR bounds. {While some other known FDP confidence bounds, in particular based on the DKW inequality, are $m$-consistent under certain assumptions, we have introduced new ones shown to satisfy this condition under more general conditions (in particular
high sparsity). New bounds based on the classical} Wellner/Freedman inequalities showed interesting behaviors, however simple modifications of KR bounds Hybrid/KR-U {by `stitching'} have been shown to be the most efficient, both asymptotically and for moderate sample size.

Overall, this work shows that $m$-consistency is a simple and fruitful criterion, and we believe that using it will be beneficial in the future to make wise choices among the rapidly increasing literature on FDP bounds.

\section*{Acknowledgements}

GB acknowledges support from:
Agence Nationale de la Recherche (ANR), ANR-19-CHIA-0021-01 (BiSCottE),
ANR-21-CE23-0035 (ASCAI), and IDEX REC-2019-044;
Deutsche Forschungsgemeinschaft (DFG) - SFB1294 - 318763901.

IM and ER have rbeen supported by ANR-21-CE23-0035 (ASCAI) and by the GDR ISIS through 
the `projets exploratoires' program (project TASTY). 
It is part of project DO 2463/1-1, funded by the Deutsche Forschungsgemeinschaft.

\bibliography{biblio} 

 \appendix

\section{Proofs}

\subsection{Proof of Proposition~\ref{cor:wellner}}\label{sec:newecdf}

For $j\geq 1$, let $\delta_j=\delta j^{-2}$, $\tau_j=2^{-j}$ and
\begin{align*}
A_j&=\left\{\forall t \in [\tau_j,1], \: n^{-1} \sum_{i=1}^n \ind{p_i\leq t} \leq t\:  \lambda_j\right\};\\
\lambda_j &= h^{-1}\left(\frac{\log(1/\delta_j)}{n \tau_j /(1-\tau_j) }\right) ,
\end{align*}
so that by Wellner's inequality, we have $\P(A_j)\geq 1-\delta_j$ and with a union bound $\P(\cap_{j\geq 1} A_j)\geq 1-\delta \pi^2/6$.
Now let $t\in (0,1)$ and $j_0=\min\{j\geq 1\::\: t\geq \tau_j\} = \min\{j\geq 1\::\: j \geq \log_2(1/t)\}$, so that $j_0=\lceil  \log_2(1/t) \rceil \geq 1$. 
%$\tau_{j_0} \leq t < \tau_{j_0-1}$, that is, $j_{0}-1< \log_2 (1/t) \leq j_0$. 
This yields
$$
\log(1/\delta_{j_0}) = \log(1/\delta) + 2 \log\left(\lceil  \log_2(1/t) \rceil\right).$$
On the event $\cap_{j\geq 1} A_j$, we have, since $t\in [\tau_{j_0}, 1]$ by definition,
$$
n^{-1} \sum_{i=1}^n \ind{p_i\leq t} \leq t\:  \lambda_{j_0} =t  h^{-1}\left(\frac{\log(1/\delta_{j_0})}{n \tau_{j_0} /(1-\tau_{j_0})}\right)= t h^{-1}\left(\frac{\log(1/\delta) + 2 \log\left(\lceil  \log_2(1/t) \rceil\right)}{n g(t) }\right),
$$
because %$h^{-1}(\cdot)$ is nondecreasing and 
$\tau_{j_0}=2^{-\lceil  \log_2(1/t) \rceil}$.
The result then comes from replacing $\delta$ by $\delta 6/\pi^2$.

\subsection{Proof of Theorem~\ref{th:BHasymp}}\label{sec:proof:th:BHasymp}

%\begin{proof}
First let  $F_m(t)=\ol{\Phi}(\ol{\Phi}^{-1}(t)-\mu_m)$, $\Psi_m(t)=F_m(t)/t$ and observe that $\Psi_m$ is continuous decreasing on $(0,1]$ with $\lim_0 \Psi_m=+\infty$. This implies that $t^*_m, t^\sharp_m\in (0,1)$ as described in the statement both exist, with 
$$t^*_m=\Psi_m^{-1}(\tau_m(\alpha/2)),\:\: t^\sharp_m=\Psi_m^{-1}(\tau_m(2\alpha)), \:\:\tau_m(\alpha)=\frac{m}{m_1} \left(\frac{1}{\alpha}-\frac{m_0}{m}\right).$$ 
We first establish 
\begin{align}
t^*_m \gtrsim m_1/m\label{equtmstar}\\
t^\sharp_m \lesssim m_1/m\label{equtsharp}.
\end{align}
 If $\beta=0$, then $m_0/m=1-c$, $m_1/m=c$, $\mu_m=b$, $\tau_m>0$, $F_m(t)=\ol{\Phi}(\ol{\Phi}^{-1}(t)-b)$, $\Psi_m(t)=\ol{\Phi}(\ol{\Phi}^{-1}(t)-b)/t$, $\tau_m(\alpha)$ all do not depend on $m$. Hence, $t^*_m$ and $t^\sharp_m$ are both constant, which establishes \eqref{equtmstar} and \eqref{equtsharp}.
Let us now turn to the sparse case, for which $\beta\in (0,1)$. The inequality \eqref{equtsharp} follows from the upper bound
$$
0.5 t^\sharp_m/\alpha=G_m(t^\sharp_m)\leq  t^\sharp_m + m_1/m.
$$
For \eqref{equtmstar}, the analysis is slightly more involved. We first prove that  for $m$ large enough
\begin{equation}\label{toprovelemmatopk}
\ol{\Phi}^{-1}(t_m^*)\leq \mu_m-b.
\end{equation}
This will establish \eqref{equtmstar}, since it implies $F_m(t_m^*)\geq F_m(\ol{\Phi}(\mu_m-b))= \ol \Phi(-b)>0$ and also $t_m^*=(\tau_m(\alpha/2))^{-1} F_m(t^*_m)\gtrsim m_1/m$.
On the one hand,
\begin{align*}
\Psi_m(\ol \Phi(\mu_m-b)) = \frac{\ol \Phi(-b)}{\ol \Phi(\mu_m-b)} \geq \ol \Phi(-b)\frac{\mu_m-b}{ \phi(\mu_m-b)}  = \Phi(-b)m^{\beta}\sqrt{2\beta\log m} 
\end{align*}
because $\mu_m-b=\sqrt{2\beta\log m}$ and $\phi(\mu_m-b)= m^{-\beta}$, and by using $\Phi(x)\leq \phi(x)/x$ for all $x>0$.
On the other hand, 
\begin{align*}
\Psi_m(t_m^*)=\tau_m(\alpha/2)\leq \frac{2}{\alpha} m^{\beta} .
\end{align*}
Hence, for $m$ large enough, we have $\Psi_m(\ol \Phi(\mu_m-b)) \geq \Psi_m(t_m^*)=\Psi_m(\ol \Phi(\ol \Phi^{-1}(t_m^*)))$, which in turn implies  \eqref{toprovelemmatopk}.

We now turn to prove the result \eqref{mainresultlemmatopk} and follow for a classical concentration argument. Let 
$$
\hat{G}_m(t)= m^{-1}\sum_{i=1}^m \ind{p_i\leq t},\:\:\:t\in [0,1],
$$
so that $G_m(t)=\E \hat{G}_m(t)$ for all $t\in [0,1]$. Hence, for all $t\in (0,1)$,
\begin{align*}
\P(\alpha\hat{k}_\alpha/m < t)&\leq \P\left(\hat{G}_m(t) \leq t/\alpha\right)\\
&=\P\left(\hat{G}_m(t)-G_m(t) \leq t/\alpha-G_m(t)\right),
\end{align*}
because $\alpha \hat{k}_\alpha/m=\max\{t\in (0,1)\::\: \hat{G}_m(t)\geq t/\alpha\} $ by definition of $\hat{k}_\alpha$.
Applying this with $t=t^*_m$, this gives
\begin{align*}
\P(\alpha\hat{k}_\alpha/m < t^*_m)&=\P\left(\hat{G}_m(t^*_m)-G_m(t^*_m) \leq -G_m(t^*_m)\right)\\
&\leq \exp\left(-c m G_m(t^*_m)\right) \leq \exp\left(-C m_1 F_m(t^*_m)\right),
\end{align*}
for some constant $C>0$, by applying Bernstein's inequality. 
Since $F_m(t_m^*)\geq  \ol \Phi(-b)>0$, this gives $\P(\alpha\hat{k}_\alpha/m < t^*_m)\leq e^{-dm_1}$ for $m$ large enough and some constant $d>0$.

Next, for all $t\in [t^\sharp_m,1)$, still applying Bernstein's inequality,
\begin{align*}
&\P(\alpha\hat{k}_\alpha/m > t)\\
&\leq \sum_{k=1}^m \ind{\alpha k /m >t} \P\left(\hat{G}_m(\alpha k/m) - {G}_m(\alpha k/m) \geq k/m - {G}_m(\alpha k/m)\right)\\
&\leq \sum_{k=1}^m \ind{\alpha k /m >t} \exp\left(-  m \frac{(k/m - {G}_m(\alpha k/m))^2}{G_m(\alpha k/m)+ (1/3)(k/m - {G}_m(\alpha k/m))}
%\frac{1}{2}\frac{m x_k^2}{  {G}_m(\alpha k/m)+x_k/3}
\right) 
\leq m \exp\left(- C m t^\sharp_m \right)  ,
%\P\left(\hat{G}_m(\alpha k/m) - {G}_m(\alpha k/m) \geq  t/\alpha - {G}_m(t) \right),
\end{align*}
because for all $\alpha k/m\geq t^\sharp_m$, $ k/m - {G}_m(\alpha k/m)\geq G_m(\alpha k/m) \geq G_m(t^\sharp_m) = 0.5  t^\sharp_m/\alpha$ (given the monotonicity of $t\mapsto G_m(t)/t$).
Applying this for $t=t^\sharp_m\in (0,1)$, we obtain
\begin{align*}
\P(\alpha\hat{k}_\alpha/m > t^\sharp_m)&\leq e^{-d m_1},
\end{align*}
because $t^\sharp_m\geq t^*_m\gtrsim m_1/m$. This proves the result.

%\end{proof}

\subsection{Proof of Proposition~\ref{prop:adapttopk}}\label{proofprop:adapttopk}

Let us prove it for the adaptive uniform Wellner envelope (the other ones being either simpler or provable by using a similar argument). The idea is to prove that on an event where the (non-adaptive) Wellner envelope \eqref{topkWellner} is valid, we also have $m_0\leq \hat{m}_0^{{\tiny \mbox{Well}}}$. The result is implied just by monotonicity (Lemma~\ref{lem:h}).

For this, we come back to apply \eqref{WellnerUniformSimple} with $(U_1,\dots,U_n)=(p_i,i\in \cH_0)$, $n=m_0$. Hence,  on an event with probability at least $1-\delta$, we have 
for all $ t \in (0,1),$ 
\begin{align*}
  m_0^{-1} \sum_{i\in\cH_0} \ind{p_i\leq t} &\leq t\: h^{-1}\left(\frac{C_t}{t m_0}\right)\leq t\: \left(1+\sqrt{C_t/(2tm_0)}\right)^2,
\end{align*}
where we apply an upper bound coming from Lemma~\ref{lem:h}. This gives
\begin{align*}
   V_t/m_0 &\geq 1- t\: \left(1+\sqrt{C_t/(2tm_0)}\right)^2= 1- t -\sqrt{2tC_t/m_0} - C_t/(2m_0) .
\end{align*}
As a result, $V_t\geq m_0(1-t) -  \sqrt{2tC_t m_0}- C_t/2$ and thus
$
(1-t)m_0 - \sqrt{2tC_t} m_0^{1/2} -  C_t/2-V_t\leq 0,
$
which gives 
\begin{align*}
&m_0\leq \left(\frac{ \sqrt{2tC_t} + \sqrt{2tC_t + 4(1-t)( C_t/2+V_t) }}{2(1-t)}\right)^2=\left( \sqrt{\frac{tC_t}{2(1-t)^2}} + \sqrt{\frac{C_t}{{2(1-t)^2}}+\frac{V_t}{1-t}}\right)^2.
\end{align*}
Since this is uniform in $t$, we can take the minimum over $t$, which gives the $m_0$ confidence bound $\hat{m}_0^{{\tiny \mbox{Well}}}$.
%$t^22C_t + 4(1-t)(t C_t/2+V_t)= 2tC_t+4(1-t)V_t $

\subsection{Proof of Proposition~\ref{prop:inconsist}}\label{sec:proof:prop:inconsist}
Classically, the Simes-based closed testing bound on $R=\BH(\alpha)$ is trivial (i.e. equal to $|R|$) if the global intersection hypothesis $[m]=\set{1,\ldots,m}$ is not rejected by the local Simes test, that is, if $\min_{1\leq k\leq m} (  p_{(k)} / k)>\delta/m$ \citep{goeman2011multiple}. By the assumptions on $G$ ($G$ is concave,  $G(0)=0$ and $G$ has derivative $g(0)$ at $0$), %the ratio $G(t)/t $ is decreasing and  value $g(0)$ in $0$, which gives 
we have $G(t)\leq \min(1,g(0)t)$ for all $t\in [0,1]$ and thus $F(t)\leq \min(1,(\pi_0 + (1-\pi_0) g(0)) t)$ for all $t\in [0,1]$. 
Hence, the $p$-values are stochastically lower bounded by a $\mathrm{Unif}[0,\gamma]$ variable,
  with $\gamma = (\pi_0 + (1-\pi_0) g(0))^{-1}$. Let us denote $\P_\gamma$ the joint probability
  of i.i.d. $\mathrm{Unif}[0,\gamma]$ $p$-values. We therefore have 
  \begin{align*}
 \P\paren{\min_k ( p_{(k)} / k ) \leq \delta/m} &\leq  \P_\gamma\paren{\min_k ( p_{(k)} / k ) \leq \delta/m}\\
    & = \P_\gamma\paren{\min_k ( \gamma^{-1}p_{(k)} / k ) \leq \delta/(\gamma m)}\\
    &= \delta/\gamma,
  \end{align*}
  where the last equality is from \cite{Sim1986}, since under $\P_\gamma$ the rescaled $p$-values $\gamma^{-1}p_k$ are i.i.d. $\mathrm{Unif}[0,1]$. Thus, as soon as $\delta<\gamma$, the Simes-based closed testing bound
  is trivial with probability bounded away from 0 for all $m$, and thus cannot be $m$-consistent.

  On the other hand, if $\alpha> \gamma$, then there is a non-zero solution $t^*$ to the equation 
  $F(t) = t/\alpha$ (due to strict concavity of $F$, this solution is unique). It is well-known (see \citealp{Chi2007}) that
  asymptotically as $m\rightarrow \infty$, the $BH(\alpha)$ rejection threshold will tend to $t^*$
  in probability. {Therefore, $\wh{k}^{\mathrm{BH}}_\alpha $ grows to infinity at a rate of order $m$ in probability ({in the sense $\wh{k}^{\mathrm{BH}}_\alpha \asymp_{\P^{(m)}_{\pi_0,G}} m$ by using the notation of Proposition~\ref{cor-consistencytopk}}), 
  which by Proposition~\ref{prop-consistencytopk} implies the BH$(\alpha)$-consistency
  of $\ol{\FDP}_\alpha^{{\tiny \mbox{DKW}}}$ and $\ol{\FDP}_\alpha^{{\tiny \mbox{Well}}}$.}

\subsection{Proof of Theorem~\ref{th:powerLF}}\label{proof:th:powerLF}

First note that $\FDP^\infty(t)$ is an decreasing function of $\Pi(t)$ because $\frac{1-F_1(\lambda)}{1-\lambda}<1<\frac{F_1(s)}{s}$,  see \eqref{equ-FDPstarLF}. Since $\Pi(t)$ is decreasing from $\pi(0)$ to $\pi(1)=\Pi(+\infty)$, we have that 
$\FDP^\infty:[0,+\infty)\to [\underline{\alpha},\ol{\alpha}]$ is continuous increasing, where $\ol{\alpha}=\left( 1 +  \pi(1) \left( \frac{1-F_1(\lambda)}{1-\lambda} -1\right)\right))/\left(1 +  \pi(1) \left( \frac{F_1(s)}{s} -1\right) \right)$. Hence, if $\alpha'< \ol{\alpha}$, we have $0<t^*_{\alpha'}<+\infty$, $t^*_m=t^*_{\alpha'}$ for $m$ large enough, and thus $\FDP^\infty(t^*_m)=\alpha'$. If $\alpha'\geq \ol{\alpha}$,  $t^*_{\alpha'}=+\infty$, $t^*_m=m^{\beta}$ and $\FDP^\infty(t^*_m)\leq \alpha'$. Both cases are considered in what follows.
Consider the events 
\begin{align*}
\Omega_1 &= \left\{ \sup_{a\leq k\leq m}\left|k^{-1} \sum_{i=1}^{k} \ind{p_{i} >  \lambda} - k^{-1} \sum_{i=1}^{k} \P(p_{i} >  \lambda)\right| \leq 1/a^{1/4} \right\};\\
\Omega_2 &= \left\{ \sup_{a\leq k\leq m}\left|k^{-1} \sum_{i=1}^{k} \ind{p_{i} \leq s} - k^{-1} \sum_{i=1}^{k} \P(p_{i} \leq s)\right| \leq 1/a^{1/4} \right\}.
\end{align*}
By Lemma~\ref{lembinom}, the event $\Omega_1\cap \Omega_2$ occurs with probability larger than $1-2(2+a^{1/2})e^{-2a^{1/2}}$. Let 
\begin{align*}
e_1 &= 1 +  \Pi_m(m^{-\beta}t^*_{m}) \left( \frac{1-F_1(\lambda)}{1-\lambda} -1\right)=1 +  \Pi(t^*_{m}) \left( \frac{1-F_1(\lambda)}{1-\lambda} -1\right);\\
e_2 &= 1 +  \Pi_m(m^{-\beta}t^*_{m}) \left( \frac{F_1(s)}{s} -1\right)= 1 +  \Pi(t^*_{m}) \left( \frac{F_1(s)}{s} -1\right),
\end{align*}
be the numerator and denominator of $\FDP^\infty(t^*_{\alpha'})$, so that $e_1/e_2=\FDP^\infty(t^*_m)\leq \alpha'$. 
Let $k_0=\lfloor m^{1-\beta} t^*_{m}\rfloor\leq m$. 
Provided that $ k_0\geq a$, we have 
\begin{align*}
\left|k_0^{-1} \sum_{i=1}^{k_0} \P(p_{i} >  \lambda) - (1-\lambda) e_1\right| &\leq  
\left|k_{0}^{-1}\sum_{i=1}^{k_0}\pi_m(i/m) - \Pi_m(m^{-\beta}t^*_{m})\right|\left|(1-F_1(\lambda)) -(1-\lambda)\right|\\
&\leq  
\left|k_{0}^{-1}\sum_{i=1}^{k_0}\pi_m(i/m) - \Pi_m(k_0/m)\right| + \left| \Pi_m(k_0/m)-\Pi_m(m^{-\beta}t^*_{m})\right| \\
&\leq 1/a + L/m^{1-\beta} ,
\end{align*}
by applying Lemma~\ref{lemesp} and using that $\Pi(\cdot)$  is $L$-Lipschitz. 
Similarly,
\begin{align*}
\left|k_0^{-1} \sum_{i=1}^{k_0} \P(p_{i} \leq s) - s e_2 \right| &\leq   1/a +L/m^{1-\beta}.
\end{align*}
We deduce that on $\Omega_1\cap \Omega_2$ and when $ k_0\geq a$, we have
\begin{align*}
\wh{\FDP}_{k_0}&\leq \frac{ e_1+\frac{1}{a(1-\lambda)} + \frac{L}{m^{1-\beta} (1-\lambda)}+ \frac{1}{k_0(1-\lambda)} +  \frac{1}{a^{1/4}(1-\lambda)}  }{\frac{1}{a s}\vee \left(e_2 - \frac{1}{a s} - \frac{L}{m^{1-\beta} s}- \frac{1}{k_0 s} -  \frac{1}{a^{1/4}s}\right)}
\leq \frac{ e_1+r  }{e_2 - r} \leq \frac{ e_1  }{e_2 } + 4r,
\end{align*}
provided that $e_2\geq 2r $, because $e_1\leq 1$, $e_2\geq 1$, and by considering $r$ as in the statement. Since $e_1/e_2\leq \alpha'\leq \alpha-4r$ and $e_2\geq 1\geq 2r$ by assumption,  we have $\wh{\FDP}_{k_0}\leq \alpha$ and thus $\hat{k}_\alpha\geq k_0$ on $\Omega_1\cap \Omega_2$. The result is proved by noting that $\hat{r}_\alpha= \sum_{i=1}^{\hat{k}_\alpha} \ind{p_{i} \leq  s}\geq \sum_{i=1}^{{k}_0} \ind{p_{i} \leq  s} \geq (e_2-r) k_0 s\geq k_0 s/2$ on this event. 

\begin{lemma}\label{lemesp}
In the setting of Theorem~\ref{th:powerLF}, we have for all $a\geq 1$, $m\geq a$,
\begin{equation}\label{equajdjustesp}
\sup_{a\leq k\leq m}\left|k^{-1}\sum_{i=1}^k \pi_m(i/m) - \Pi_m(k/m)\right|\leq 1/a.
\end{equation}
\end{lemma}

\begin{proof}
First note that because $\pi_m$ is nonnegative continuous decreasing, we have for all $k\geq 1$,
$$(1/k) \sum_{i=1}^{k} \pi_m(i/m) \leq \Pi_m(k/m)=(m/k)\int_0^{k/m} \pi_m(s)ds\leq (1/k) \sum_{i=0}^{k-1} \pi_m(i/m).$$
Since $\pi_m(0)\leq 1$, the result is clear.
\end{proof}

This following lemma is similar to Lemma~1 in \cite{lei2016power}.
\begin{lemma}\label{lembinom}
Let $X_i\sim \mathcal{B}(p_i)$, $1\leq i\leq m$, be independent Bernoulli variables for $p_i\in [0,1]$, $1\leq i\leq m$. Then we have for all $a\geq 1$ and $m\geq a$,
\begin{equation}\label{equconcbinom}
\P\left(\sup_{a\leq k\leq m}\left|k^{-1}\sum_{i=1}^k X_i - p_i\right|\geq 1/a^{1/4}\right)\leq (2+a^{1/2})e^{-2a^{1/2}}.
\end{equation}
\end{lemma}

\begin{proof}
By Hoeffding's inequality, we have for all $x>0$,
\begin{align*}
\P\left( \sup_{1\leq k\leq a}  \left|k^{-1}\sum_{i=1}^k (H_i - \pi_m(i/m))\right| \geq x\right) \leq 2 \sum_{k\geq a}e^{-2k x^2} = \frac{2}{1-e^{-2x^2}} e^{-2a x^2}\leq (2+1/x^2 ) e^{-2a x^2}.
\end{align*}
We deduce the result by considering $x=1/a^{1/4}$.
\end{proof}

\subsection{Proof of Theorem~\ref{thonlinepower}}\label{proof:thonlinepower}

%\begin{proof}
We get inspiration from the power analysis of \cite{javanmard2018online}. 
Let $c=\min(\alpha-W_0,W_0)$.  By definition \eqref{equalphaLORD}, the LORD procedure makes (point-wise) more rejections than the procedure given by the critical values
%Recall that the classical LORD procedure is given by 
\begin{align} 
    \alpha_T =  c \max\{\gamma_{T-\tau_j},j\geq 0\},
    %W_0 \gamma_T + (\alpha-W_0) \gamma_{T-\tau_1} + \alpha\sum_{j\geq 2} \gamma_{T-\tau_j},\quad T \geq 1,
    \label{eqn:LORD}
\end{align}
where, for any $j \geq 1$, $ \tau_j $ is the first time that the procedure makes $j$ rejections, that is,
 \begin{equation}
    \tau_j = \min\{t \geq 0\::\:R(t) \geq j\}\:\:\:\mbox{ ($\tau_j = + \infty$ if the set is empty)},
    \label{eqn:tauj}
\end{equation}
(note that  $\tau_{0}=0$)
for $R(T)=\sum_{t=1}^T \ind{p_t\leq \alpha_t}$. Let $\Delta_j = \tau_j-\tau_{j-1}$ the time between the $j$-th rejection and the $(j-1)$-th rejection.  {It is clear that $(R(t))_{t\geq 1}$ is a renewal process with holding times $(\Delta_j)_{j\geq 1}$ and jump times $(\tau_j)_{j\geq 1}$. In particular, the $\Delta_j$'s are i.i.d.}
%Clearly, $\alpha_T\geq \alpha^{(0)}_T:=c \max\{\gamma_{T-\tau_j},j\geq 0\}$ where
As a result, we have for all $r,k\geq 1$, 
\begin{align*}
\P(R(k)< r)&\leq \P(\tau_r\geq k)=\P(\Delta_1+\dots+\Delta_{r}\geq k) %=\P(\mathcal{B}(r,\E \Delta)\leq k) \\
%&=\P(\mathcal{B}(r,\E \Delta) -r\E \Delta \leq k-r\E \Delta)\leq \P(\mathcal{B}(r,\E \Delta) -r\E \Delta \leq -\delta r\E \Delta),
\leq r \E\Delta_1/k ,
\end{align*}
 %\et{should we further justify this?} 
where
\begin{align*}
\E \Delta_1&=\sum_{m\geq 1} \P(\Delta_1\geq m)= \sum_{m\geq 1} \prod_{\ell=1}^m (1-G(c \gamma_\ell))\leq \sum_{m\geq 1} e^{-m G(c \gamma_m)}.
\end{align*}
In addition, since $G$ is concave, %for all $x\in (0,1)$,
$$
\frac{G(x)}{x}\geq g'(x)=\pi_0 + \pi_1 c \:e^{\mu \bar \Phi^{-1}(x)} \geq  e^{c' \sqrt{2\log( 1/x)}}\geq (\log (1/x))^{\gamma+2},
$$
for $x$ small enough and $c,c'>0$ some constants. %  \et{Should we make this more precise?}. 
This gives for large $m\geq M$, $e^{-m G(c \gamma_m)}\leq e^{-cm\gamma_m (\log (1/(c \gamma_m)))^{2+\gamma} } \leq e^{-2\log m }$, for some $M>0$, by the choice made for $\gamma_m$.
As a result,
$$
\E \Delta_1\leq C+\sum_{m\geq M} e^{-m G(c \gamma_m)}\leq C+\sum_{m\geq M} e^{-cm\gamma_m (\log (1/(c \gamma_m)))^{\gamma+2} }\leq C+\sum_{m\geq 1} e^{-2\log m }=C+\pi^2/6,
$$
for some constant $C>0$.
This gives 
\begin{align*}
\P(R(k)< r)
\leq r (C+\pi^2/6)/k .
\end{align*}
 and taking  $r=k^{1-a}$ gives \eqref{equonlinepower}.
%\end{proof}

%\section{Other proofs}

\section{Tools of independent interest}

\subsection{A general envelope for a sequence of tests}

An important basis for our work is the following theorem, which has the flavor of Lemma~1 of \cite{katsevich2020simultaneous}, but based on a different martingale inequality, derived from a Freedman type bound (see Section~\ref{sec:freedman}).
{Also, while the pre-ordered and online settings are different, this result can be applied to both settings.}  %$\pi$ can be depending on $X$ eg knockoff (preordered) where it is independent of the $p$-values so that we have \eqref{condlemma1}.}

\begin{theorem}\label{th-mainresult}
Consider a potentially infinite set of null hypotheses $H_1,H_2,\dots$ for the distribution $P$ of an observation $X$, with associated $p$-values $p_1,p_2,\dots$ (based on $X$).
Consider an ordering $\pi(1),\pi(2),\dots$ (potentially depending on $X$) and a set of critical values $\alpha_1,\alpha_2,\dots$ (potentially depending on $X$). Let $\lambda\in [0,1)$ be a parameter and assume that there exists a filtration
$$
\mathcal{F}_k=\sigma\left( (\pi(i))_{1\leq i\leq k}, (\ind{p_{\pi(i)}\leq \alpha_i})_{1\leq i\leq k},  %( \alpha_i)_{1\leq i\leq k}, 
(\ind{p_{\pi(i)}> \lambda})_{1\leq i\leq k} \right), \:\: k\geq 1,
$$
%\et{in case $\lambda=0$, we should set $\mathcal{F}_k=\sigma\left( (\pi(i))_{1\leq i\leq k}, (\ind{p_{\pi(i)}\leq \alpha_i})_{1\leq i\leq k}   \right)$...}
such that for all $k\geq 2$,
\begin{equation}\label{condlemma1}
 \P(p_{\pi(k)}\leq t\:|\: \mathcal{F}_{k-1}, H_{\pi(k)}=0) \leq t \mbox{ for all $t\in[0,1]$.} 
%\P(p_{\pi(k)}\leq \alpha_k\:|\: \mathcal{F}_{k-1}, H_{\pi(k)}=0) \leq \alpha_{k} \mbox{ and } \P(p_{\pi(k)}> \lambda\:|\: \mathcal{F}_{k-1},H_{\pi(k)}=0) \geq 1-\lambda.
\end{equation}
Then, for any $\delta\in (0,1)$, with probability at least $1-\delta$, it holds 
$$\forall k \geq 1, \:\:\sum_{i=1}^k (1-H_{\pi(i)}) \ind{p_{\pi(i)}\leq \alpha_i} \leq \ol{V}_k,$$ for
\begin{align}
\ol{V}_k&=\sum_{i=1}^k (1-H_{\pi(i)})\ind{p_{\pi(i)} >  \lambda}\frac{\alpha_i }{1-\lambda}
+ \Delta\left(\sum_{i=1}^k (1-H_{\pi(i)}) \nu_i\right)
%+ 2\sqrt{\varepsilon_k(\delta)}\sqrt{ \sum_{i=1}^k (1-H_{\pi(i)}) \nu_i} +  \varepsilon_k(\delta) 
\label{boundlemma1},
\end{align}
where 
$\Delta(u)=2\sqrt{\varepsilon_u}\sqrt{{u \vee 1}} +  {\frac{1}{2}} \varepsilon_u$, 
	$\varepsilon_u=\log((1+\kappa)/\delta)+ 2\log\left(1+\log_2\left( {u \vee 1}\right)\right)$, $u>0$, $\kappa= \pi^2/6$.
%$\varepsilon_k(\delta)=\log((1+c)/\delta)+ 2\log\left(1+\left(\log_2\left(\sum_{i=1}^k  (1-H_{\pi(i)}) \nu_i\right)\right)_+\right)$, $c= \pi^2/6$ 
and $\nu_i=\alpha_i( 1 + \min(\alpha_i,\lambda)/(1-\lambda))$,  
 for $i\geq 1$.
\end{theorem}

\begin{proof}
By Lemma~\ref{lem:martin}, we can apply Corollary~\ref{cor:freedman} (it self coming from Freedman's inequality) with $$\xi_i=
(1-H_{\pi(i)})\left(\ind{p_{\pi(i)} \leq \alpha_i}-F_i(\alpha_i) \frac{\ind{p_{\pi(i)} >  \lambda}}{1-F_i(\lambda)}\right)
%\ind{i\in \cH_0} \left(\ind{p_{i}(X_i) \leq \alpha_i}-F_i(\alpha_i) \frac{\ind{p_i(X_i) >  \lambda_i}}{1-F_i(\lambda_i)}\right)
,$$ 
where $F_i(\alpha_i)$ and $F_i(\lambda)$ are defined by \eqref{equFk}. 
First note  that $\xi_i\leq 1=: B$ almost surely.
Let us now prove 
\begin{equation}\label{equ-gilles}
%\E(\xi_i^2\:|\: \mathcal{F}_{i-1})\leq \ind{i\in \cH_0} F_i(\alpha_i)/(1-\lambda_i).
\E(\xi_i^2\:|\: \mathcal{F}_{i-1})\leq (1-H_{\pi(i)}) \nu_i . %F_i(\alpha_i)( 1 + \min(\alpha_i,\lambda_i)/(1-\lambda_i)).
\end{equation}
Indeed, assuming first $\alpha_i\leq \lambda$, we have by \eqref{condlemma1}, 
\begin{align*}
&\E(\xi_i^2\:|\: \mathcal{F}_{i-1}) =
(1-H_{\pi(i)})\left(\E(\ind{p_{\pi(i)} \leq \alpha_i}\:|\: \mathcal{F}_{i-1})+(F_i(\alpha_i))^2 \frac{\E(\ind{p_{\pi(i)} >  \lambda}\:|\: \mathcal{F}_{i-1})}{(1-F_i(\lambda))^2}\right)\\
 %\ind{i\in \cH_0} \left( \E(\ind{p_{i}(X_i) \leq \alpha_i}\:|\: \mathcal{F}_{i-1})+\frac{F_i(\alpha_i)^2}{(1-F_i(\lambda_i))^2} \E\left(\ind{p_i(X_i) >  \lambda_i}\:|\: \mathcal{F}_{i-1}\right)\right)\\
&\leq (1-H_{\pi(i)}) (\alpha_i  + \alpha_i^2/({1-\lambda})) = (1-H_{\pi(i)})   \nu_i.
\end{align*}
%\gil{[Je prﾃδｩfﾃδｩrerais garder la version de la ligne d'avant car alors en prenant $\lambda_i = 1-\alpha_i$ on a juste un facteur $2$, c'est joli et c'est Knockoff style]}

which gives \eqref{equ-gilles}. Now, if $\alpha_i>\lambda$, still  by \eqref{condlemma1}, 
\begin{align*}
\E(\xi_i^2\:|\: \mathcal{F}_{i-1})=&\: (1-H_{\pi(i)})\left(\E(\ind{p_{\pi(i)} \leq \alpha_i}\:|\: \mathcal{F}_{i-1})+(F_i(\alpha_i))^2 \frac{\E(\ind{p_{\pi(i)} >  \lambda}\:|\: \mathcal{F}_{i-1})}{(1-F_i(\lambda))^2}\right.\\
&\left.-2 \frac{F_i(\alpha_i)}{1-F_i(\lambda)}  \E( \ind{\lambda<p_{\pi(i)} \leq \alpha_i}\:|\: \mathcal{F}_{i-1})\right)
\\
= &\: (1-H_{\pi(i)}) \big[F_i(\alpha_i)  + (F_i(\alpha_i))^2/(1-F_i(\lambda)) - 2 F_i(\alpha_i)(F_i(\alpha_i)-F_i(\lambda))/(1-F_i(\lambda) )\big]  \\
=&\: (1-H_{\pi(i)})  F_i(\alpha_i)\big[1+ (2 F_i(\lambda)-F_i(\alpha_i))/(1-F_i(\lambda))\big]\\
\leq&\: (1-H_{\pi(i)}) F_i(\alpha_i)\big[1+ F_i(\lambda)/(1-F_i(\lambda))\big] \leq  (1-H_{\pi(i)})\nu_i,
\end{align*}
which implies \eqref{equ-gilles} also in that case.
Finally, \eqref{equ-gilles} is established, which yields
$$
%\forall t \geq 1,\:\: S_t\leq 2\sqrt{\varepsilon_t(\delta)}\sqrt{ \sum_{i=1}^t \ind{i\in \cH_0} \nu_i} + 4 \varepsilon_t(\delta).
\forall k \geq 1,\:\: S_k\leq
2\sqrt{\varepsilon_k(\delta)}\sqrt{ \sum_{i=1}^k (1-H_{\pi(i)}) \nu_i} + 4 \varepsilon_k(\delta) 
$$
and thus \eqref{boundlemma1}.
\end{proof}

\begin{lemma}\label{lem:martin}
In the setting of Theorem~\ref{th-mainresult}, let
\begin{equation}\label{equFk}
F_k(\alpha_k)=\P(p_{\pi(k)}\leq \alpha_k\:|\: \mathcal{F}_{k-1}, H_{\pi(k)}=0), \:\:F_k(\lambda) = \P(p_{\pi(k)}\leq \lambda\:|\: \mathcal{F}_{k-1}, H_{\pi(k)}=0)
\end{equation}
%\et{not quite the same than defining $F_i(x)$...}
the process $(S_k)_{k\geq 1}$ defined by $$S_k=\sum_{i=1}^k (1-H_{\pi(i)})\left(\ind{p_{\pi(i)} \leq \alpha_i}-F_i(\alpha_i) \frac{\ind{p_{\pi(i)} >  \lambda}}{1-F_i(\lambda)}\right), \quad k\geq 1,$$ is a martingale with  respect to the filtration $(\mathcal{F}_k)_{k\geq 1}$.
\end{lemma}

\begin{proof}
First, $S_k$ is clearly $\mathcal{F}_k$ measurable. Second, we have for all $k\geq 2$,
\begin{align*}
\E(S_k\:|\: \mathcal{F}_{k-1}) &= \E\left( S_{k-1}+ (1-H_{\pi(k)})\left(\ind{p_{\pi(k)} \leq \alpha_k}-F_k(\alpha_k) \frac{\ind{p_{\pi(k)} >  \lambda}}{1-F_k(\lambda)}\right)\:\Big|\: \mathcal{F}_{k-1}\right) \\
&= S_{k-1} +  (1-H_{\pi(k)})(F_k(\alpha_k)-F_k(\alpha_k) ) 
= S_{k-1}.
\end{align*}
\end{proof}

\subsection{Uniform-Empirical version of Freedman's inequality}\label{sec:freedman}

{We establish a time-uniform, empirical Bernstein-style confidence bound for bounded
martingales. Various related inequalities have appeared in the literature, in particular
in the online learning community. The idea is based on `stitching' together time-uniform
bounds that are accurate on different segments of (intrinsic) time. The use of the stitching
principle has been further pushed and developed into many refinements by~\cite{Howardetal2021},
who also propose a uniform empirical Bernstein bound as a byproduct. The version given
here, based on a direct stitching of Freedman's inequality, has the advantage of being self-contained
with an elementary proof (though the numerical constants may be marginally worse than \citeauthor{Howardetal2021}'s).}

We first recall Freedman's inequality in its original version \citep{freedman1975tail}.
Let $(\xi_i,\cF_i)_{i\geq 1}$ be a supermartingale difference sequence,
i.e. $\e{\xi_i|\cF_{i-1}} \leq 0$ for all $i$. Define
$S_n:=\sum_{i=1}^n \xi_i$ (then $(S_n,\cF_n)$ is a supermartingale), and
$V_n:= \sum_{i=1}^n \var{\xi_i|\cF_{i-1}}$.

\begin{theorem}[Freedman's inequality; \protect{\citealp[Theorem 4.1]{freedman1975tail}}]
 Assume $\xi_i\leq 1$ for all $i\geq 1$. Then for all $t,v>0$:
 \begin{equation}
   \prob{S_n \geq t \text{ and } V_n \leq v \text{ for some } n\geq 1} \leq
   \exp\paren{- \varphi(v,t)}, \label{eq:mainfreedman}
 \end{equation}
 where
 \begin{equation}
   \varphi(v,t) := (v+t) \log \paren{1 + \frac{t}{v}} - t.
 \end{equation}
\end{theorem}

We establish the following corollary (deferring the proof for now):
\begin{corollary} \label{cor:freedmanaux}
   Assume $\xi_i\leq 1$ for all $i\geq 1$. Then for all $\delta \in (0,1)$ and $v>0$:
 \begin{equation}
   \prob{S_n \geq \sqrt{2 v \log \delta^{-1}} + \frac{\log \delta^{-1}}{2} \text{ and } V_n \leq v \text{ for some } n\geq 1} \leq \delta. \label{eq:corfreedman}
 \end{equation}
\end{corollary}%
%The following direct consequence appears in a related form in \citep[Lemma 3]{kakade2008generalization} for fixed $k$. Here we give a version that holds uniformly in $k$.
Following the stitching principle applied to the above we obtain the following.

\begin{corollary}\label{cor:freedman}
 Assume $\xi_i\leq B$ for all $i\geq 1$, where $B$ is a constant. {Put $\wt{V}_k := (V_k \vee B^2)$ and $\kappa = \pi^2/6$. Then for all $\delta \in (0,1/(1+\kappa))$, with probability
 at least $1-(1+\kappa)\delta$} it holds
 \[
   \forall k \geq 1 : S_k \leq 2  \sqrt{{\wt{V}_k} \eps(\delta,k)}
   % (\log \delta^{-1} + 2\log \log n)}
   + {\frac{1}{2}} B \eps(\delta,k), %(\log \delta^{-1}+ 2 \log \log n).
 \]
 where $\eps(\delta,k) := \log \delta^{-1} + 2 \log(1 + \log_2 (\wt{V}_k /B^2))$.

% If $\abs{\xi_i} \leq B$ for all $i\geq 1$, observe that $\eps(\delta,k) \leq \log \delta^{-1} + O(\log \log k)$.
\end{corollary}

\begin{proof}
Denote $v_j^2 := 2^{j}B^2$, $\delta_j := (j \vee 1)^{-2} \delta$, $j\geq 0$, and define the nondecreasing sequence of stopping times $\tau_{-1}=1$ and $\tau_j := \min \set[1]{k \geq 1: V_k > v_j^2}$ for $j\geq 0$. %Observe that $\tau_{\lceil \log_2 n \rceil} = n$.
Define the events for $j\geq 0$:
\begin{align*}
 A_j & := \set{ \exists k\geq 1: S_k \geq \sqrt{2v_j^2\log \delta_j^{-1}} + \frac{1}{2} B \log \delta_j^{-1} \text{ and } V_k \leq v_j^2},\\
  A'_j & := \set{ \exists k \text{ with }  \tau_{j-1} \leq k < \tau_{j} : S_k \geq 2\sqrt{\wt{V}_k \eps(\delta,k)} + \frac{1}{2} B \eps(\delta,k)}.
         %\paren{\log \delta^{-1} + 2 \log(1 + (\log_2 (\inner{S}_k^2 /B^2))_+)}} \right. \\
     %& \qquad \qquad  \left.
       %\paren{\log \delta^{-1} + 2 \log(1 + (\log_2 (\inner{S}_k^2 /B^2))_+)} \right\}.
         %\text{ and } v^2_{k+1} \leq \inner{S}_n \leq v^2_k}.
\end{align*}
From the definition of $v_j^2,\delta_j$, we have $j = \log_2 (v_j^2/B^2)$ for $j\geq 1$.
For $j\geq 1$, $\tau_{j-1} \leq k < \tau_{j}$ implies $\wt{V}_k = V_k$, $v_{j-1}^2 = v^2_j/2 < \wt{V}_k \leq v_{j}^2 $, and further
\[
  \log \delta_j^{-1} = \log \delta^{-1} + 2 \log \log_2 (v_j^2 /B^2)
  % \leq \log \delta^{-1} + 2 \log(1 + \log_2 (\inner{S}_k /B^2)_+) =
  \leq \eps(\delta,k).\]
Therefore it holds $A'_j \subseteq A_j$.
Furthermore, for $j=0$, we have $v_0^2=B^2,\delta_0=\delta$.
Further, if $k <\tau_0$ it implies $V_k <B^2$ and therefore $\wt{V}_k = B^2$, thus
$\eps(\delta,k) = \log \delta^{-1}$. Hence
%Thus, provided $\log \delta^{-1}\geq 1$ i.e. $\delta \leq 1/e$, it holds %$k<\tau_j$ implies $\inner{s}_k \leq B^2$, hence
\begin{multline*}
  A'_0 \subseteq \set{ \exists k \text{ with } k<\tau_0 : S_k \geq 2\sqrt{B^2 \log \delta_0^{-1}} + \frac{1}{2} B \log \delta_0^{-1}}\\
  \subseteq
  \set{ \exists k \geq 1 : S_k \geq  \sqrt{2v_0^2\log \delta_0^{-1}} + \frac{1}{2}B \log \delta_0^{-1}
  \text{ and } V_k \leq v_0^2 } = A_0.
\end{multline*}
Therefore, since by~\eqref{eq:corfreedman} it holds $\prob{A_j} \leq \delta_j$ for all $j \geq 0$:
\begin{align*}
 \prob{\exists k\leq n: S_k \geq 2  \sqrt{V_k\eps(\delta,k) } + B \eps(\delta,k)}  
  = \prob[4]{\bigcup_{j \geq 0 } A'_j}
  \leq \prob[4]{\bigcup_{j \geq 0 } A_j}
  \leq \delta \sum_{j\geq 0} (j \vee 1)^{-2} \leq 3 \delta. 
\end{align*}

\end{proof}

\begin{proof}[Proof of Corollary~\ref{cor:freedmanaux}]
  It can be easily checked that $\varphi(v,t)$ is increasing in $t$ (for $v,t>0$).
  Thus $S_n \geq t \Leftrightarrow \varphi(p,(S_n)_+) \geq \varphi(p,t)$.
Since $\varphi(v,0)=0$, and~$\lim_{t \rightarrow \infty} \varphi(v,t) = \infty$,
it follows that for any $\delta \in (0,1]$,
there exists a unique real $t(v,\delta)$ such
that $\varphi(v,t(v,\delta)) = -\log \delta$. It follows that~\eqref{eq:mainfreedman} is equivalent to:
\begin{equation}
  \forall v >0, \forall \delta \in (0,1]: \qquad
     \prob{ A_{v,\delta}} \leq
   \delta, \label{eq:mainfreedman2}
\end{equation}
where
\[
  A_{v,\delta}:=
  \set{\varphi(v,(S_n)_+) \geq -\log \delta \text{ and } T_n \leq v \text{ for some } n\geq 1}.
\]
Observe that $\varphi(v,t) = v h\paren{\frac{v+t}{v}}$, where $h$ is the function defined by~\eqref{funh}.
Since $h(\lambda) \geq 2 (\sqrt{\lambda}-1)^2$ from Lemma~\ref{lem:h}, we deduce $\varphi(v,t) \geq 2(\sqrt{v+t} -\sqrt{v})^2$
thus, whenever $\varphi(v,(S_n)_+) \leq -\log \delta$, we have:
\[
  \sqrt{v + (S_n)_+} \leq \sqrt{v} + \sqrt{\frac{\log \delta^{-1}}{2}};
\]
taking squares on both sides entails
\[
  S_n \leq \sqrt{2 v \log \delta^{-1}} + \frac{\log \delta^{-1}}{2},
\]
proving~\eqref{eq:corfreedman}.

\end{proof}

\section{Auxiliary results}

\begin{lemma}\label{lem:h}
The function $h$ defined by \eqref{funh} is increasing strictly convex from $(1,\infty)$ to $(0,\infty)$, while $h^{-1}$ is increasing strictly concave from $(0,\infty)$ to $(1,\infty)$. The functions $h$ and $h^{-1}$ satisfy the following upper/lower bounds:
\begin{align*} 
2 (\sqrt{\lambda}-1)^2  &\leq h(\lambda)\leq (\lambda-1)^2/2,\:\:\:\:\:\lambda>1\\
1+\sqrt{2y}&\leq h^{-1}(y)\leq  (1+ \sqrt{y/2})^2,\:\:\:\:\:y>0
\end{align*}
In particular, %$h^{-1}(y)-1\lesssim \sqrt{2y}$ as $y\to 0$.
$h^{-1}(y)-1\leq \sqrt{2y} + \cO(y)$ as $y\to 0$.
In addition, for any $c>0$, $x\in (1,+\infty)\mapsto xh^{-1}(c/x)$ is increasing.
\end{lemma}
  
\begin{proof}
Clearly, $h'=\log$, which is positive and increasing on $(1,\infty)$. This gives the desired property for $h$ and $h^{-1}$. Next, the bounds can be easily obtained by studying the functions $\lambda\mapsto(\lambda-1)^2/2-h(\lambda)$ and $\lambda\mapsto h(\lambda)-2 (\sqrt{\lambda}-1)^2$.
For the last statement, since $h^{-1}$ is strictly concave and $h^{-1}(0)=1$, we have that $y\in (0,\infty)\mapsto (h^{-1}(y)-1)/y$ is decreasing.  Since $y\in (0,\infty)\mapsto 1/y$ is also decreasing, this gives that $y\in (0,\infty)\mapsto h^{-1}(y)/y$ is decreasing. This gives the last statement.
\end{proof}

\begin{figure}[h!]
\includegraphics[scale=0.6]{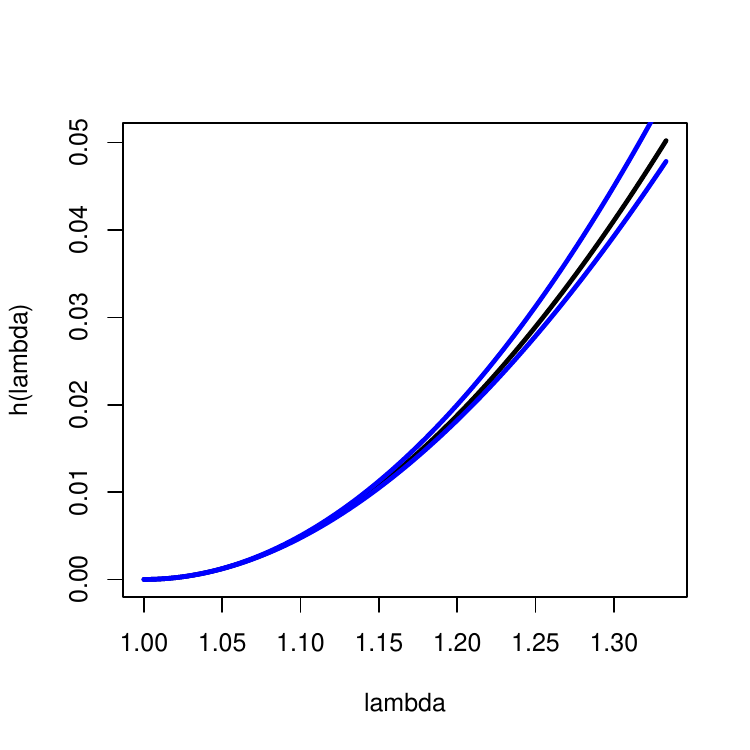}
\includegraphics[scale=0.6]{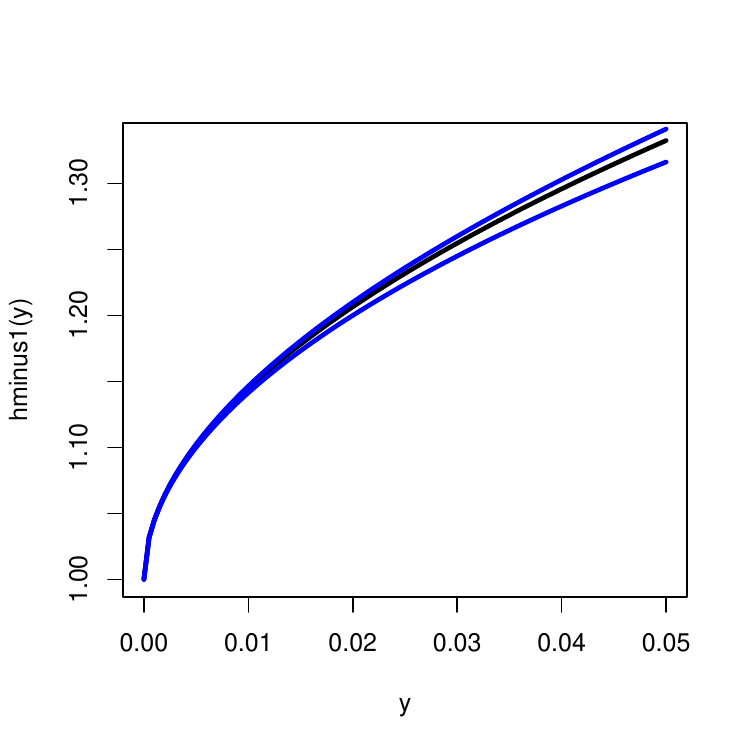}
\caption{Displaying $h$ (left) and $h^{-1}$ (right). Bounds of Lemma~\ref{lem:h} are displayed in blue.
\label{fig:h}}
\end{figure}

\begin{lemma}[Wellner's inequality, Inequality 2, page 415, with the improvement of Exercise 3 page 418 of \citealp{shorack2009empirical}]\label{Wellnerineq}
Let $U_1,\dots,U_n$ be $n\geq 1$ i.i.d. uniform random variables. For all $\lambda\geq 1$, $a\in [0,1)$, we have
$$
\P\left(\exists t \in [a,1]\::\: n^{-1} \sum_{i=1}^n \ind{U_i\leq t}/t \geq \lambda \right)\leq e^{-n a h(\lambda)/(1-a)},
$$
for $h(\cdot)$ defined by \eqref{funh}. %$h(\lambda)=\lambda(\log \lambda -1)+1$ increasing one to one map from $(1,\infty)$ to $(0,\infty)$.
\end{lemma}

\begin{lemma}\label{lem:KRconstant}
The KR constants in \eqref{FDPKRknockoff} and \eqref{onlineKR} satisfy, as $a\to \infty$,
\begin{align*}
\frac{ \log(1/\delta_a)}{a \log(1+\frac{1 - \delta_a^{B/a}}{B})}  &= 1 + O\left(\frac{\log(a)}{a} \right);\\
\frac{ \log(1/\delta_a)}{a \log(1 + \log(1/\delta_a)/a )}  &= 1 + O\left(\frac{\log(a)}{a} \right),
\end{align*}
where $\delta_a=c\delta/a$, $c=\pi^2/6$ and the $O(\cdot)$ depends only on the constants $\delta>0$ and $B>0$.
\end{lemma}

\section{Additional experiments}\label{sec:addexp}

\subsection{Interpolated bounds}

We reproduce here the figures of the numerical experiments in the top-$k$ and preordered settings, by adding the interpolated bounds. On each graph, the median of the generated interpolated bound is marked by a star symbol, which is given in addition to the former boxplot (of the non-interpolated bound). By doing so, we can evaluate the gain brought by the interpolation operation in each case. Note that the interpolated bound is not computed for $m\geq 10^5$ for computational cost reasons.

\begin{center}
\begin{figure}[!]
\begin{tabular}{cc}
$\alpha=0.05$ & $\alpha=0.1$\\
\includegraphics[scale=0.12]{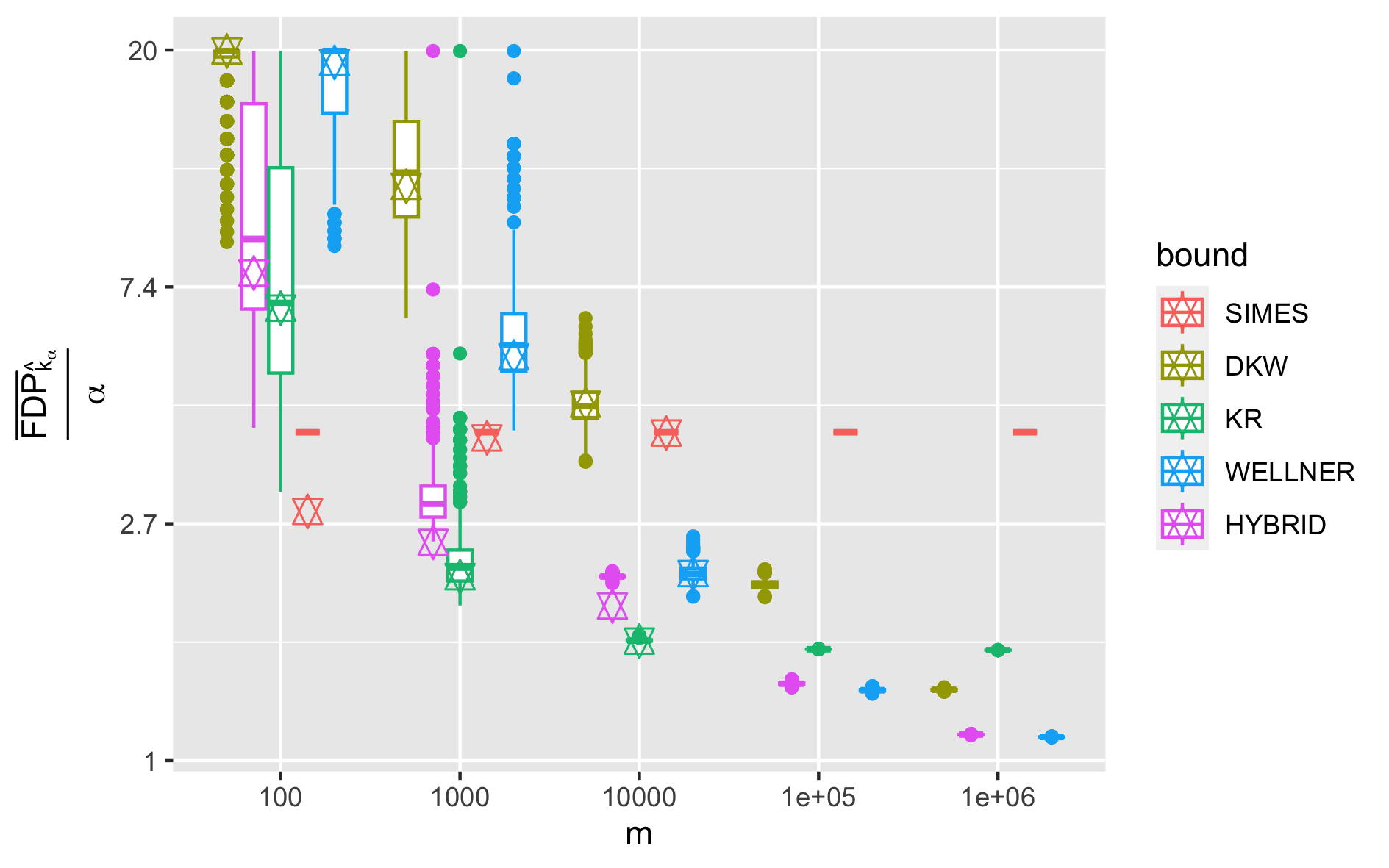}&\includegraphics[scale=0.12]{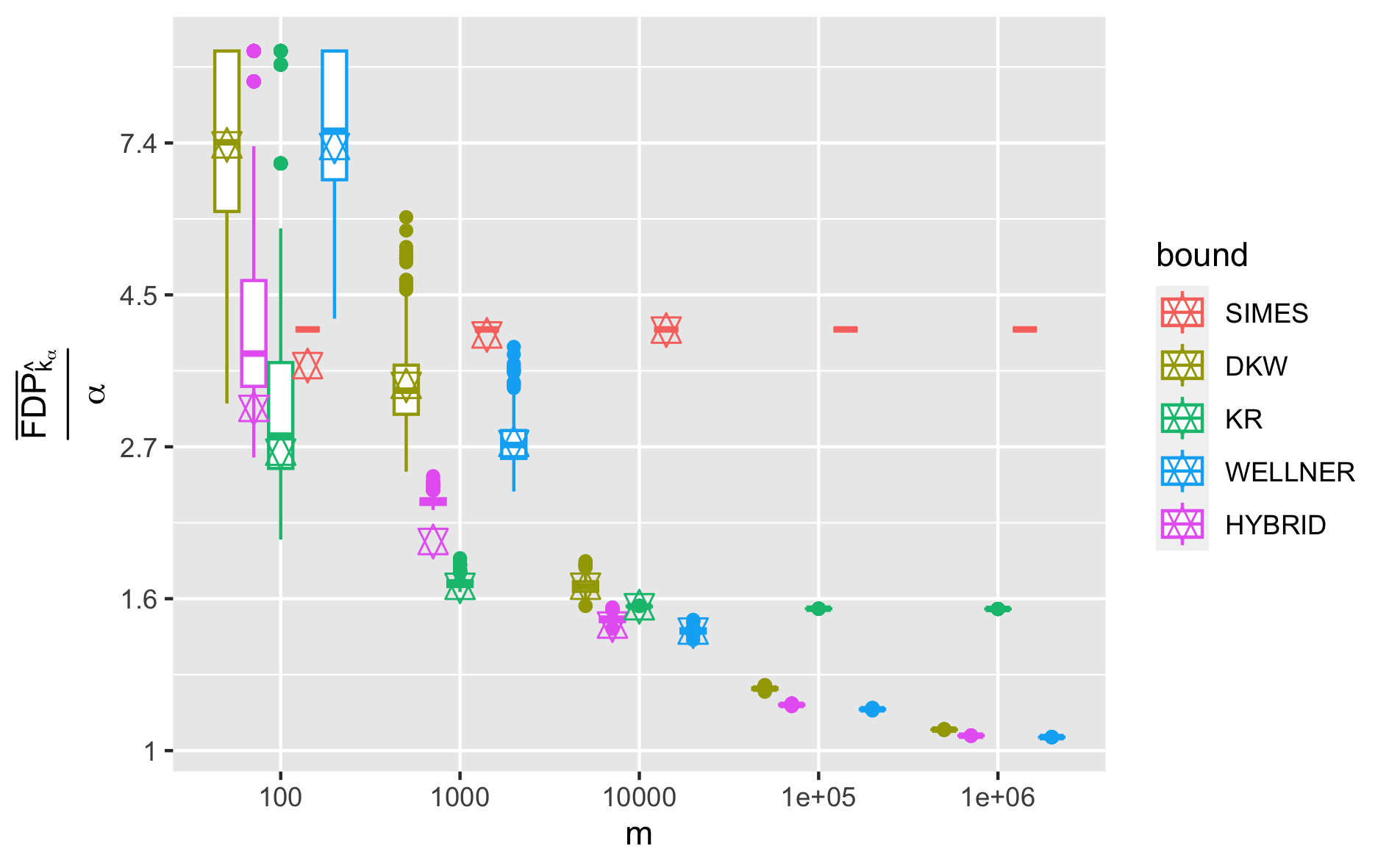}\\
$\alpha=0.15$ & $\alpha=0.2$\\
\includegraphics[scale=0.12]{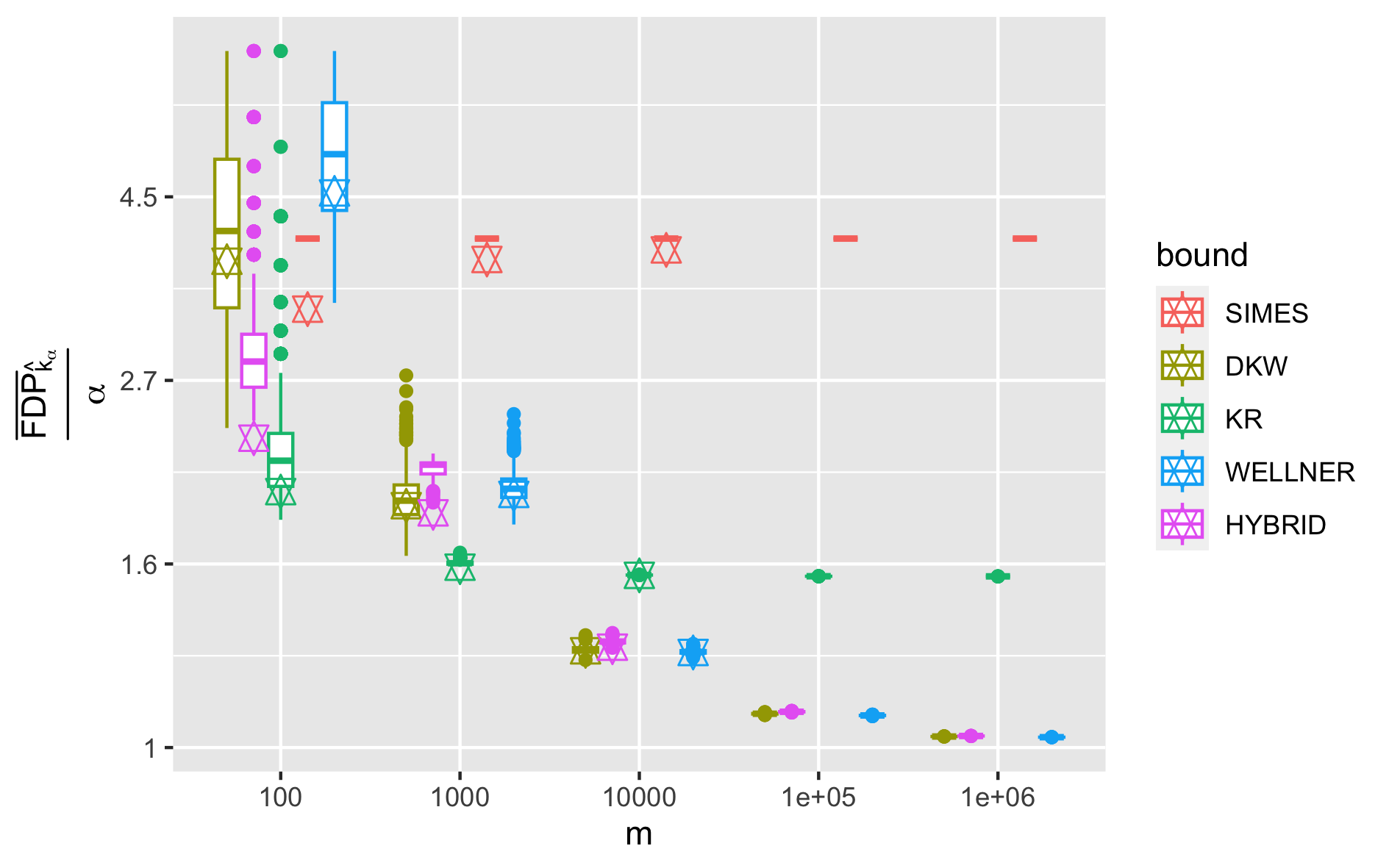}&\includegraphics[scale=0.12]{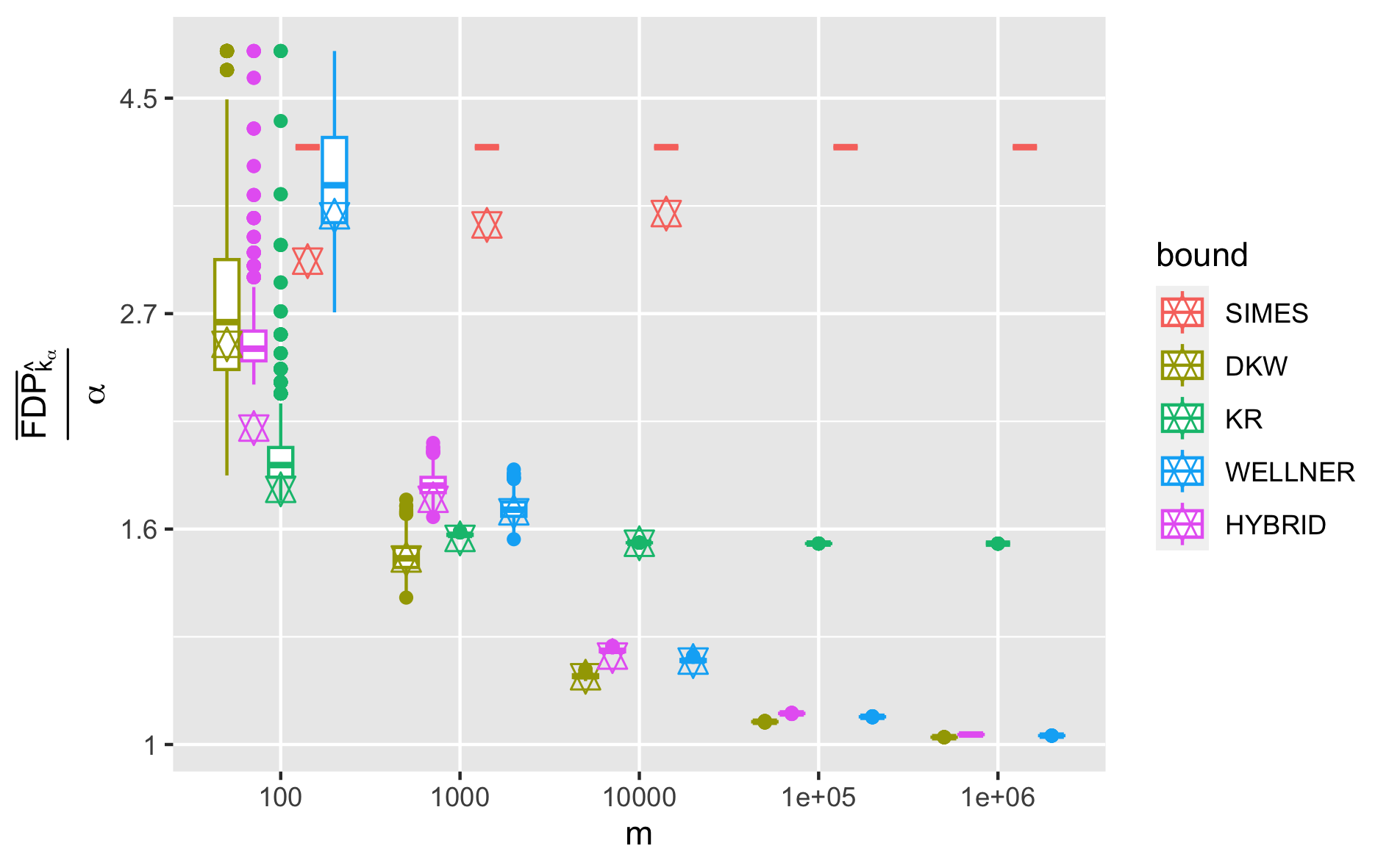}
\end{tabular}
\caption{Figure~\ref{fig:topkdense} where we have  superposed in each case  the (median of the) interpolated bounds (star symbols). Top-$k$ dense case ($\pi_0=0.5$, $\mu=1.5$).\label{fig:topkdense_inter}}
\end{figure}
\end{center}

\begin{center}
\begin{figure}[h!]
\begin{tabular}{cc}
\includegraphics[scale=0.12]{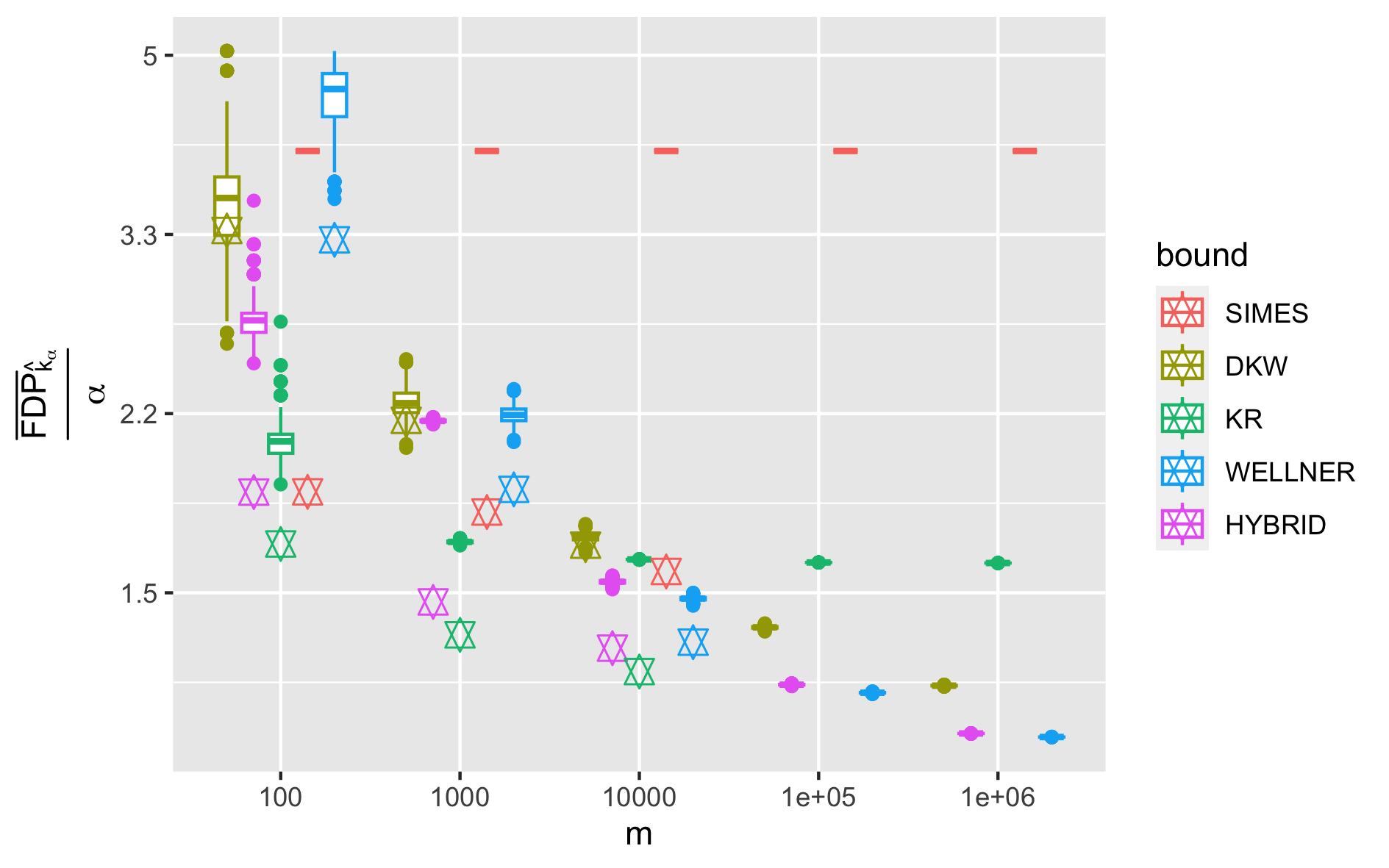}&\includegraphics[scale=0.12]{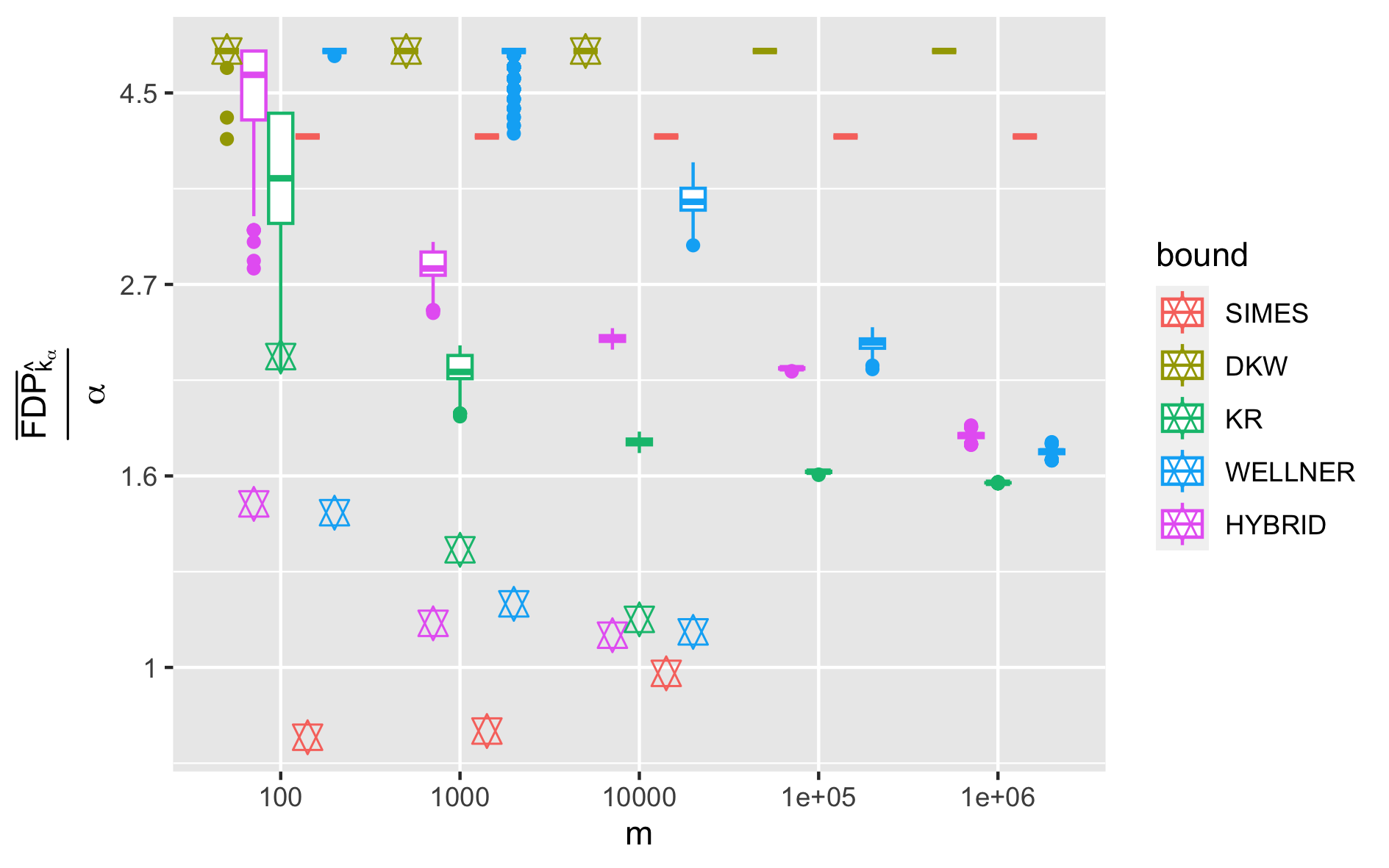}\\
\end{tabular}
\caption{Figure~\ref{fig:topksparse} where we have  superposed in each case  the (median of the) interpolated bounds (star symbols). Top-$k$ sparse case $\pi_0=1- 0.5 m^{-0.25}$, $\mu=\sqrt{2\log m}$  (left) $\pi_0= 1- 0.5 m^{-0.55}$, $\mu=${10}  (right), $\alpha=0.2$. \label{fig:topksparse_inter}} 
\end{figure}
\end{center}

\begin{center}
\begin{figure}[h!]
\begin{tabular}{cc}
Non adaptive & Adaptive \\
\includegraphics[scale=0.12]{topk/with_interp/dense_alpha_0_2.png}&\includegraphics[scale=0.12]{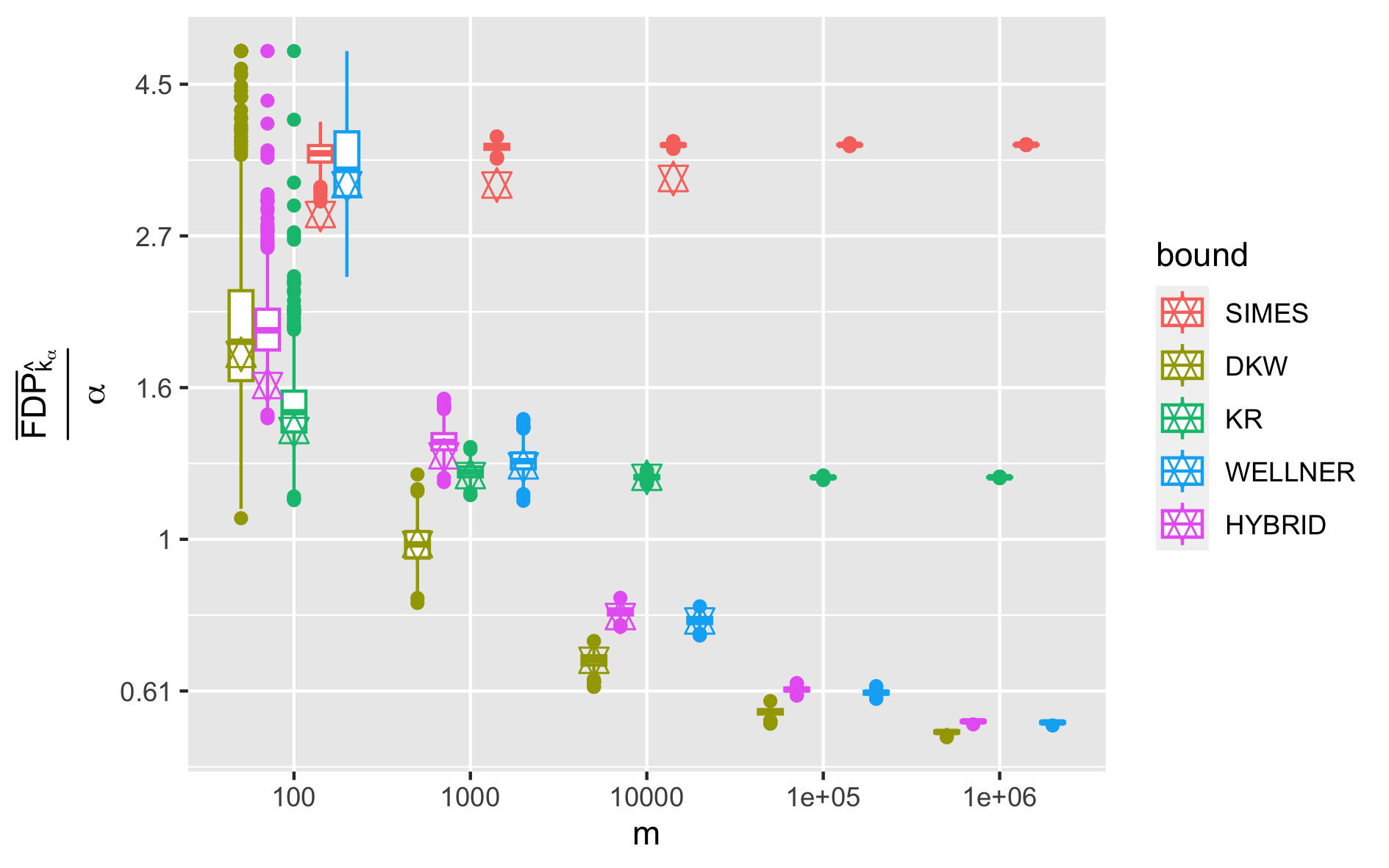}
\end{tabular}
\caption{Figure~\ref{fig:topkadapt} where we have  superposed in each case  the (median of the) interpolated bounds (star symbols).  Top-$k$ dense case with nonadaptive bounds (left) and adaptive bounds (right) ($\pi_0=0.5$, $\alpha=0.2$). \label{fig:topkadapt_inter}}
\end{figure}
\end{center}

\begin{center}
\begin{figure}[h!]
\begin{tabular}{cc}
$\alpha=0.05$ & $\alpha=0.1$\\
\includegraphics[scale=0.12]{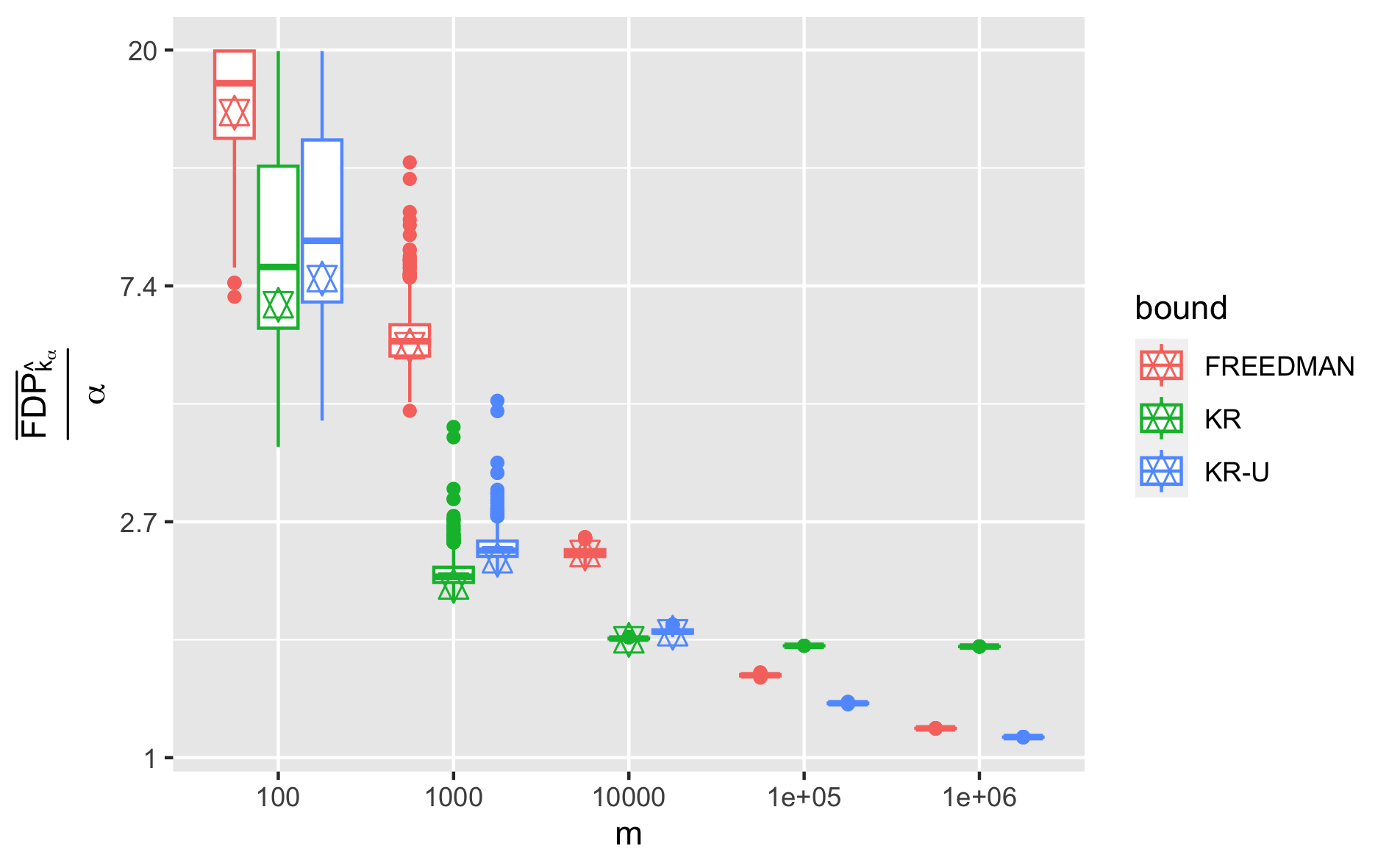}&\includegraphics[scale=0.12]{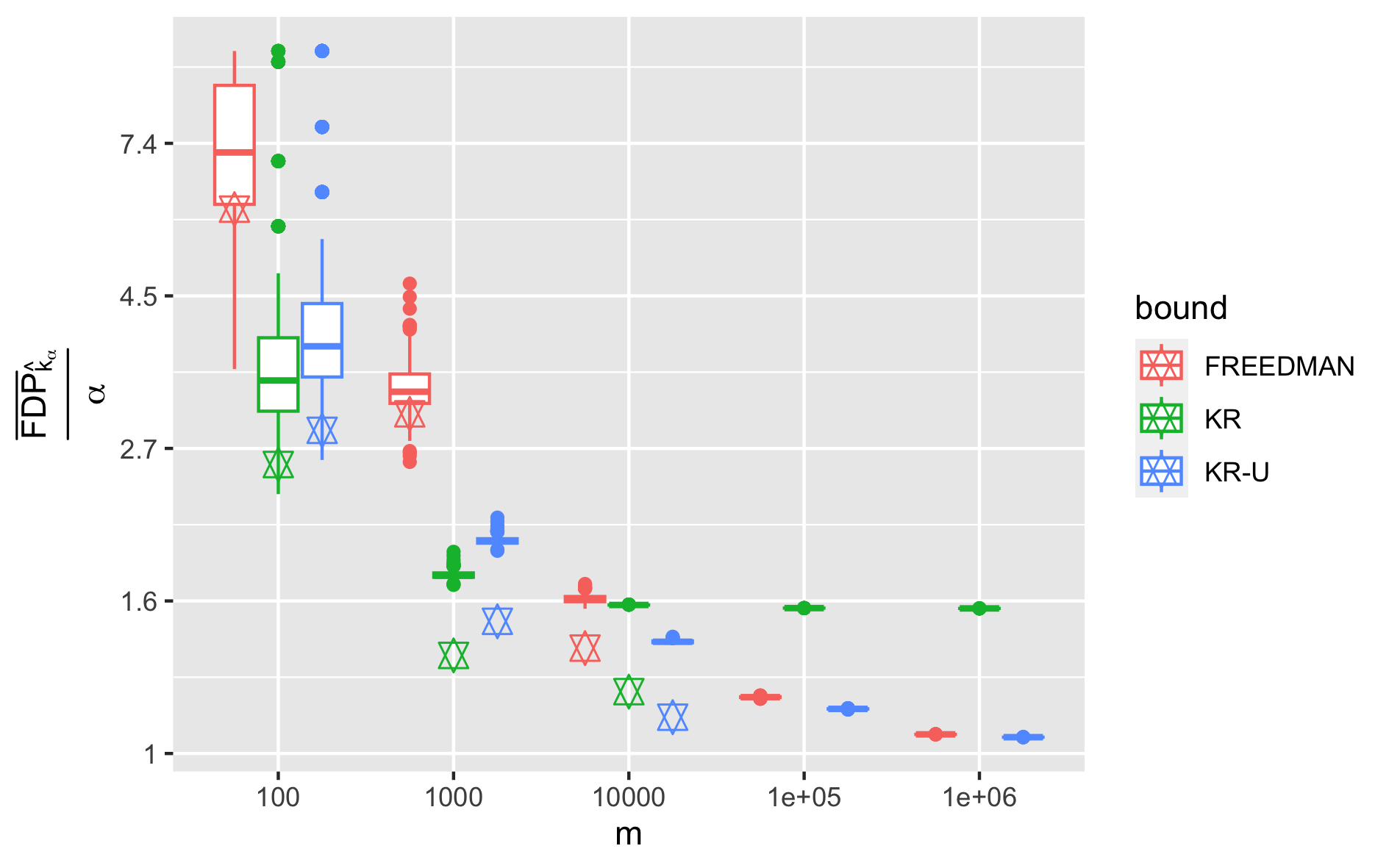}\\
$\alpha=0.15$ & $\alpha=0.2$\\
\includegraphics[scale=0.12]{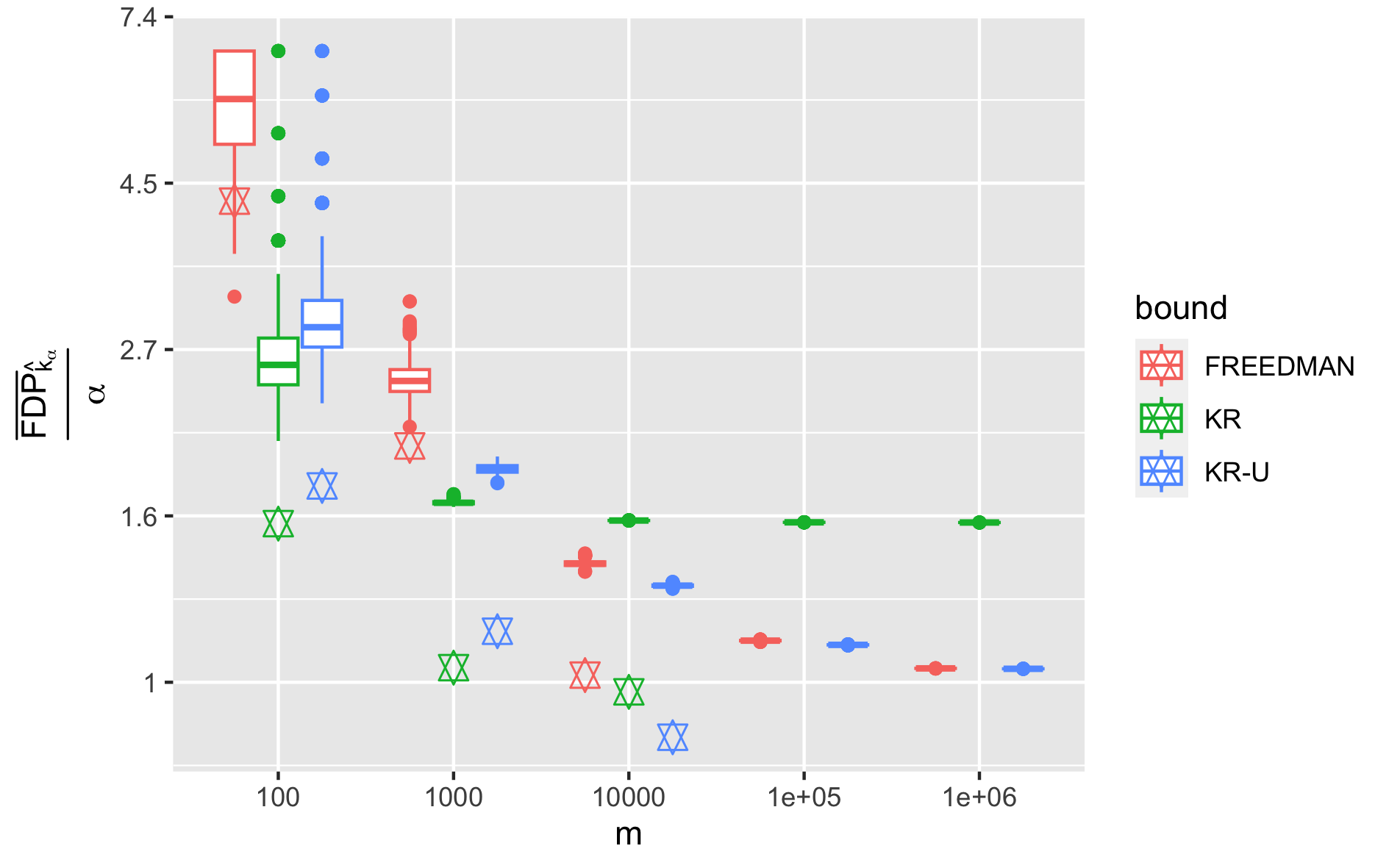}&\includegraphics[scale=0.12]{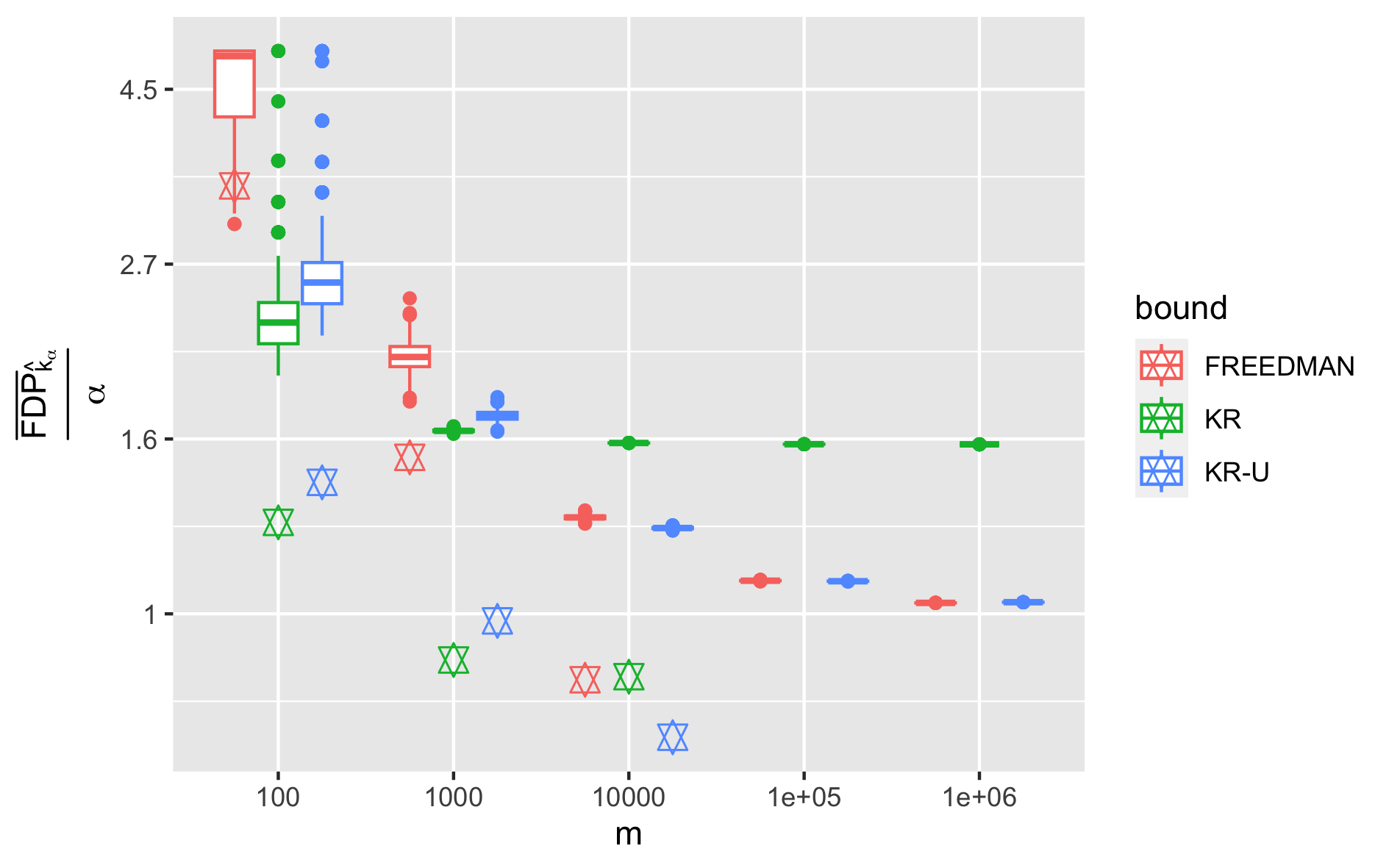}
\end{tabular}
\caption{Figure~\ref{fig:preorderLFdense} where we have  superposed in each case  the (median of the) interpolated bounds (star symbols). Preordered dense ($\beta=0$) LF setting with LF procedure ($s=0.1 \alpha$, $\lambda=0.5$).\label{fig:preorderLFdense_inter}}
\end{figure}
\end{center}

\begin{center}
\begin{figure}[h!]
\begin{tabular}{cc}
$\alpha=0.05$ & $\alpha=0.1$\\
\includegraphics[scale=0.12]{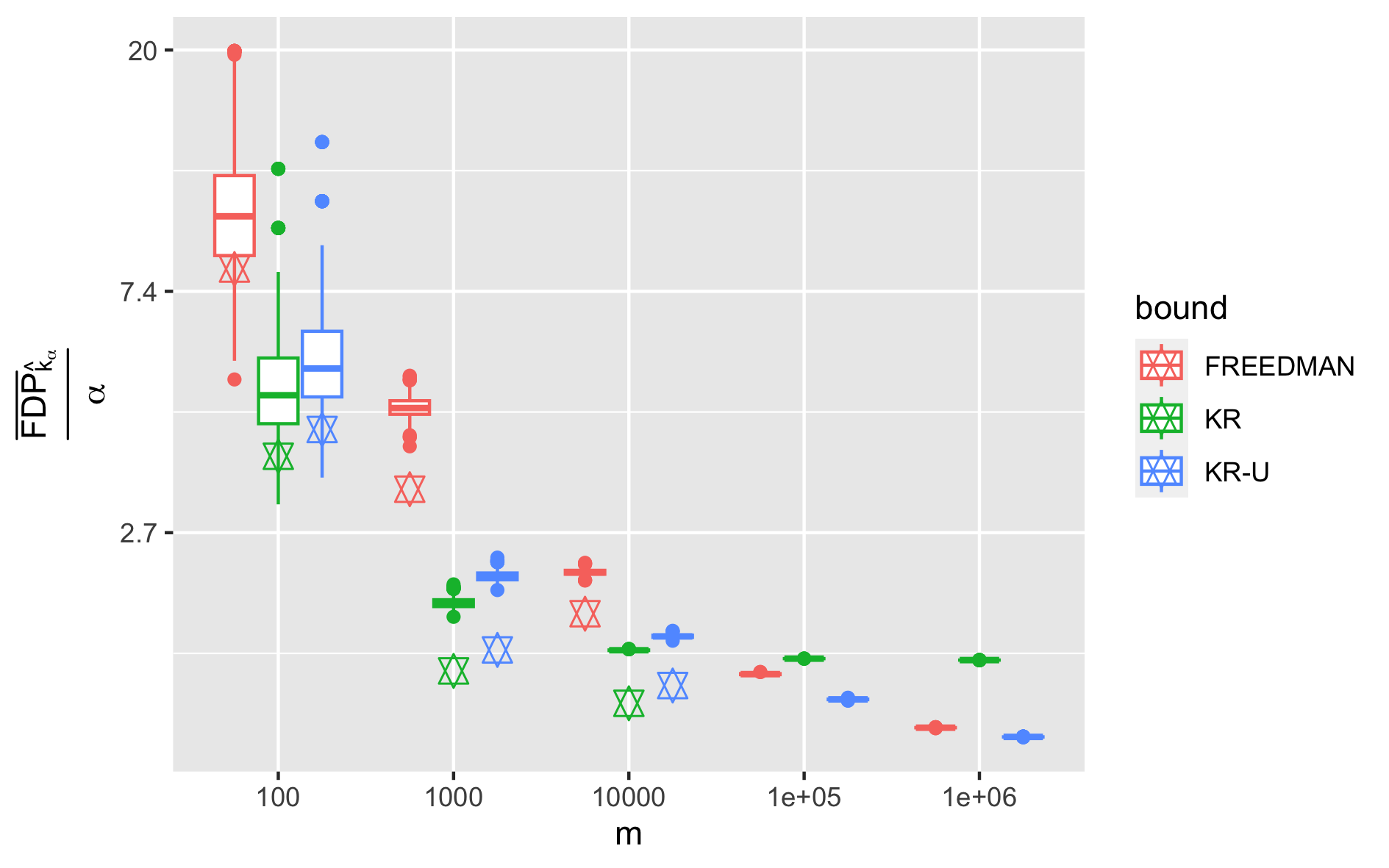}&\includegraphics[scale=0.12]{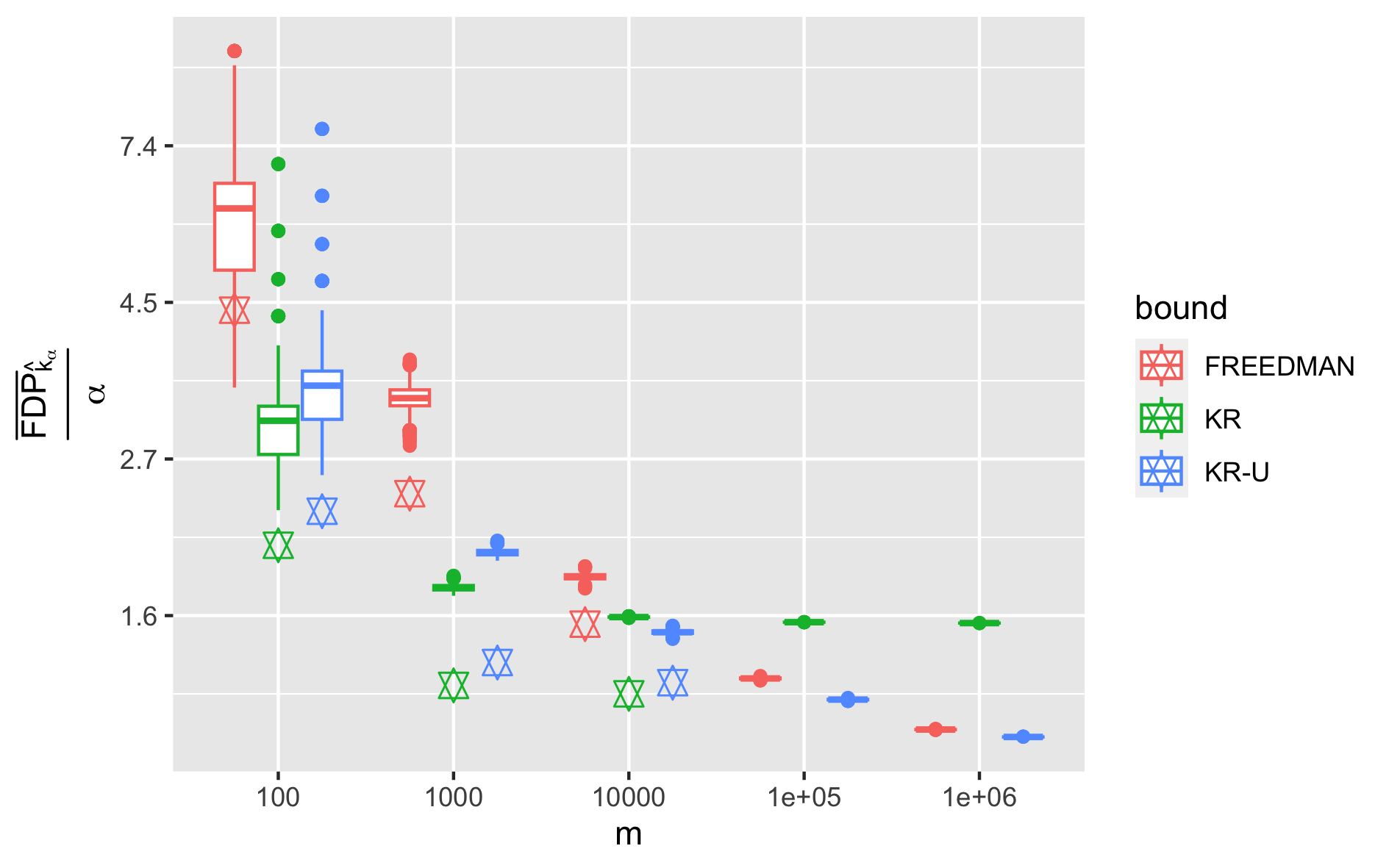}\\
$\alpha=0.15$ & $\alpha=0.2$\\
\includegraphics[scale=0.12]{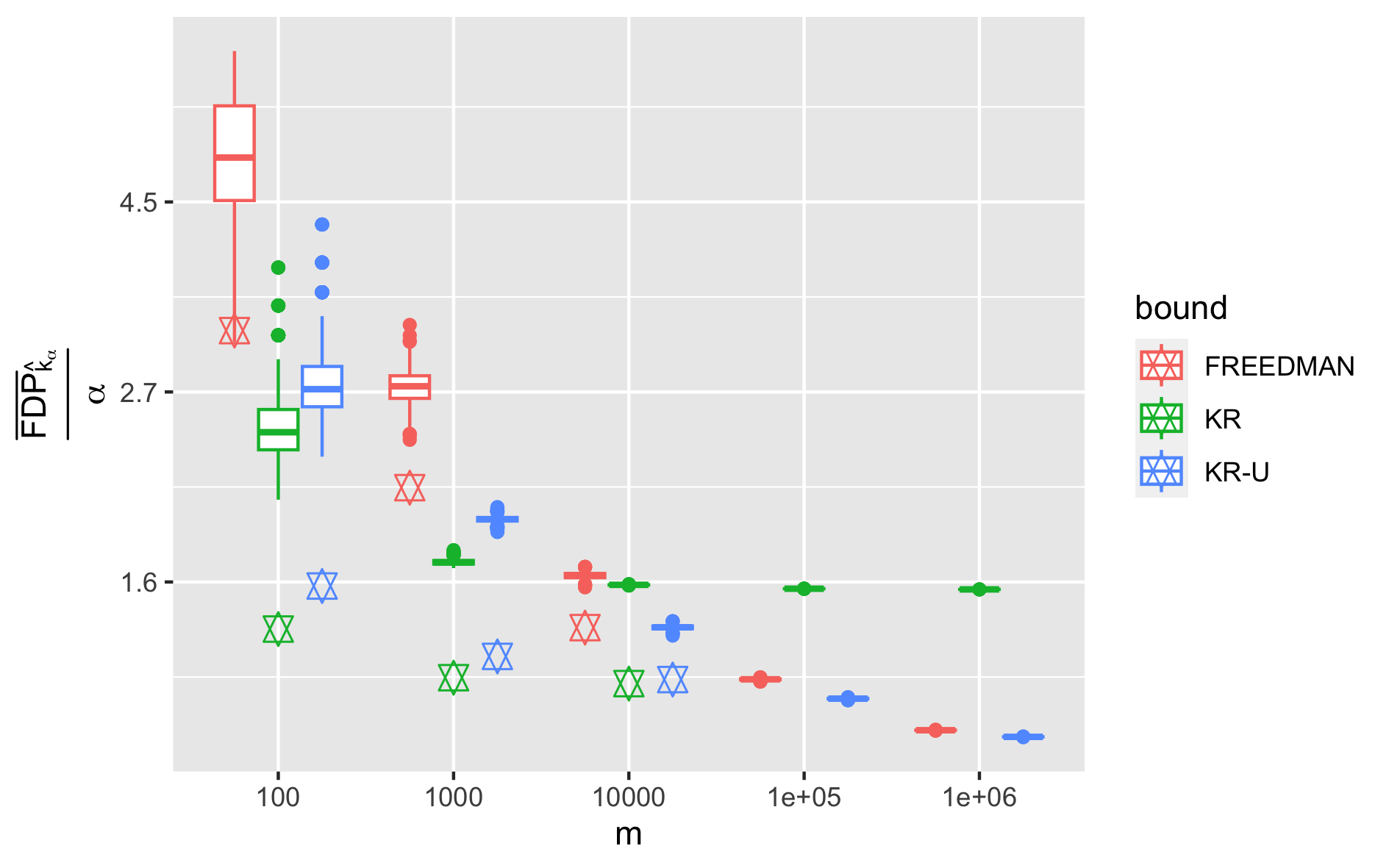}&\includegraphics[scale=0.12]{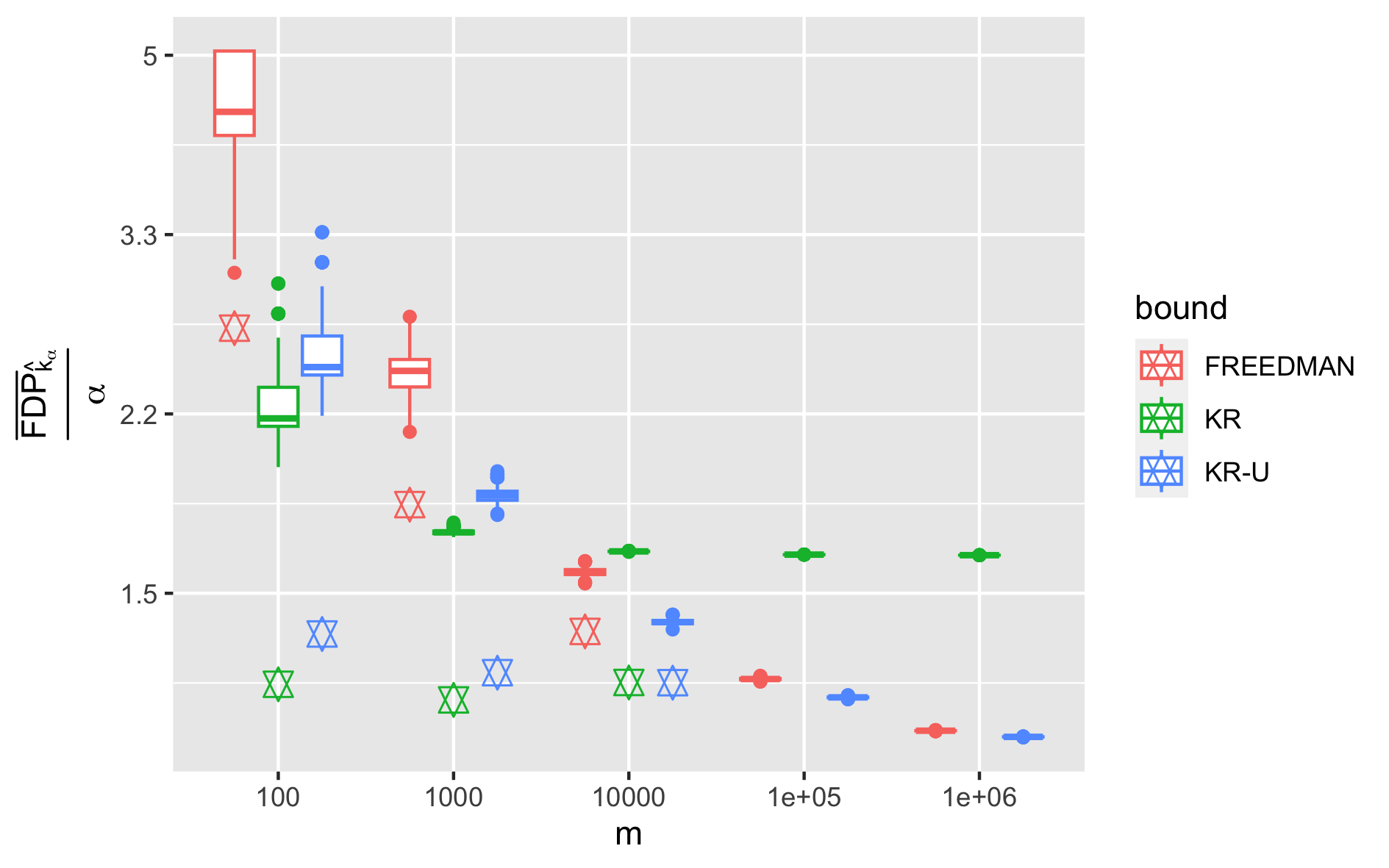}
\end{tabular}
\caption{Figure~\ref{fig:preorderLFsparse} where we have  superposed in each case  the (median of the) interpolated bounds (star symbols). Preordered sparse ($\beta=0.25$) LF setting with LF procedure ($s=0.1 \alpha$, $\lambda=0.5$). \label{fig:preorderLFsparse_inter}}
\end{figure}
\end{center}

\begin{center}
\begin{figure}[h!]
\begin{tabular}{cc}
$\alpha=0.15$ & $\alpha=0.2$\\
\includegraphics[scale=0.12]{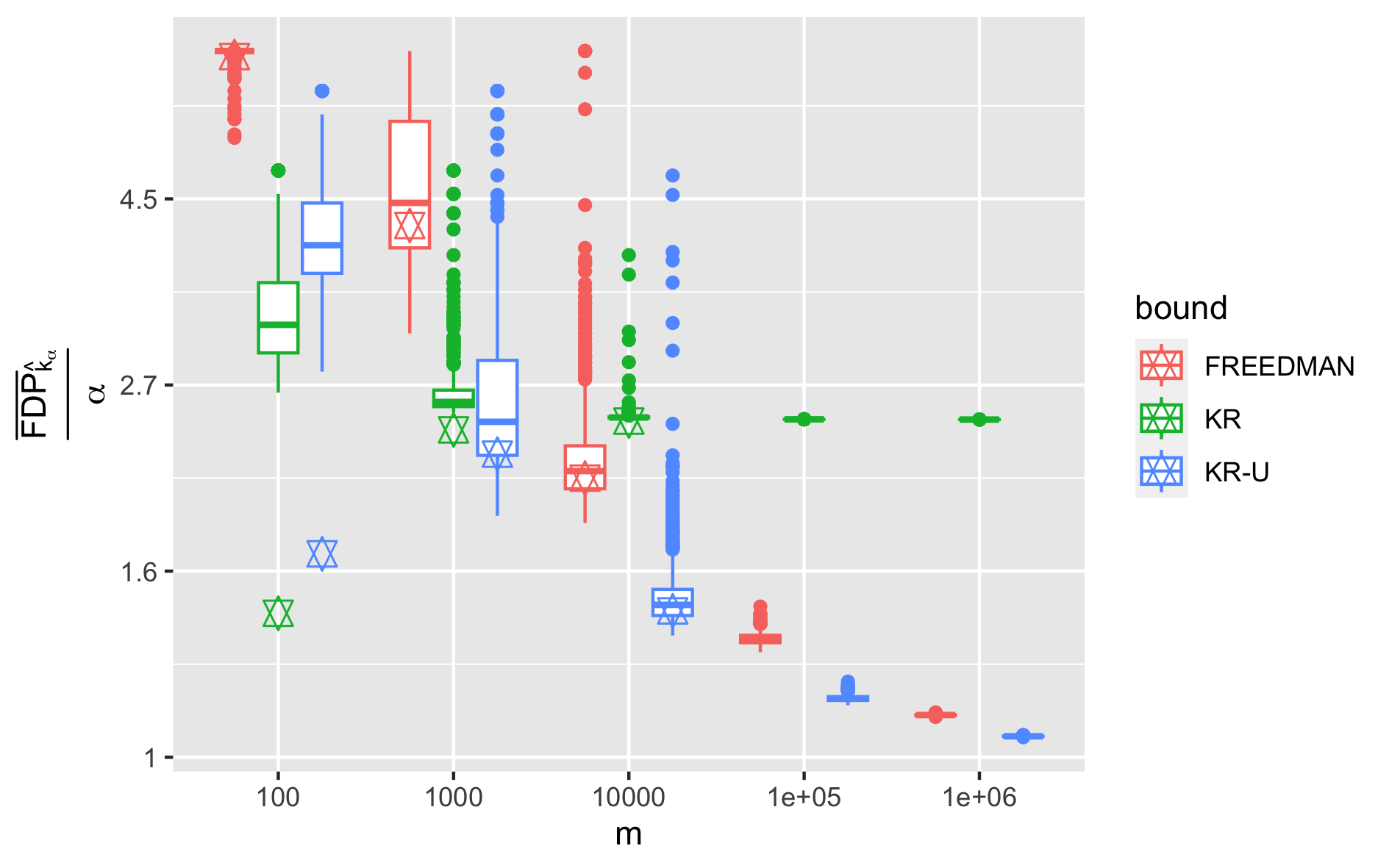}&\includegraphics[scale=0.12]{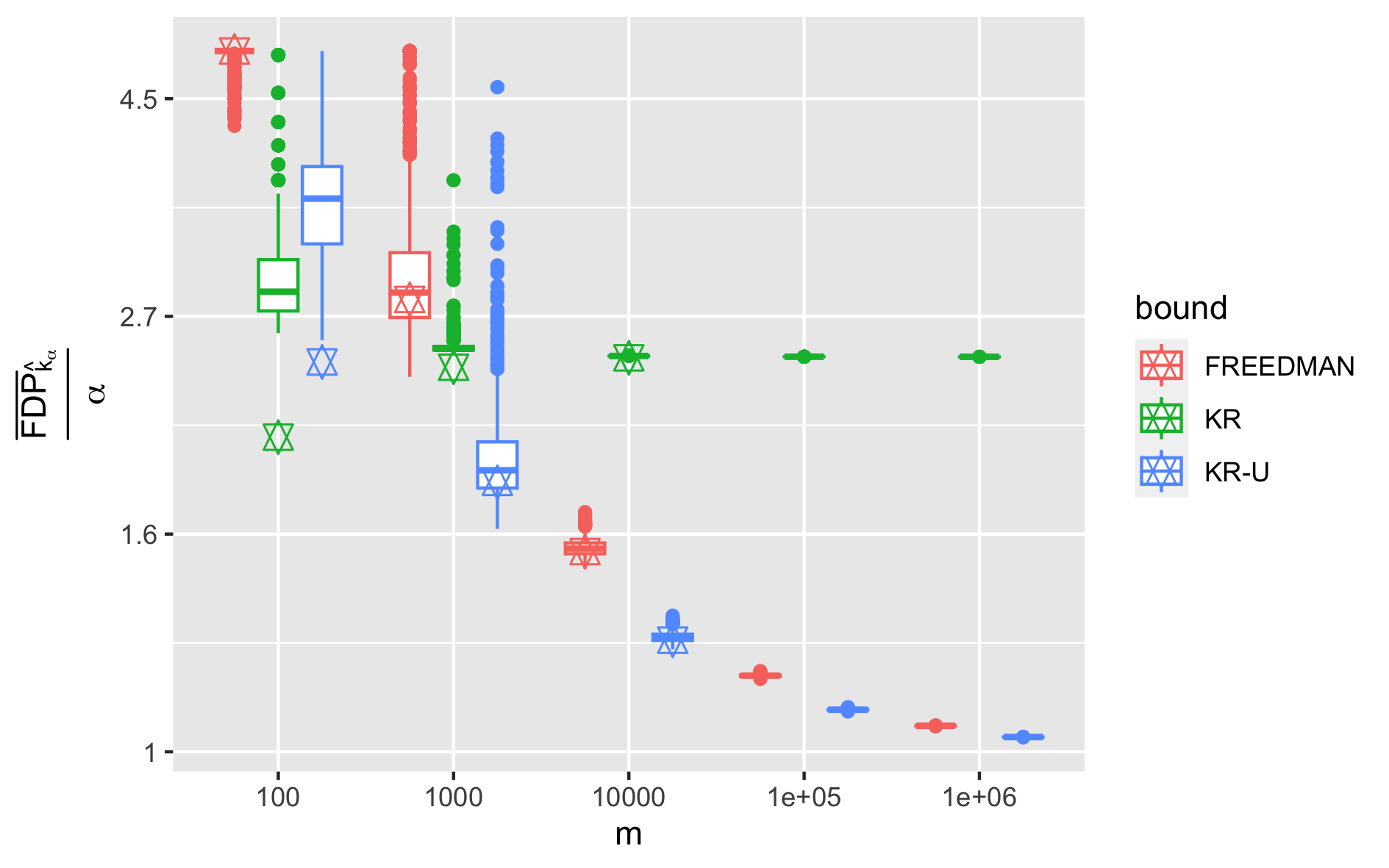}
\end{tabular}
\caption{Figure~\ref{fig:preorderknockoff} where we have  superposed in each case  the (median of the) interpolated bounds (star symbols). Pre-ordered dense ($\beta=0$) knockoff setting with BC procedure (i.e., LF procedure with $s=\lambda=0.5$). \label{fig:preorderknockoff_inter}}
\end{figure}
\end{center}

\begin{center}
\begin{figure}[h!]
\begin{tabular}{cc}
$\alpha=0.15$ & $\alpha=0.2$\\
\includegraphics[scale=0.12]{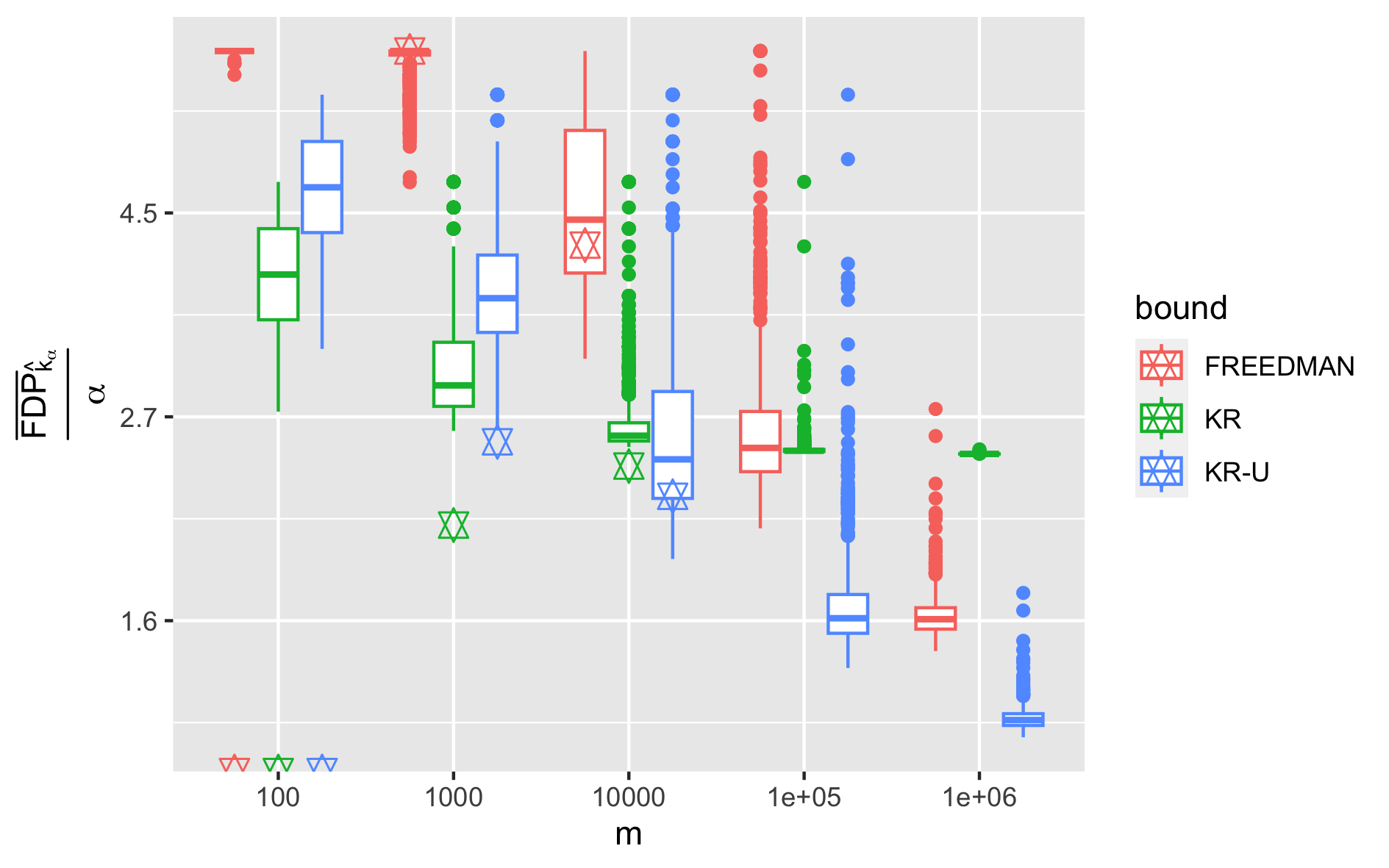}&\includegraphics[scale=0.12]{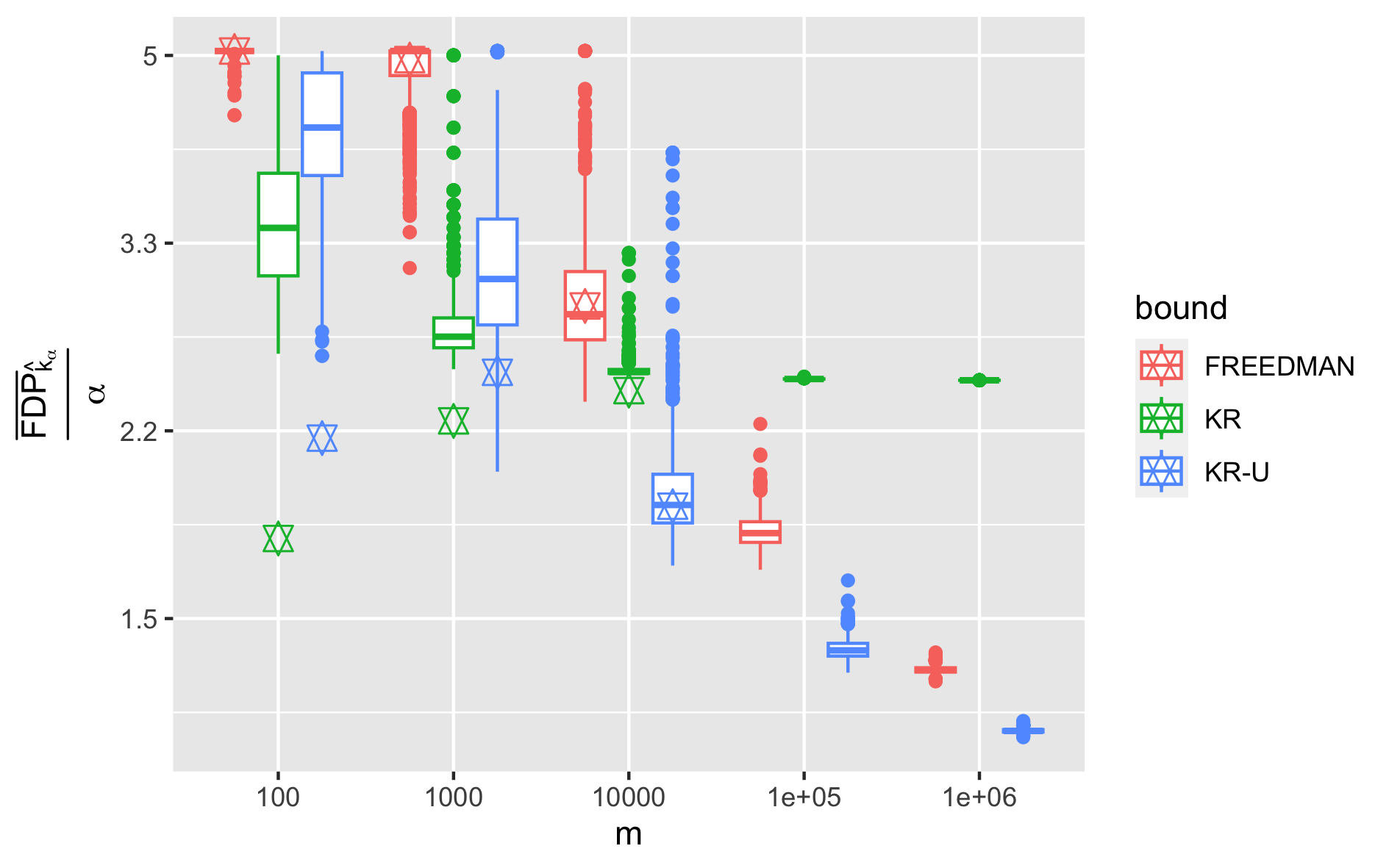}
\end{tabular}
\caption{Figure~\ref{fig:preorderknockoffsparse} where we have  superposed in each case  the (median of the) interpolated bounds (star symbols). Pre-ordered sparse ($\beta=0.25$) knockoff setting with BC procedure (i.e., LF procedure with $s=\lambda=0.5$). \label{fig:preorderknockoffsparse_inter}}
\end{figure}
\end{center}

  \subsection{{Closed testing bounds}}\label{sec:closedexp}

{
Let us consider the top-$k$ setting with $m$ null hypotheses, and consider any nonnegative sequence $\ell_{i,k}\in [0,1]$, $1\leq i\leq k$, $1\leq k\leq m$. Let $\ell_{i,k}=\ell_{i,i}$ for $i>k\geq 1$ and $\l_{i,0}=1$.
Assume that $\ell_{i,k}\geq \ell_{i,k'}$ for $1\leq k\leq k'\leq m$ for all $i\geq 1$. It includes the following cases:
\begin{itemize}
\item Simes: $\ell_{i,k}=\delta i/k$;
\item KR: 
$
\ell_{i,k}= \frac{\log(1+\log(1/\delta))}{\log(1/\delta)} i/k -1/k,
$
(for $\delta\leq 0.31$);
\item DKW: 
$
\ell_{i,k}=i/k- \sqrt{\log(1/\delta)/2}\: k^{-1/2}.
$
\end{itemize}
\begin{theorem}[Lemma~6 in \cite{GHS2021}]
In the top-$k$ setting, consider any sequence $(\ell_{i,k})_{i,k}$ as above and assume that for all $S\subset \cH_0$, $\P(\exists i\in \{1,\dots,|S|\}\::\: p_{(i:S)}\leq \ell_{i,|S|} )\leq \delta$. Then the closed-testing FDP envelope 
\begin{align}
\overline{\FDP}_k&=\min_{1\leq k'\leq k} \Big\{  \sum_{1\leq j\leq k}\ind{p_{(j)}> \ell_{k',\hat m_0}}+k'-1\Big\}/k;\label{equFDPbarclosed}\\
\hat{m}_0&=\max\{0\leq j\leq m\::\: \mbox{for all $i\in \{1,\dots,j\}$, } p_{(m-j+i)}>\ell_{i,j}\}\label{m0chapclosed}
\end{align}
%In addition, provided that for all $S\subset \cH_0$, $\P(\exists i\in \{1,\dots,|S|\}\::\: p_{(i:S)}\leq \ell_{i,|S|} )\leq \delta$, we have that $\overline{\FDP}_k$
 is valid in the sense of \eqref{confenvelop}.
\end{theorem}
%\todo{Formulation bizarre: ne vaudrait-il pas mieux dire: sous l'hypothese (blah), la closed-testing envelope donnee
%  par l'expression suivante est valide:\ldots? Aussi je ne comprends pas l'hypothese: il ne faut pas $S \subset \cH_0$}
The form of the closed-testing FDP bound \eqref{equFDPbarclosed} turns out to coincide with the adaptive interpolated bounds of Section~\ref{sec:interptopk} (improved by adding an integer part). This is exemplified in the next result for the Simes sequence.
\begin{lemma}\label{lem:simesclosed}
Consider the Simes sequence $\ell_{i,k}=\delta i/k$. Then, on the event where all $p$-values are different from all  thresholds $\ell_{i,k}$,  the closed-testing bound \eqref{equFDPbarclosed} is equal to 
$$
k\wedge \min_{1\leq k'\leq k} \{k-k'+  k'\wedge \lfloor  \hat{m}_0 p_{(k')}/\delta\rfloor\},
$$
%in which the estimator $\hat m_0$ \eqref{m0chapclosed} is used.
\end{lemma}
%\todo{Formulation peu claire: est-ce a dire quel que soit $\hat m_0$, l'expression ci dessus e~\eqref{equFDPbarclosed} co\"incident?}
Lemma~\ref{lem:simesclosed} shows that, for the Simes threshold, the closed testing bound improves the interpolated one only in the way $m_0$ is estimated. 
The closed-testing $m_0$ estimator \eqref{m0chapclosed} is by essence more accurate than those that we proposed in  Section~\ref{sec:adaptive}, but is also more computationally demanding. 
In addition, in our experiments, the improvement is modest in general, as shown in Figure~\ref{fig:close-testing-adaptivetopk}.  
We see the closed-testing versions of our bounds as advisable when $m$ is small, because the improvement seems to be the most significant in that case while the complexity is still low. 
%improvement of closed testing wrt the adaptive interpolated bounds is modest in general
In addition, this figure also suggests that the closed testing versions of Simes and KR bounds are $m$-inconsistent, which corroborates the theoretical findings of Corollary~\ref{prop:inconsist} in the Simes case.
%, in which the more accurate (but more computationally demanding) $m_0$ estimator \eqref{m0chapclosed} has been plugged
%\wt{\FDP}^{{\tiny \mbox{Simes}}}_k &:=\min_{k'\leq k} \{k-k'+  k'\wedge (m p_{(k')}/\delta)\} /k
%For instance, for the Simes sequence, it leads to the envelope
%$$
%\overline{\FDP}_k=1\wedge (\min_{1\leq k'\leq k} \{k  - k' + \hat m_0 p_{(k')}/\delta\} -1)/k,
%$$ 
%which is close to the adaptive interpolated Simes bound that we obtained in Sections~\ref{sec:adaptive} and~\ref{sec:interptopk}, with a slightly better $m_0$ estimator. 
%\et{where the $-1$ comes from here? I can not reach it with the adaptive interpolated!}
}

{
\begin{proof}
Define
$\mathcal{U}=\{ u_{(k')}, 1\leq k'\leq k\}$, with $u_{(k')}=\lfloor \hat m_0 p_{(k')}/\delta\rfloor$. 
We have
\begin{align*}
k\wedge \min_{1\leq k'\leq k} \{k-k'+  k'\wedge \lfloor \hat m_0 p_{(k')}/\delta\rfloor\} &=k\wedge \min_{1\leq k'\leq k} \{k-k'+   \lfloor \hat m_0 p_{(k')}/\delta\rfloor\} \\
%&\min_{0\leq k'\leq k} \{k \ind{k'\leq  \lfloor m p_{(k')}/\delta\rfloor } + (k-k'+\lfloor m p_{(k')}/\delta\rfloor )\ind{k'>  \lfloor m p_{(k')}/\delta\rfloor} \}  \\
&=k\wedge \min_{u\in \mathcal{U}} \Big\{ k-\sum_{j=1}^k \ind{u_{(j)} \leq u } + u \Big\}\\
&= k\wedge\min_{u\in \mathcal{U}} \Big\{ \sum_{j=1}^k \ind{u_{(j)} > u } + u \Big\}\\
&= k\wedge\min_{u\in \mathcal{U}} \Big\{ \sum_{j=1}^k \ind{u_{(j)} \geq  u +1} + u \Big\}\\
&= k\wedge\min_{u\in \mathcal{U}} \Big\{ \sum_{j=1}^k \ind{  \hat m_0 p_{(j)}/\delta  \geq  u +1} + u \Big\}\\
&=k\wedge\min_{v\in \mathcal{U}+1} \Big\{ \sum_{j=1}^k \ind{  p_{(j)} \geq  v\delta/\hat m_0 } + v-1\Big\}.
\end{align*}
Fix now $w\in \{1,\dots,k\}$, and let us prove
\begin{equation}\label{equtoproved}
 \sum_{j=1}^k \ind{  p_{(j)} \geq  w\delta/\hat m_0 } + w-1\geq k\wedge\min_{v\in \mathcal{U}+1} \Big\{ \sum_{j=1}^k \ind{  p_{(j)} \geq  v\delta/\hat m_0 } + v-1\Big\}.
\end{equation}
First observe that for all $j \in \{1,\dots,k\}$,
\begin{equation}\label{equinterm}
p_{(j)} \geq  w\delta/\hat m_0 \Leftrightarrow \hat m_0 p_{(j)}/\delta \geq  w \Leftrightarrow \lfloor \hat m_0 p_{(j)}/\delta \rfloor  \geq  w \Leftrightarrow u_{(j)} +1 >  w.
%  \mbox{$p_{(j)} \geq  w\delta/\hat m_0$ iff  $ \hat m_0 p_{(j)}/\delta \geq  w$ iff $\lfloor \hat m_0 p_{(j)}/\delta \rfloor  \geq  w$ iff $u_{(j)} +1 >  w$.}
\end{equation}
Hence if for all $v\in \mathcal{U}+1 $ we have $v> w$, \eqref{equtoproved} is satisfied. Otherwise, there is one $v\in \mathcal{U}+1$ such that  $v\leq w$ and we can consider $v_w=\max\{v\in \mathcal{U}+1  \::\: v\leq w\}$ the maximum of the elements of $\mathcal{U}+1$ that are below $w$. From \eqref{equinterm}, we have for all
$j \in \{1,\dots,k\}$, $p_{(j)} \geq  w\delta/\hat m_0 \Leftrightarrow u_{(j)} +1 >  v_w \Leftrightarrow p_{(j)} \geq  v_w\delta/\hat m_0$,
%$j \in \{1,\dots,k\}$, $p_{(j)} \geq  w\delta/\hat m_0$ iff  $u_{(j)} +1 >  v_w$ iff $p_{(j)} \geq  v_w\delta/\hat m_0$,
which means $\sum_{j=1}^k \ind{  p_{(j)} \geq  w\delta/\hat m_0 } =\sum_{j=1}^k \ind{  p_{(j)} \geq  v_w\delta/\hat m_0 }$ and thus since $w\geq v_w$, the inequality \eqref{equtoproved} is also satisfied. 
 This establishes in any case 
%Then we have 
%$ \sum_{j=1}^k \ind{  p_{(j)} \geq  w\delta/m } + w-1\geq \sum_{j=1}^k \ind{  p_{(j)} \geq  w\delta/m } + v_w-1$. In addition,  we have $p_{(j)} \geq  w\delta/m$ iff  $ m p_{(j)}/\delta \geq  w$ iff $\lfloor m p_{(j)}/\delta \rfloor  \geq  w$ iff $u_{(j)} +1 >  w$ iff $u_{(j)} +1 >  v_w$. %Hence, if $w<k$, $p_{(j)} \geq  w\delta/m$ iff $(u_{(j)} +1)\wedge k >  w$ ($w< k$). Since  $(u_{(j)} +1)\wedge k\in \mathcal{V}$ we thus have  $p_{(j)} \geq  w\delta/m$ iff $(u_{(j)} +1)\wedge k >  v_w$ 
%And, by applying the same reasoning, the latter is true iff $p_{(j)} \geq  v_w\delta/m$. Hence, we have proved $p_{(j)} \geq  w\delta/m$ iff $p_{(j)} \geq  v_w\delta/m$ and thus for all $w\in \{1,\dots,k\}$, we have 
%$ \sum_{j=1}^k \ind{  p_{(j)} \geq  w\delta/m } + w-1\geq \sum_{j=1}^k \ind{  p_{(j)} \geq  v_w\delta/m } + v_w-1$.
%This establishes \eqref{equtoproved}.
$$
k\wedge\min_{v\in \mathcal{U}+1} \Big\{ \sum_{j=1}^k \ind{  p_{(j)} \geq  v\delta/m } + v-1\Big\} 
=k\wedge\min_{1\leq w\leq k} \Big\{ \sum_{j=1}^k \ind{  p_{(j)} \geq  w\delta/m } + w-1\Big\} .
$$
This gives the result.
\end{proof}
\begin{center}
	\begin{figure}[h!]
		\begin{tabular}{cc}
%			$\alpha=0.05$ & $\alpha=0.1$\\
%			\includegraphics[scale=0.12]{topk/with_closetesting/dense_adapt_alpha_0_05.png}&\includegraphics[scale=0.12]{topk/with_closetesting/dense_adapt_alpha_0_1.png}\\
%			$\alpha=0.15$ & $\alpha=0.2$\\
%			\includegraphics[scale=0.12]{topk/with_closetesting/dense_adapt_alpha_0_15.png}&\includegraphics[scale=0.12]{topk/with_closetesting/dense_adapt_alpha_0_2.png}
			$\alpha=0.1$ & $\alpha=0.2$\\
		\includegraphics[scale=0.12]{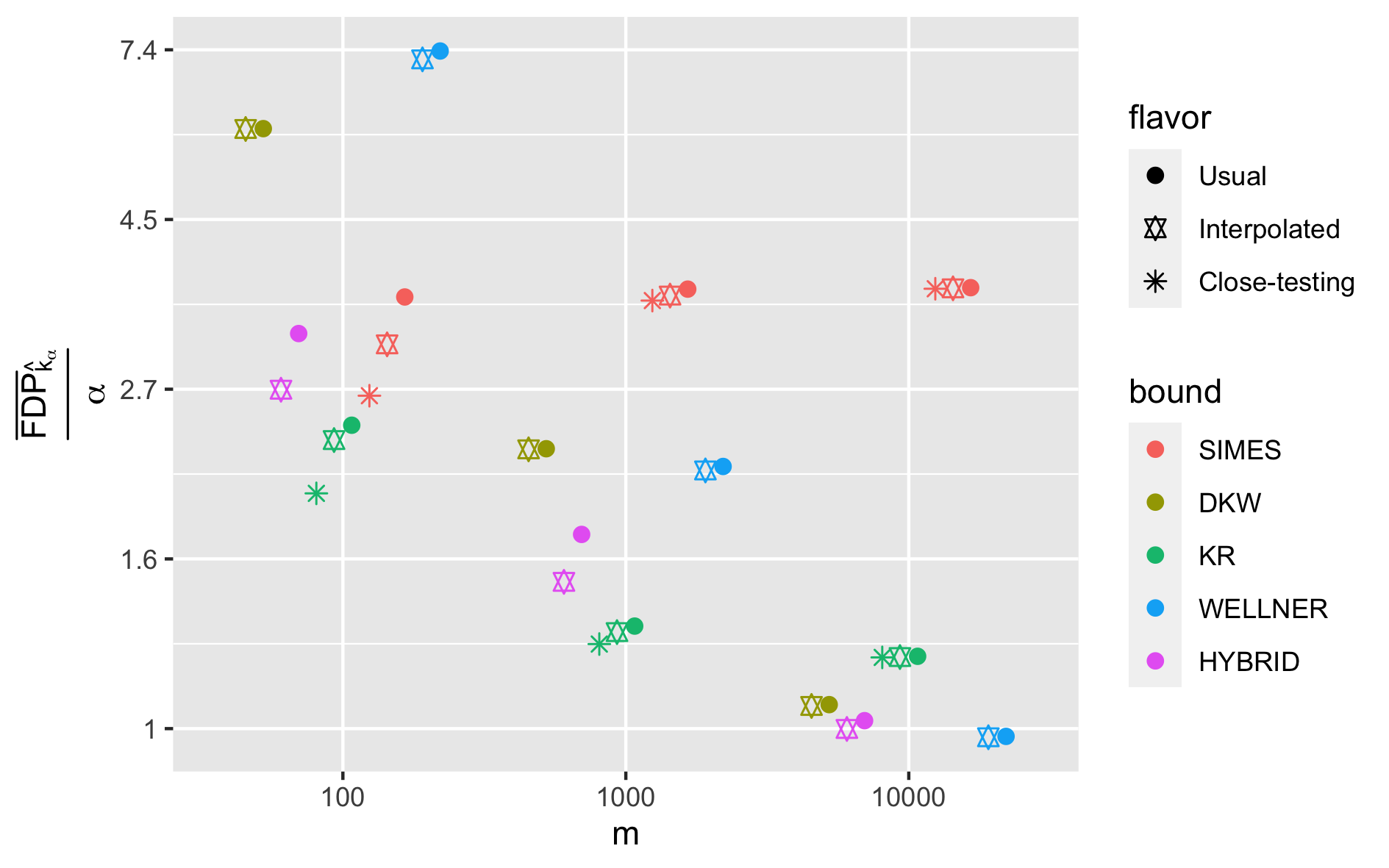}
		&\includegraphics[scale=0.12]{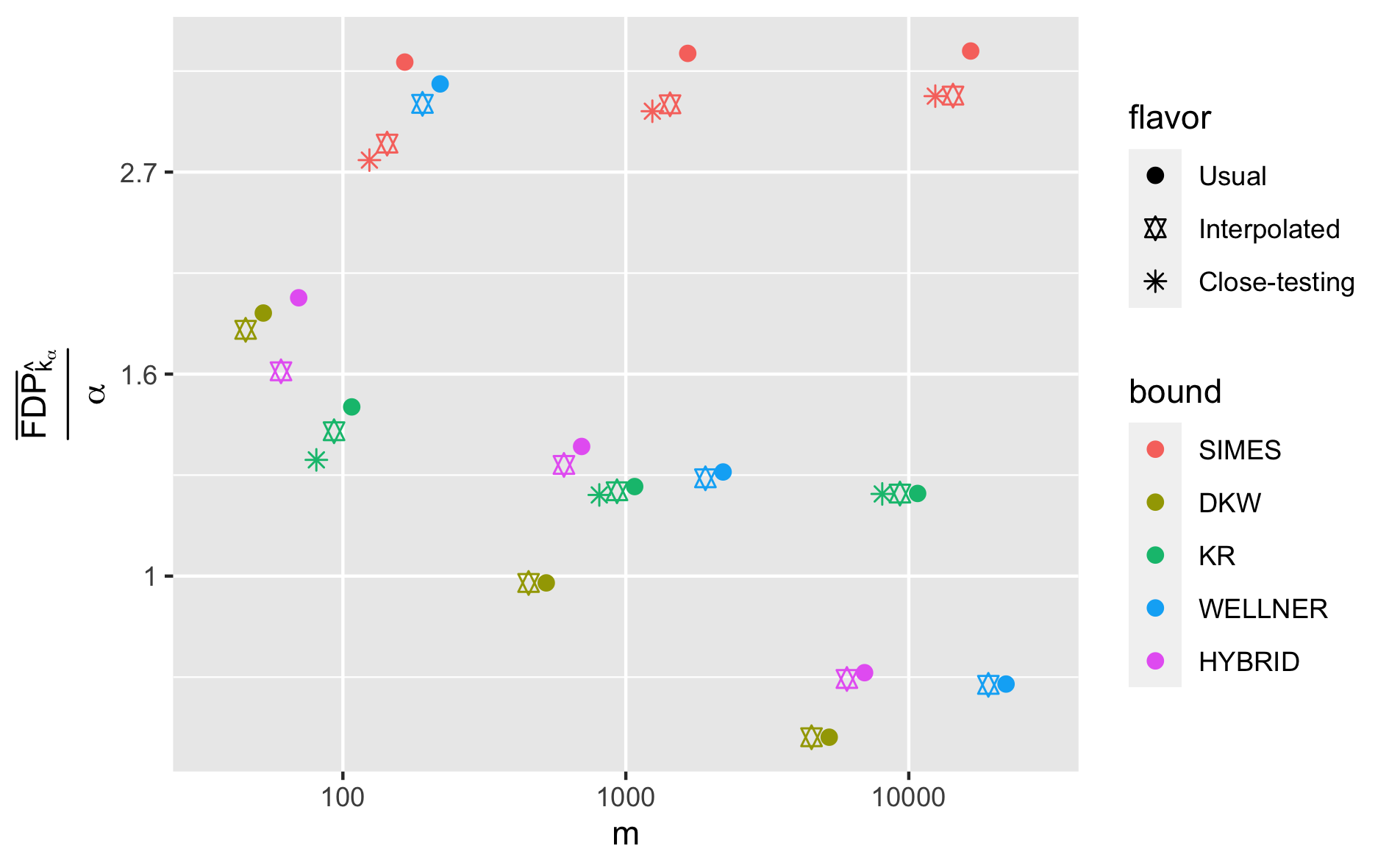}
		\end{tabular}
		\caption{
%		Same presentation as in Figure~\ref{fig:topkdense}. 
		{Median of the adaptive bounds of Section~\ref{sec:adaptive} (plain circles), median of interpolated bounds of Section~\ref{sec:interptopk} (hollow star), and median of closed-testing bounds given by \eqref{equFDPbarclosed} (asterix) in function of $m\in \{100,1000,10000\}$. The closed-testing is only computed for Simes and KR bounds. The simulation setting is the same as the one used for the right panel of Figure~\ref{fig:topkadapt} ($\pi_{0}= 0.5$).\label{fig:close-testing-adaptivetopk}}}
	\end{figure}
\end{center}
}

\end{document}